%% file: thesishugo.tex
\documentclass[12pt,oneside]{book}
\usepackage{amssymb,amsmath,amsfonts,enumerate,bbm,mcode,verbatim}
\usepackage{color}
\usepackage{amsthm}
\usepackage[OT2,T1]{fontenc}
\usepackage{geometry}
\usepackage[active]{srcltx}
\usepackage[nathan]{nathanchap}
\usepackage[all]{xy}
\usepackage{graphicx}
\usepackage{float}
\usepackage{makeidx}
\usepackage{palatino}

\DeclareSymbolFont{cyrletters}{OT2}{wncyr}{m}{n}
\DeclareMathSymbol{\Sha}{\mathalpha}{cyrletters}{"58} 
\newtheorem{theorem}{Theorem}[chapter]
\newtheorem{proposition}[theorem]{Proposition}
\newtheorem{lemma}[theorem]{Lemma}
\newtheorem{definition}[theorem]{Definition}
\newcommand{\longhookrightarrow}{{\lhook\joinrel\relbar\joinrel\rightarrow}}
\newcommand{\tmop}[1]{\ensuremath{\operatorname{#1}}}
\newcommand{\assign}{:=}
\newtheorem{corollary}[theorem]{Corollary}

\newcommand{\longrightarrowlim}{\mathop{\longrightarrow}\limits}
\newcommand{\ls}{L_{\sigma}}
\newcommand{\os}{\mathcal{O}}
\newcommand{\oy}{\mathcal{O}_Y}
\newcommand{\s}{\sigma}
\newcommand{\p}[1]{\mathbb{P}^{#1}}
\newcommand{\mb}[1]{\mathbb{#1}}
\newcommand{\pic}{\tmop{Pic}}
\newcommand{\lrw}{\longrightarrow}
\newcommand{\Ll}{\Lambda}
\newcommand{\la}[1]{\lambda_{#1}}

\newcommand{\ti}[1]{\tilde{#1}}

\newcommand{\hh}{H^1(G,\tmop{Pic}Y)}
\newcommand{\dra}{\dashrightarrow}
\newcommand{\od}{\mathcal{O}_D}
\newcommand{\fn}{\mb{F}_n}

\newcommand{\me}{\sim_M}
\newcommand{\br}{\tmop{Br}}

\DeclareFontFamily{OT1}{pzc}{}
\DeclareFontShape{OT1}{pzc}{m}{it}{<-> s * [1.10] pzcmi7t}{}
\DeclareMathAlphabet{\mathpzc}{OT1}{pzc}{m}{it}

\mathchardef\mhyphen="2D

\newtheorem{varexample}[theorem]{Example}
\newenvironment{example}{\begin{varexample}\em}{\em\end{varexample}}

\newtheorem{vardefex}[theorem]{Example and Definition}
\newenvironment{defex}{\begin{vardefex}\em}{\em\end{vardefex}}

\newtheorem{varprop}[theorem]{Properties}

\newtheorem{varnotation}[theorem]{Notation}
\newenvironment{notation}{\begin{varnotation}\em}{\em\end{varnotation}}

\newtheorem{construction}[theorem]{Construction}

\theoremstyle{definition}
\newtheorem{remark}[theorem]{Remark}

\makeindex

\begin{document}

\begin{center}
\begin{center}
\includegraphics[width=3cm,viewport= 25 106 354 654]{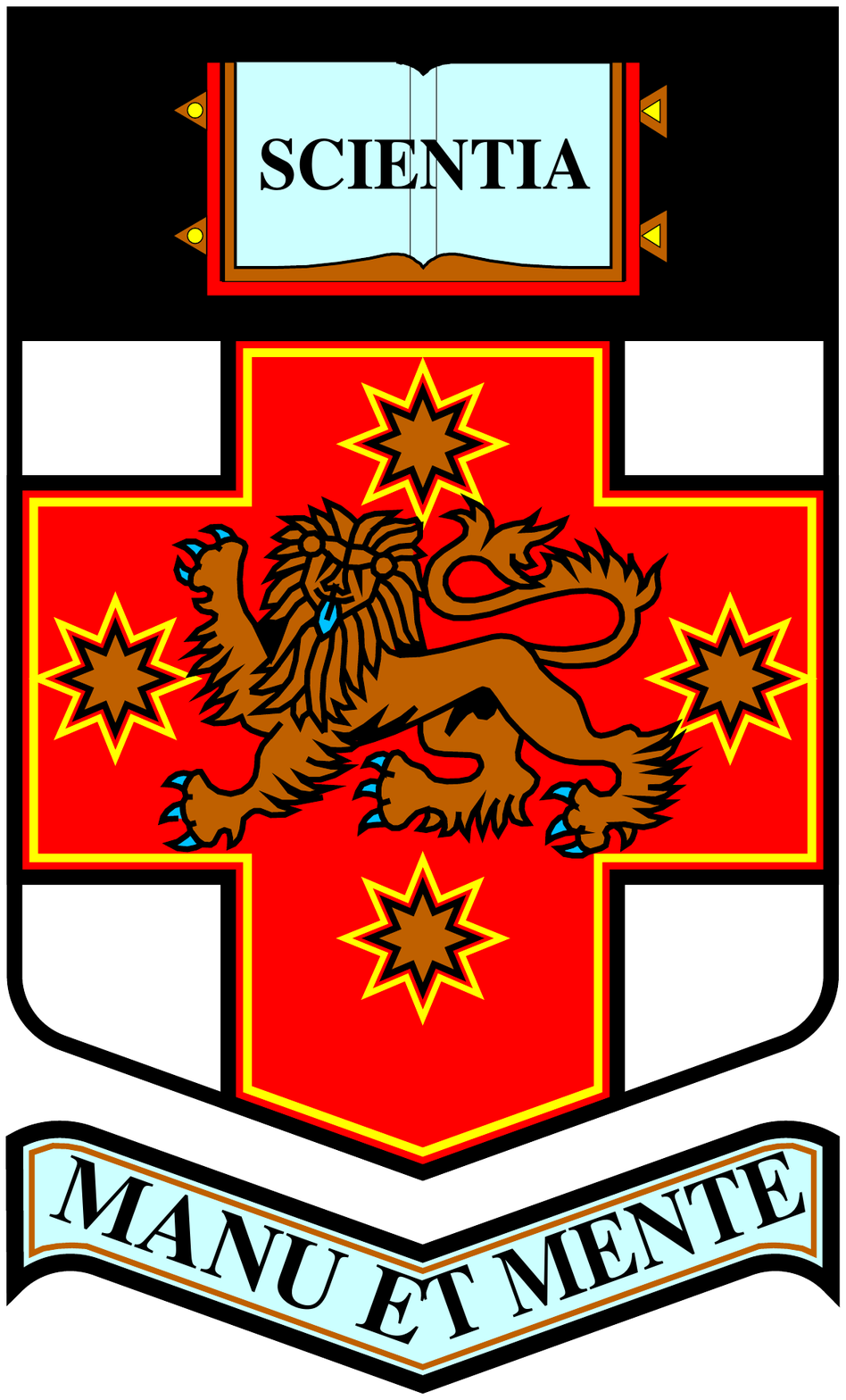} 
\end{center}
 \Huge{The Explicit Construction of Orders on Surfaces}
\thispagestyle{empty}
\vspace{-1cm}

\vspace{2cm}
\vspace{0.5cm} 
\normalsize{A thesis presented to} \\ \vspace{0.5cm} \large{The University of 
New South Wales} \\ \vspace{0.5cm} \normalsize{in fulfillment of the thesis 
requirement} \\ \normalsize{for the degree  of} \\ \vspace{0.5cm} 
\large{Doctor of Philosophy} \\ \vspace{0.5cm} \normalsize{by}

\vspace{2cm}

{\scshape \LARGE{Hugo Bowne-Anderson}}

\vspace{2cm}

\end{center}
\newpage
\renewcommand{\thepage}{\roman{page}}
\addcontentsline{toc}{chapter}{Abstract}
\setcounter{page}{1}

%

\setcounter{page}{-1}
\chapter*{Abstract}

\input{abstract}

\author{Hugo Bowne-Anderson}
\title{}
\maketitle

\chapter*{Acknowledgements}
\addcontentsline{toc}{chapter}{Acknowledgements}
\input{acknowledgements}

\newpage
\markboth{Contents}{Contents}

\addcontentsline{toc}{chapter}{Contents}
\tableofcontents

%


%


\listoffigures
\addcontentsline{toc}{chapter}{List of Figures}

\newpage

\renewcommand{\thepage}{\arabic{page}}
\setcounter{page}{1}

\include{c00}

\include{c1}
\include{c2}
\include{c3}
\include{app1}

\appendix

\include{c5}
\include{c6}

\appendix

\addcontentsline{toc}{chapter}{Index}
\printindex

\markboth{Bibliography}{Bibliography}
\addcontentsline{toc}{chapter}{Bibliography}

\bibliography{thesishugo}

\end{document}

%% file: abstract.tex
The study of orders over surfaces is an integral aspect of noncommutative algebraic geometry. Although there is a substantial amount known about orders, relatively few concrete examples have been constructed explicitly. Of those already constructed, most are del Pezzo orders, noncommutative analogues of del Pezzo surfaces, the simplest case. We reintroduce a noncommutative analogue of the well-known commutative cyclic covering trick and implement it to explicitly construct a vast collection of numerically Calabi-Yau orders, noncommutative analogues of surfaces of Kodaira dimension 0. This trick allows us to read off immediately such interesting geometric properties of the order as ramification data and a maximal commutative quotient scheme. We construct maximal orders, noncommutative analogues of normal schemes, on rational surfaces and ruled surfaces. We also use Ogg-Shafarevich theory to construct Azumaya algebras and, more generally, maximal orders on elliptically fibred surfaces.

%% file: acknowledgements.tex

There are many people who contributed in varying forms to the existence of this thesis. First and foremost is Daniel Chan, supervisor and friend, whose patience, encouragement and love for mathematics made this project a joy to work on. Our weekly meetings for the best part of four and a half years provided a continual grounding of and inspiration for whatever was happening in my research at any given point in time.

I am indebted to everyone with whom I had the pleasure of discussing mathematics in general and algebraic geometry in particular. Most notably, Rajesh Kulkarni and Ulf Persson, the latter responsible for teaching me boundless information concerning explicit geometry on algebraic surfaces via conversations, emails and papers. A special and very warm thank you to my mathematical brothers, Kenneth Chan and Boris Lerner, friends and colleagues, proof-readers and critics, without whom this sometimes seemingly eternal process may have actually taken forever.

My friends and my family, thank you. Thank you for your understanding. Your various impacts  on this document are less traceable, more elusive and yet without your continual support and love, it just simply wouldn't exist.

%% file: c00.tex
\chapter{Introduction}

``Abandoning geometric intuition--the only means that so far has allowed us to find the way in this tangled territory--would mean extinguishing the feeble flame that can lead us into the dark forest.''\\ \emph{La geometria algebrica e la scuola italiana}, Guido Castelnuovo, 1928
\\
\\
The discipline of modern algebraic geometry is concerned with the deep and subtle interplay between algebraic objects, such as commutative rings and modules, and geometric objects, such as curves, surfaces and more generally schemes. Until relatively recently, the deep correspondences in existence have been confined to the commutative world. During the past two decades, however, a significant number of algebro-geometric techniques have been utilised in the study of noncommutative algebras and this is precisely the realm of noncommutative algebraic geometry. 

Our particular interests lie in studying objects known as orders on commutative schemes: an order on a scheme $X$ is a torsion-free sheaf of $\os_X$-algebras which is generically a central simple $K(X)$-algebra. Not only are orders of interest to noncommutative algebraic geometers, but they have been used with success in commutative algebraic geometry, for example in Artin and Mumford's initial construction of a unirational variety which is not rational \cite{artmum}. The relationship between the study of orders and that of commutative algebraic geometry is strong; in particular, the fact that orders are finite over their centre means that a great deal of classical commutative algebraic geometry can be utilised to study them, a prime example being classifying orders using ramification data on the central scheme. 

Although a great deal is known about orders, relatively few explicit examples have been realised and it is to this goal that we dedicate our project: the explicit construction of orders on schemes. In particular, we are interested in constructing maximal orders, which are noncommutative analogues of normal schemes. More precisely, our interests are in constructing maximal orders on surfaces since maximal orders on curves are trivial Azumaya algebras. It is our view that in order to study the geometry of orders, we require a sophisticated body of examples with which to work.
\\
\\
Throughout algebraic geometry's rich history, there has been a continual insistence placed on the importance of explicit geometry, a prime example being the work of the Italian School, working around the turn of the 20th century, which culminated in what has become known as the Enriques classification of algebraic surfaces \cite{enclass}. In the movement toward classification, it was Corrado Segre who initially suggested a birational classification of surfaces as an approach and he directed everyone to look at surfaces embedded in projective space. In noncommutative algebraic geometry, however, there seems to be no good analogue of constructing surfaces via embeddings in projective space. There is, however, another natural way of constructing surfaces, the cyclic covering trick, which does generalise to the noncommutative setting. The cyclic covering trick allows us to construct new varieties as covers of pre-existing ones and the noncommutative generalisation will allow us to construct a vast collection of noncommutative surfaces.

\begin{construction}{\bf (The cyclic covering trick)} \cite[Chap. 1, Sec. 17]{bpv}\label{cyc1}\index{cyclic covering trick}
Let $Z$ be a scheme, $D\subset Z$ an effective divisor and $L$ a line bundle on $Z$ such that
\begin{align*}
 L^{\otimes n} & \simeq \os_Z(-D).
\end{align*}
Then there exists an $n\colon 1$ cyclic cover $\pi\colon Y\to Z$ totally ramified on $D$ and unramified elsewhere.
\end{construction}
%

The cyclic covering trick works in the following fashion: given an ${L\in\tmop{Pic}Z}$ as required in Construction \ref{cyc1}, we set ${\mathcal{A}\assign\os_Z\oplus L\oplus \ldots \oplus L^{n-1}}$, a sheaf of $\os_Z$-algebras, and $Y\assign\mathbf{Spec}_Z\mathcal{A}$, the relative Spec of $\mathcal{A}$, as described in \cite[Chapter II, Exercise 5.17]{ag}. Then we have a corresponding finite scheme morphism $\pi\colon Y\to Z$ and $\pi_*\oy=\os_Z\oplus L\oplus  \ldots \oplus L^{n-1}$. The noncommutative cyclic covering trick in essence provides a recipe for constructing similar such $\mathcal{A}$, with the important difference that $\mathcal{A}$ will now be a noncommutative sheaf of algebras. This means we can no longer form $\mathbf{Spec}_Z\mathcal{A}$ in the same fashion as in the commutative case, but we can still consider the sheaf $\mathcal{A}$, which in nice cases is an order, as having an associated noncommutative scheme. The data necessary for the noncommutative cyclic cover are
\begin{enumerate}[1)]
 \item a commutative cyclic cover $\pi\colon Y\to Z$ and
 \item an invertible bimodule $\ls$ on $Y$ along with an effective divisor $D\subset Y$ such that $\ls^{\otimes n}\simeq\oy(-D)$. 
\end{enumerate}
The corresponding cyclic algebra is then $\mathcal{A}=\oy\oplus\ls\ldots\oplus\ls^{n-1}$.
Part of the beauty of the noncommutative cyclic covering trick is the ease with which one may read off interesting geometric properties of the resultant cyclic algebra $\mathcal{A}$ post-construction: the ramification is easily described, as is the canonical divisor along with a maximal commutative quotient scheme.

Chan introduced the noncommutative cyclic covering trick in \cite{cyc} and up until this point, a great deal was already known about orders: for example, Chan and Ingalls had classified orders over surfaces in the form of a minimal model program \cite{mmpord}, an order-theoretic analogue of the Mori program, which is the modern approach to birational geometry. 
Although there was already quite an evolved theory of orders over surfaces, there was a serious absence of concrete examples. There were, however, existence results in the form of the Artin-Mumford sequence (this sequence first appeared in \cite[Section 3, Theorem 1]{artmum} and \cite{tannen} provides a well-needed exposition of its various meanings). This sequences classifies the possible ramification data of all maximal orders and tells us whether, given particular ramification data, such an order exists.

Given existence results and a solid theory, the glaring missing link is a substantial collection of explicit examples and it is this empty region which both \cite{cyc} and this thesis are attempting to populate. Chan's initial constructions in \cite{cyc} already opened up new avenues: he was able to construct many examples of del Pezzo orders, some of which he then demonstrated were ruled surfaces in the sense of Van den Bergh \cite{ncp1}. These constructions in turn allowed the authors of \cite{ncy} to study the moduli of line bundles on a particular class of del Pezzo orders, a project entirely unfeasible until Chan introduced the noncommutative cyclic covering trick.

In a sense, however, Chan constructed the easy examples.  We make this specific: as we have seen, the first step in constructing a noncommutative cyclic cover $\mathcal{A}$ is to find a commutative cyclic cover $\pi\colon Y\to Z$. The Riemann-Hurwitz formula along with the formula for $K_A$, the canonical divisor of $A$, inform us that $Y$ is del Pezzo if $\mathcal{A}$ is del Pezzo and $Y$ is minimal Kodaira dimension $0$ if $\mathcal{A}$ is numerically Calabi-Yau (we say an order $A$ is del Pezzo if $-K_A$ is ample; we say $A$ is numerically Calabi-Yau if $K_A$ is numerically trivial). It is for this reason that Chan is so successful in constructing del Pezzo orders on surfaces: the fact that $Y$ is del Pezzo in his case implies that $\tmop{Br}(Y)=0$. Moreover, the noncommutative cyclic covering can be used to construct representatives of all Brauer classes ramified on an irreducible divisor $D\subset Z$ which pull back trivially to $Y$, which is all classes in this case.

The project contained herein of constructing numerically Calabi-Yau orders is far more difficult, since $\tmop{Kod}(Y)=0$ implies that $\tmop{Br}(Y)\neq 0$. At the beginning of \cite[Section 9]{cyc}, Chan states this precise difficulty and ``contents [himself] with mentioning some interesting examples''.
%
%
He does so in 3 simple cases and the project contained in this thesis is a continuation of this work, constructing a vast array of numerically Calabi-Yau orders, among others. This will not only pave the way for further study of the explicit geometry of orders over surfaces, but also demonstrates the sophisticated developments of Chan's techniques necessary to take these explicit constructions of orders to this next step, to construct more complex and more interesting orders over surfaces. Such developments include implementations of the ``surjectivity of the period map'' and the Strong Torelli theorem for K3 surfaces, along with precise correspondences between the noncommutative cyclic covering trick and Ogg-Shafarevich theory for elliptic fibrations.

As given by the Artin-Mumford sequence, the world of orders is a vast and populous one, in which cyclic algebras are only a small part of the landscape. However, the existence of this recipe for constructing cyclic algebras singles them out. This is precisely why they are special: most orders seem entirely resistant to being constructible, yet these noncommutative cyclic covers can be constructed very naturally. It is for this reason that we strive to find them and, fortunately for us, although they are indeed special, there are still many of them. 

\section{Thesis Plan}
In Chapter \ref{background}, we provide a brief introduction to the theory of orders on schemes, their geometry and the notion of the ramification locus of an order. We then introduce Chan's noncommutative cyclic covering trick \cite{cyc}, in essence a recipe for constructing orders, and describe the geometry of orders constructed in this manner. 

Chapter \ref{ncy} is devoted to constructing numerically Calabi-Yau orders, which are the noncommutative analogues of surfaces of Kodaira dimension 0 (see \cite{ncy}). The major results of this chapter revolve around an implementation of both the ``surjectivity of the period map`` and the Strong Torelli theorem for K3 surfaces. To this end, we necessarily delve into the geometry of K3 surfaces: the K3 lattice, cohomology of line bundles, linear systems and the transcendental lattice. Using the surjectivity of the period map, we reverse engineer K3 surfaces which are double planes and proceed to construct numerically  Calabi-Yau orders on $\p2$ ramified on sextics. To demonstrate the versatility of the technique,  we also construct orders on $\p1\times\p1$ and $\mathbb{F}_2$. We also include a brief section on the construction of numerically Calabi-Yau orders on surfaces ruled over elliptic curves.

In Chapter \ref{os}, our focus is on constructing orders on surfaces of Kodaira dimension $>-\infty$. We devote our attention to elliptically fibred surfaces which play a major role in the classification of such surfaces. We invoke Ogg-Shafarevich theory to first construct Azumaya algebras on elliptically fibred surfaces, then extend our results to construct orders ramified on fibres of the fibration. 

In Appendix \ref{br}, we provide a brief introduction to Brauer groups of fields and schemes while in Appendix \ref{code}, we provide MATLAB code used in Chapter \ref{ncy}.

%% file: c1.tex
\chapter{Preliminaries}\label{background}

%
%
%
%
%
%
%
%
%
%

In the following thesis, unless explicitly stated otherwise, all schemes are noetherian, integral, normal and of finite type over an algebraically closed field $k$, ${\tmop{char}(k)=0}$. A curve (respectively surface) is a scheme of dimension 1 (respectively 2) over $k$.

For a scheme $Z$, we shall denote the generic point by $\eta$\index{$\eta$} and the function field by $K(Z)$\index{$K(Z)$}. We denote by $Z^1$\index{$Z^1$} its codimension 1 skeleton, the set consisting of all irreducible codimension one subvarieties of Z.
For any scheme morphism ${f\colon Z\to Y}$ and any point $y\in Y$, we define the fibre $Z_y\assign Z\times_Y K(y)$\index{fibre of a morphism, $Z_y$}, where $K(y)$ is the residue field at $y\in Y$.

\section{Orders}

%
\begin{definition}\label{ordef}
 An {\bf order}\index{order} $A$ on a scheme $Z$ is a coherent sheaf of $\os_Z$-algebras such that 
\begin{enumerate}
\item[(i)] $A$ is torsion free and
\item[(ii)] $K(A)\assign A\otimes_{\os_Z} K(Z)$ is a central simple $K(Z)-$algebra.
\end{enumerate}
%
\end{definition}
\begin{remark}
 The set of orders on $Z$ contained in $K(A)$ is a partially ordered set with respect to inclusion. We call an order {\bf maximal}\index{order!maximal} if it is maximal in this poset. We deal primarily with maximal orders since they are noncommutative analogues of normal schemes.
\end{remark}
\begin{remark}\label{remark}
From property (ii) in Definition \ref{ordef} and by the Artin-Wedderburn Theorem (see Theorem \ref{aw}), an order is generically isomorphic to $M_n(D),$ for some division $K(Z)-$algebra $D.$
\end{remark}

\begin{example}
Sheaves of Azumaya algebras are maximal orders. Recall that a sheaf of Azumaya algebras  is a sheaf of $\os_Z$-algebras $A$ such that the canonical homomorphism $A\otimes_{\os_Z}A^\circ\to{\mathcal End}_{\os_Z}(A)$ is an isomorphism (see Appendix \ref{br} for further details). 
\end{example}
One way of studying the geometry of orders is by looking at ramification data, which we now define for a normal order over a surface. The definition of normality for an order can be found in \cite[Def. 2.3]{mmpord}. It suffices to say that all the orders we shall be interested in are normal. In fact, in the case of orders over surfaces, maximality implies normal (\cite[Section 2]{mmpord}). For a complete treatment of the ramification of orders, see \cite[Section 2.2]{mmpord}.\index{order!ramification data}

Let $A$ be a normal order over a surface $Z$, $C\subset Z$ a prime divisor with corresponding codimension one point $c \in Z$. Letting $A_c\assign A\otimes_{\os_Z}\os_{Z,c},$ then
\begin{eqnarray*}
A_c/J(A_c) & \simeq & \prod_{i=1}^nB
\end{eqnarray*}
where $B$ is a central simple $L-$algebra and $L$ is a cyclic field extension of $K(C)$. Thus we can associate to every prime divisor $C\subset Z$ a product of cyclic field extensions
\begin{eqnarray*}
\ti{K}(C)\assign Z(A_c/J(A_c)) & \simeq & L^n.
\end{eqnarray*}
We define
\begin{eqnarray*}
e(C) & \assign & \tmop{dim}_{K(C)}\ti{K}(C)
\end{eqnarray*}
and $e(C)$ is called the {\bf ramification index}\index{order! ramification index} of $A$ at $C$. If $e(C)>1$, we say that $A$ ramifies at $C$ and we let the {\bf discrimant divisor} $D$ be the union of all divisors $C$ where $A$ ramifies. We define the ramification data of $A$ as follows: for each component $D_i$ of $D$, we set $\ti{D}_i$ to be the curve such that $\os_{\ti{D}_i}$ is the integral closure of $\od$ in $\ti{K}(D_i)$ and set $\ti{D}$ to be the disjoint union of the curves $\ti{D}_i$.
We then define the {\bf ramification data} $R(A)$ of $A$ to be $\{ \ti{D}\twoheadrightarrow D\to Z \}$. Note that if $A$ is not ramified on a divisor $C$, then $A_c$ is Azumaya over $\os_{Z,c}$. 
%
%
The following is an instructive elementary example of ramification.
\begin{example}\label{egram}
Let $\zeta$ be a primitive $n$th root of unity. Then 
\begin{eqnarray*}
A & := \frac{k[x,y]\langle u,v\rangle }{(u^n-x,v^n-y,uv-\zeta vu)}
\end{eqnarray*}
is an order on the affine plane $\tmop{Spec}k[x,y]$. $A$ is ramified solely on the coordinate axes with ramification $\tmop{Spec}k[x,u]/(u^n-x)$ over $\tmop{Spec}k[x]$ and $\tmop{Spec}k[y,v]/(v^n-y)$ over $\tmop{Spec}k[y]$.

\end{example}



\section{The Noncommutative Cyclic Covering Trick}

Chan's noncommutative cyclic covering trick is both a noncommutative generalisation of the cyclic covering trick (Construction \ref{cyc1}) and an algebro-geometric generalisation of the construction of cyclic algebras. The latter are classical objects of noncommutative algebra, which Lam describes as providing a rich source of examples of rings which exhibit different ``left'' and ``right'' behaviour (\cite{lamlec}).
\begin{example}(Cyclic Algebras)\label{cyc}\index{cyclic algebras}
Let $R$ be a commutative noetherian ring, ${\s\colon R\to R}$ an automorphism of order $n$. We form the ring $R[x;\s]$ whose elements are left polynomials $\sum_i r_ix^i$ and in which $xr=\s(r)x$ for all $r\in R$. Letting $\alpha\in R^G$ where $G=\langle\s|\s^n=1\rangle$, we can then form the quotient $A=R[x;\s]/(x^n-\alpha)$, which is known as a {\bf cyclic $S$-algebra}, where $S=R/G$.
\end{example}

In order to perform the noncommutative cyclic covering trick, we first need a noncommutative analogue of a line bundle: recall that in \cite{ncp1} Van den Bergh constructs a monoidal category of
quasi-coherent $\oy$-bimodules, where $Y$ is a scheme. The invertible objects have the form $\ls$ where $L\in \tmop{Pic}Y$ and $\sigma\in\tmop{Aut}Y$.\index{invertible bimodule} One may think of this intuitively as the $\oy$-module $L$ where
the right module structure is skewed through by $\sigma$ so that $_{\oy}L\simeq L$ and $L_{\oy}\simeq\sigma^*L$ (this ``intuitive'' description is taken pretty much verbatim from \cite{del_mod} purely because there they said it so well). For a rigorous definition, see \cite{cyc} or go straight to the source \cite{av_twist}.  We can tensor two invertible bimodules according to the following formula which we may take as definition:
\begin{eqnarray*}
\ls\otimes M_\tau & \simeq & (L\otimes\sigma^*M)_{\tau\sigma}.
\end{eqnarray*}
In the following, we use invertible bimodules to construct orders on surfaces.

\subsection{The Trick}
\index{noncommutative cyclic covering trick}
Let $Y$ be a scheme,
 $\sigma:Y\longrightarrow Y$ an automorphism of order $n$, $G=\langle\sigma|\sigma^n=1\rangle$ and $L\in\tmop{Pic}Y$. Let $C$ be an effective Cartier divisor and suppose there exists an isomorphism of invertible bimodules
\begin{eqnarray*}
\phi: L^{ n}_{\sigma} & \longrightarrowlim^{\sim} & \mathcal{O}_Y(-C)
\end{eqnarray*}
for some integer $n$. We also write $\phi$ for the composite morphism $L^n_{\sigma} \longrightarrowlim^{\sim} \mathcal{O}_Y(-C)\hookrightarrow\mathcal{O}_Y$ and consider it a relation on the tensor algebra
\begin{eqnarray*}
T(Y;\ls) & := & \bigoplus_{i\geq 0}\ls^i.
\end{eqnarray*}
There is a technical condition we shall need: we say the relation $\phi$ satisfies the {\bf overlap condition}\index{overlap condition} if the following diagram commutes:
\begin{displaymath}
 \xymatrix{
\ls \otimes_Y \ls^{n-1} \otimes_Y \ls \ar[r]^(0.6){\phi\otimes 1} \ar[d]^{1 \otimes \phi} & \oy \otimes_Y \ls \ar[d]^\wr\\
\ls \otimes_Y \oy \ar[r]^\sim                                                            & \ls
}
\end{displaymath}
We set $A(Y;\ls,\phi) \assign T(Y;\ls)/(\phi)$ and note that, by \cite[Prop. 3.2]{cyc}, if the relation $\phi:\ls^n\longrightarrow\os_Y$ satisfies the overlap condition then
\begin{eqnarray*}
A(Y;\ls,\phi) & = & \bigoplus_{i=0}^{n-1}\ls^i.
\end{eqnarray*}
This is the noncommutative cyclic cover and $A=A(Y;\ls,\phi)$\index{$A=A(Y;\ls,\phi)$} is known as a cyclic algebra (we commonly write $A(Y;\ls)$ when the relation $\phi$ is clear). 
%
%
%
Chan's first example \cite[Example 3.3]{cyc} tells us what these noncommutative cyclic covers look like generically:

\begin{eqnarray*}
K(A) & \simeq & \frac{K(Y)[z;\sigma]}{(z^n-\alpha)}, \tmop{where} \alpha \in K(Z).\\
\end{eqnarray*}
It follows from this that $Y$ is a maximal commutative quotient scheme of $A$. One of the reasons these cyclic covers are so interesting is that one can determine their geometric properties such as the ramification with relative ease.
The following results inform us that these cyclic covers are normal orders and describe the ramification in the case where $L_\s^n\simeq\oy$ and $Y$ is a surface. For the remainder of the chapter $\pi\colon Y\to Z$ will be cyclic \footnote{That is, such that Gal$(K(Y)/K(Z))$ is cyclic.} with Galois group $G$ and $\s$ the generator of $G$.
\begin{theorem}\cite[Theorem 3.6]{cyc}\label{normet}
Let $Y,Z$ be quasi-projective surfaces, $\pi\colon Y\to Z$ $n\colon 1$. Let $A=A(Y;\ls,\phi)$ be a cyclic algebra arising from a relation of the form $\phi\colon\ls^n\simeq\oy$. Assuming $\phi$ satisfies the overlap condition, then $A$ is a normal order and for $C\in Z^1$, the ramification index of $A$ at $C$ is precisely the ramification index of $\pi$ above $C$.
\end{theorem}
\begin{theorem}\label{ram}\cite[Thm 3.6 and Prop. 4.5]{cyc}
 Suppose that $Y,Z$ are smooth quasi-projective surfaces and that the $n\colon 1$ quotient
map $\pi\colon Y\to Z$ is totally ramified at $D\subset Y$. Consider the cyclic algebra $A=A(Y;L_\s)$ arising
from a relation of the form $L_\s^n\simeq\oy$. Then the ramification of $A$ along $\pi(D)$ is the cyclic cover of $D$ defined by the n-torsion line bundle $L_{|D}$.
\end{theorem}
\begin{remark}
To see that the line bundle $L_{|D}\in\tmop{Pic}D$ is $n-$torsion, notice that since $\pi$ is totally ramified at $D$, $D$ is fixed by $\s$ and hence $\s^*(L_{|D})\simeq L_{|D}$. Moreover, since $L_\s^n\simeq\oy$, $(L_{|D})_\s^n\simeq\os_{D}$, which in turn means that $L_{|D}\otimes\s^*(L_{|D})\otimes\ldots\otimes(\s^*)^{n-1}(L_{|D})\simeq\os_{D}$. We conclude that $(L_{|D})^n\simeq\os_{D}$.
\end{remark}
Chan only deals with cyclic covers $\pi\colon Y\to Z$ which are totally ramified and so Theorem \ref{ram} suffices for his purposes. We shall require an analogous result concerning ramification when $\pi$ is not totally ramified. 

\begin{lemma}\label{untot}
 Assume $Y,Z$ to be smooth and that $\pi\colon Y\to Z$ is ramified at $D\subset Y$, where $D\assign\pi^{-1}(D')$, the reduced inverse image of $D'\subset Z$. We also assume that $D=\coprod_{i=1}^dD'_i$ such that the ramification index at each $D'_i$ is $m$, where $m=\frac{n}{d}$. Consider the cyclic algebra $A=A(Y;\ls)$ arising from a relation $L_\s^n\simeq\oy$. Then the ramification of $A$ along $D'$ is the cyclic cover of $D'$ defined by the m-torsion line bundle $(L\otimes\s^*L\otimes\ldots\otimes\s^{*d}L)_{|D'_i}$, for any $i\in\{1,\ldots,d\}$.
\begin{proof}
To compute the ramification of $A$ along $D'$,
we need to compute $A\otimes K(D')$: letting $z$ be a local generator for $L_{|D}$,
\begin{eqnarray*}
 A \otimes K(D') & = & \frac{(\prod_{i=1}^dK(D'))[\varepsilon,z]}{(\varepsilon^m,z^{dm}-\alpha)}, \alpha\in K(D').
\end{eqnarray*}
Thus
\begin{eqnarray*}
 (A\otimes K(D'))/J(A\otimes K(D')) & = & \frac{(\prod_{i=1}^dK(D'))[z]}{(z^{dm}-\alpha)}
\end{eqnarray*}
and
\begin{eqnarray*}
 Z((A\otimes K(D'))/J(A\otimes K(D'))) & = & \frac{K(D')[z^d]}{(z^{dm}-\alpha)}.
\end{eqnarray*}
Since $z^d$ is a local generator for $(L\otimes\s^*L\otimes\ldots\otimes\s^{*d}L)_{|D}$, the ramification is exactly as specified.

\end{proof}

\end{lemma}

The following lemma will allow us to verify when the cyclic covering trick produces maximal orders.
\begin{lemma}\label{max}
A cyclic algebra $A$ on $Z$ is maximal if for all irreducible components $D_i$ of the discriminant divisor $D$, $\ti{D_i}$ is irreducible.
\begin{proof}
In \cite[Theorem 1.5]{ausgold}, we see that $A$ is maximal if it is reflexive and maximal over codimension 1 points. Cyclic algebras are reflexive by \cite[Cor. 3.4]{cyc}. Moreover, if $A$ is not ramified at $c\in Z^1$, $A\otimes\os_{Z,c}$ is Azumaya over $\os_{Z,c}$, implying maximality. Thus it suffices to check maximality over the ramification locus of $A$. From \cite[Theorem 2.3]{ausgold}, $A_c$ is maximal if it is hereditary and $A_c/J(A_c)$ is a simple $A_c$-algebra. Since $A$ is normal, it is hereditary in codimension 1 and from the discussion preceding Example \ref{egram}, we know that $A_c/J(A_c)\simeq\prod_{i=1}^n B$, where $B$ is a central simple $L$-algebra and $L$ is a cyclic field extension of $K(D_i)$. Then $\ti{K}(D_i)=L^n$ and $\ti{D_i}$ is irreducible implies that $n=1$. It follows that $A_c/J(A_c)\simeq B$ is a central simple $L$-algebra, implying $A_c/J(A_c)$ simple and hence maximal.
\end{proof}
\end{lemma}

To construct these cyclic algebras, we are interested in finding invertible bimodules $L_\s$ such that $\ls^n\simeq\oy(-C)$. Our first objects of interests are relations of the form $L_\s^n\longrightarrowlim^\sim\oy$, to which end we define $Rel_{io}$ to be the set of all (isomorphism classes of) relations $\phi:L_\s^n\longrightarrowlim^\sim\oy$ which are isomorphisms and satisfy the overlap condition. The tensor product endows $Rel_{io}$ with an abelian group structure as follows: given two relations $\phi\colon\ls^n\to\oy$, $\psi\colon M_{\sigma}^n\to\oy$, we define their product to be the relation
\begin{eqnarray*}
\phi\otimes\psi\colon (L\otimes_{\oy}M)_\s^n\longrightarrowlim^\sim \ls^n\otimes M_{\sigma}^n\longrightarrowlim^{\phi\otimes\psi}\oy\otimes\oy=\oy.
\end{eqnarray*}
There exists a subgroup $E$ of $Rel_{io}$ which will soon play an important role. Its definition is technical and we need not state it here. It suffices to say that elements of $Rel_{io}$\index{$Rel_{io}$} in the same coset of $E$ result in Morita equivalent cyclic algebras, that is, cyclic algebras with equivalent module categories. We refer the reader to the discussion following Corollary 3.4 of \cite{cyc}. 
\begin{remark}
 From the discussion preceding \cite[Lemma 3.5]{cyc}, we see that relations of the form ${\ls^n\simeq\oy}$ can be classified using cohomology. We consider $\pic Y$ as a $G$-set and recall that since $G$ is cyclic, the group cohomology of any $G-$set $M$ can be computed as the cohomology of the periodic sequence
\begin{eqnarray*}
\ldots\longrightarrowlim^N M \longrightarrowlim^D M\longrightarrowlim^N M \longrightarrowlim^D\ldots
\end{eqnarray*}
where $N=(1+\s+\ldots\s^{n-1}), D=(1-\s).$
Thus 1-cocycles of the $G$-set $\tmop{Pic}Y$ are precisely invertible bimodules $\ls$ such that ${\ls^n\simeq\oy}$. Moreover, from \cite[Lemma 3.5]{cyc}, we have a group homomorphism $f \colon Rel_{io}/E \to H^1(G,\tmop{Pic}Y)$ which sends a relations $\phi \colon \ls^n\simeq \oy$ to the class $\lambda \in \hh$ of the $1$-cocycle $L$.
\end{remark}

The following result will prove invaluable in retrieving elements of $Rel_{io}/E$ and thus constructing cyclic algebras.

\begin{proposition}\label{prop}\cite[Lemma 3.5, Theorem 3.9]{cyc}
Suppose that $Z$ is smooth and $\pi\colon Y\to Z$ is \'etale. Then the following sequence is exact:
\begin{eqnarray*}
0\lrw(\tmop{Pic}Y)^G\lrw H^2(G,\os(Y)^*)\lrw Rel_{io}/E\lrw H^1(G,\tmop{Pic}Y) \lrw H^3(G,\os(Y)^*).
\end{eqnarray*}
Moreover, $Rel_{io}/E\simeq\tmop{Br}(Y/Z)$.
\end{proposition}

%
%
%
%
%
%
The following proposition allows us to verify easily in certain cases that relations satisfy the overlap condition.
\begin{proposition}\cite[Prop. 4.1]{cyc}\label{ol}
Suppose that $Y$ is smooth and quasi-projective, $\pi\colon Y \to Z$ $n\colon 1$, $\os(Y)^*=k^*$ (this last condition holds, for example, if $Y$ is a projective variety), and the lowest common multiple of the ramification indices of $\pi\colon Y\to Z$ is $n$. Then all relations constructed from elements of $H^1(G,\tmop{Pic}Y)$ satisfy the overlap condition.
\end{proposition}
The following remark will ensure that many of the orders constructed herein are nontrivial in $\tmop{Br}(K(Z))$.
\begin{remark}\cite[Cor 4.4]{cyc}\label{inbr}
Assuming $Y$ to be smooth and projective, if $\pi\colon Y\to Z$ is totally ramified at an irreducible divisor $D\subset Y$, we have a embedding $\Psi\colon\hh\hookrightarrow\tmop{Br}(K(Y)/K(Z))$ given by the following: let $L\in\tmop{Pic}Y$ represent a $1$-cocycle, $\phi$ the corresponding relation, which is unique up to scalar multiplication; Then $\Psi(L)\assign K(A(Y;\ls,\phi))\in\tmop{Br}(K(Y)/K(Z))$.
\end{remark}
%

%
%
%

%% file: c2.tex
\chapter{Constructing Numerically Calabi-Yau orders}\label{ncy}

In this chapter, unless explicitly stated otherwise, all schemes are defined over $\mathbb{C}$ and all sheaf cohomology is complex analytic.
\section{Numerically Calabi-Yau Orders on $\p2$}

In \cite{ncy}, Chan and Kulkarni classify numerically Calabi-Yau orders on surfaces, which are the noncommutative analogues of surfaces of Kodaira dimension zero and which we now define here.

\begin{definition}\cite{ncy}\index{order!numerically Calabi-Yau}
Let $A$ be a maximal order on a surface $Z$ with ramification curves $D_i$ and corresponding ramification indices $e_i$. Then the {\bf canonical divisor}\index{order!canonical divisor} $K_A\in\tmop{Div}Z$ is defined by
\begin{eqnarray*}
 K_A & \assign & K_Z+\sum\left(1-\frac{1}{e_i}\right)D_i.
\end{eqnarray*}
We say an order $A$ is {\bf numerically Calabi-Yau} if $K_A$ is numerically trivial.
\end{definition}
In this chapter, we construct examples of numerically Calabi-Yau orders on rational surfaces, beginning with $\p2$. 
The first interesting numerically Calabi-Yau orders on $\p2$ are those ramified on a smooth sextic with ramification index 2. In the following, we construct {\bf quaternion orders} (that is, orders of rank $4$)\index{order!quaternion} on $\p2$ with the desired ramification.

\subsection{Constructing orders ramified on a sextic}\label{sextic}

We wish to construct quaternion orders on $\p2$ ramified on a sextic $C$. In \cite[Example 9.2]{cyc}, Chan does this using a K3 double cover of $\p2$ ramified on $C$ and we wish to take the same approach. However, our plan of attack is to reverse engineer the $K3$ surface. We first construct a $K3$ surface $Y$ with a desired Picard lattice and then construct an automorphism $\tau$ of $\tmop{Pic}Y$ which we show is induced by an automorphism $\s$ of $Y.$ We shall make our choices such that $H^1(G,\tmop{Pic}Y)$ is non-trivial, where $G=\langle\s|\s^2=1\rangle,$ and this will allow us to construct orders on $Y/G$, which by construction will be the projective plane. In contrast to Chan's approach, we shall be able to explicitly compute $H^1(G,\tmop{Pic}Y).$ To this end, we necessarily digress on the $K3$ lattice $\Ll$ and related phenomena.

\subsubsection{K3 surfaces}
The following results concerning K3 surfaces can be found in \cite{bpv}, unless stated otherwise. Recall that a {\bf K3 surface}\index{surface, K3} $Y$ is defined to be a smooth, projective surface with trivial canonical bundle such that $h^1(Y,\oy)=0.$\index{K3 surface} This is, in fact, one of many equivalent definitions. 
Triviality of the canonical bundle implies there exists a nowhere vanishing holomorphic $2-$form $\omega_Y$, unique up to multiplication by a scalar and known as the {\bf period}\index{K3 surface! period of}\index{$\omega_Y$} of $Y$.
We are interested in the K3 lattice $\Ll$, which we now define.
\begin{definition}\index{lattice}
A {\bf lattice} $(L,\langle,\rangle)$ is a finitely generated free $\mb{Z}$-module $L$ equipped with an integral, symmetric bilinear form $\langle,\rangle$. We say a sublattice $M$ of $L$ is {\bf primitive}\index{lattice!primitive} if $L/M$ is torsion free.
\end{definition}
\begin{example}\index{lattice! K3}
For any K3 surface $Y$, $H^2(Y,\mb{Z})\simeq\mb{Z}^{22}$ and, equipped with the cup product, is a lattice isomorphic to
\begin{eqnarray*}
\Lambda & := & \mb{E}\perp\mb{E}\perp\mb{H}\perp\mb{H}\perp\mb{H},
\end{eqnarray*}

where $\mb{E}\simeq\mb{Z}^8$ with bilinear form given by the matrix
$$
\left( \begin{array}{cccccccc}
-2 &  &  & 1 &\\
 & -2 & 1\\
 & 1 & -2 & 1\\
1 &  &  1 & -2 & 1\\
 &  &  & 1 & -2 & 1\\
 &  &  &  &  1 & -2 & 1\\
 &  &  &  &  &  1 & -2 & 1\\
 &  &  &  &  &  &  1 & -2\\
\end{array} \right)
$$

and $\mb{H}\simeq\mb{Z}^2$ with bilinear form given by

$$
\left( \begin{array}{cc}
0 & 1\\
1 & 0
\end{array} \right).
$$
We call $\Ll$ the K3 lattice, $\mb{H}$ the hyperbolic plane\index{lattice!hyperbolic plane}. In the following, we let $\{\lambda_1,\ldots,\lambda_8\}$ and $\{\lambda'_1,\ldots,\lambda'_8\}$ generate the first and second copies of $\mb{E}\subset\Ll$ respectively. We also let $\{\mu_1,\mu_2\}$ be generators for the first copy of $\mb{H},$ $\{\mu'_1,\mu'_2\}$ generators of the second and $\{\mu''_1,\mu''_2\}$ generators of the third.
\end{example}
Recall that corresponding to the exponential sequence 
\begin{eqnarray*}
 0\to\mb{Z}\to\oy\to\oy^*\to 0
\end{eqnarray*}
 we have a long exact sequence 
\begin{eqnarray*}
\ldots\lrw H^1(Y,\oy)\to H^1(Y,\oy^*)\to H^2(Y,\mb{Z})\lrw\ldots
\end{eqnarray*}
Since $H^1(Y,\oy)$ vanishes for K3 surfaces, we have an injection $\tmop{Pic}Y\hookrightarrow H^2(Y,\mb{Z})$ and this is a primitive embedding.
For K3 surfaces, $\tmop{Pic}Y\simeq\tmop{NS}(Y),$ the Neron-Severi lattice. In the following, many of the results we state about $\tmop{Pic}Y$ are stated about $\tmop{NS}(Y)$ in the literature.
We define the {\bf transcendental lattice}\index{lattice!transcendental}\index{$T_Y$} ${T_Y\assign(\tmop{Pic}Y)^\perp}$, which is also a primitive sublattice of $H^2(Y,\mb{Z}).$ The following lemma will prove very useful shortly.

\begin{lemma}\label{pictembed}
 $H^2(Y,\mb{Z})$ contains an isomorphic copy of $\tmop{Pic}Y\oplus T_Y$ as a sublattice and $H^2(Y,\mb{Z})/(\tmop{Pic}Y\oplus T_Y)$ is a torsion $\mb{Z}-$module.
\begin{proof}
Let $L\in\tmop{Pic}Y\cap T_Y.$ Then $L\cdot M=0$ for all $M\in\tmop{Pic}Y$. However, on a K3 surfaces, the only line bundle numerically equivalent to $\oy$ is $\oy$ itself, implying $L=0\in H^2(Y,\mb{Z}).$ Thus the lattice $\tmop{Pic}Y\oplus T_Y$ embeds in $H^2(Y,\mb{Z}).$ Since the cup product is non-degenerate, rank$(H^2(Y,\mb{Z}))=$ rank$(\tmop{Pic}Y)+$ rank$(T_Y)$ and the result follows.
\end{proof}
\end{lemma}
The fact that $K_Y$ is trivial for a K3 surface $Y$ makes studying the geometry of K3 surfaces far simpler than would otherwise be. The following results, concerning line bundles and curves on K3 surfaces, demonstrate this and will be made use of shortly.

\begin{proposition}\label{rr}
 Let $L$ be a line bundle on a K3 surface $Y$ such that $L\cdot L\geq-2$. Then either $L$ or $L^{-1}$ is an effective class, that is, either $h^0(L)>0$ or $h^0(L^{-1})>0$. Moreover, if $L\cdot L=-2$ and $h^0(L)>0$, there is a unique effective divisor $D$ such that $L\simeq \oy(D)$.
\begin{proof}
This is \cite[Chap. VIII, Prop 3.6]{bpv}.
\end{proof}
\end{proposition}

\begin{remark}\label{k3curve}
 For an irreducible smooth curve $C$ on a K3 surface, the adjunction formula for curves on a surface yields $g(C)=C^2/2+1$. Thus $C^2\geq -2$ and if $C^2=-2$, $C$ is necessarily rational. In this case we call $C$ a {\bf nodal curve}\index{nodal curve} since we can blow it down, but only at the expense of creating a nodal singularity.
\end{remark}

The following result, which is a formulation of the "surjectivity of period map"(\cite[Chap. VIII, Sec. 14]{bpv}), will allow us to construct $K3$ surfaces with our desired Picard lattices:
\begin{proposition}\label{mor}\cite[Cor. 1.9]{mor84}
Let $S\longhookrightarrow\Ll$ be a primitive sublattice with signature $(1,\rho-1),$ where $\rho=\tmop{rank}(S).$ Then there exists a K3 surface $Y$ and an isometry ${\tmop{Pic}Y\simeq S}.$ 
\end{proposition}
\begin{remark}\label{prim}
If $S$ is a direct summand of $\Ll$, then $S$ is a primitive sublattice. This observation will prove invaluable in applying the above Proposition \ref{mor} to verify the existence of certain K3 surfaces.
\end{remark}

\begin{example}\label{s_2}
Set $S=\mb{Z}^3=\langle s_1,s_2,s_3\rangle$ with bilinear form given by
$$
\left( \begin{array}{ccc}
-2 & 3 & 0\\
3 & -2 & 1\\
0 & 1  & -2
\end{array}\right).
$$
We can embed $S$ in $\Ll$ via
\begin{align*}
\gamma: S & \longhookrightarrow  \Ll\\
        s_1& \longmapsto          \la1+\mu_1\\
	s_2& \longmapsto          \la2+3\mu_2\\
	s_3& \longmapsto         \la3.
\end{align*}Since $\{\gamma(s_1),\gamma(s_2),\gamma(s_3),\la4,\ldots,\la8,\la1',\ldots,\la8',\mu_1,\mu_2,\mu_1',\mu_2',\mu_1'',\mu_2''\}$ is a basis for $\Ll$, $\gamma$ is a primitive embedding by Remark \ref{prim}. Moreover, the signature of $S$ is $(1,2)$ (Maple performed this calculation), implying by Proposition \ref{mor} that there exists a $K3$ surface $Y$ and an isometry $\tmop{Pic}Y\simeq S.$ Since $s_i^2=-2,$ Proposition \ref{rr} tells us that for each $i$, either $s_i$ is an effective class or $-s_i$ is. We assume without loss of generality that $s_1$ is effective. Then since $s_1\cdot s_2=3$ and $s_2\cdot s_3=1,$ $s_2$ and $s_3$ are necessarily effective classes. Similarly, $(s_1+s_2-s_3)^2=-2$ and one can deduce that $s_1+s_2-s_3$ is an effective class.  We shall soon see in Remark \ref{bita} that $Y$ is a double cover of $\p2$ ramified on a sextic $C$ with two tritangents.
\end{example}

\begin{notation}
 For the rest of this chapter, given a class of line bundles $s_i$ on a surface $Y$, we shall abuse notation by using $s_i$ to mean a representative line bundle of this class. For an effective class $s_i$, we shall also let $S_i$ be an effective divisor such that $s_i\simeq\oy(S_i)$.
\end{notation}

We shall need to retrieve an ample line bundle on $Y$ and the following lemma will allow us to do so.

\begin{lemma}\label{ample}
 Let $Y$ be a K3 surface and $\tmop{Pic}Y=\langle s_1,\ldots,s_n\rangle$, where the $s_i$ are effective classes. Let $s\in\tmop{Pic}Y$ be an effective class such that $h^0(s-s_i)>0$, for all $i$. Assume that $s^2>0$ and for all $i$, the following hold: $s\cdot s_i> 0$ and $s\cdot(s-s_i)> 0$. Then $s$ is ample.

\begin{proof}
By the Nakai-Moishezon criterion \cite[Chap. V, Thm 1.10]{ag}, $s$ is ample if and only if both $s^2>0$ and $s\cdot C> 0$ for all irreducible curves $C$ on $Y$. The first condition is one of our hypotheses so we need only verify the second. 
Since we have assumed that $s\cdot s_i> 0$ and $s\cdot(s-s_i)> 0$ for all $i$, we need only show that $s\cdot C>0$ for $C\nsim s_i$ and ${C\nsim(s-s_i)}$. Such a $C$ is an effective class distinct from the effective classes $s_i$ and $s-s_i$, implying $C\cdot(s-s_i)\geq 0$ and $C\cdot s_i\geq 0$ for all $i$. Moreover, since the $s_i$ generate $\tmop{Pic}Y$, there exists a $j$ such that $C\cdot s_j>0$ (otherwise $C\cdot D=0$ for all $D\in\tmop{Pic}Y$, implying $C\sim0$). Then
\begin{eqnarray*}
 C\cdot s & = & C\cdot s_j+C\cdot(s-s_j)\\
                  & > & 0.
\end{eqnarray*}
By the Nakai-Moishezon criterion, $s$ is an ample class.

\end{proof}
\end{lemma}

\begin{remark}\label{nodal}
We now return to the setting of Example \ref{s_2}: for $i\in\{1,2,3\}$, each $S_i$ is unique by Proposition \ref{rr}. We now show that each $S_i$ is nodal: since $s=s_1+s_2$ satisfies the conditions of Lemma \ref{ample}, $s$ is ample. Moreover, since $s\cdot s_i=1$, for all $i$, given the fact that an ample divisor will intersect an effective class strictly positively, each $S_i$ is irreducible and thus nodal. Similarly $S_4\simeq S_1+S_2-S_3$ is a nodal class.
\end{remark}
\begin{figure}[H]
\begin{center}
\includegraphics*[trim= 0 208 0 0]{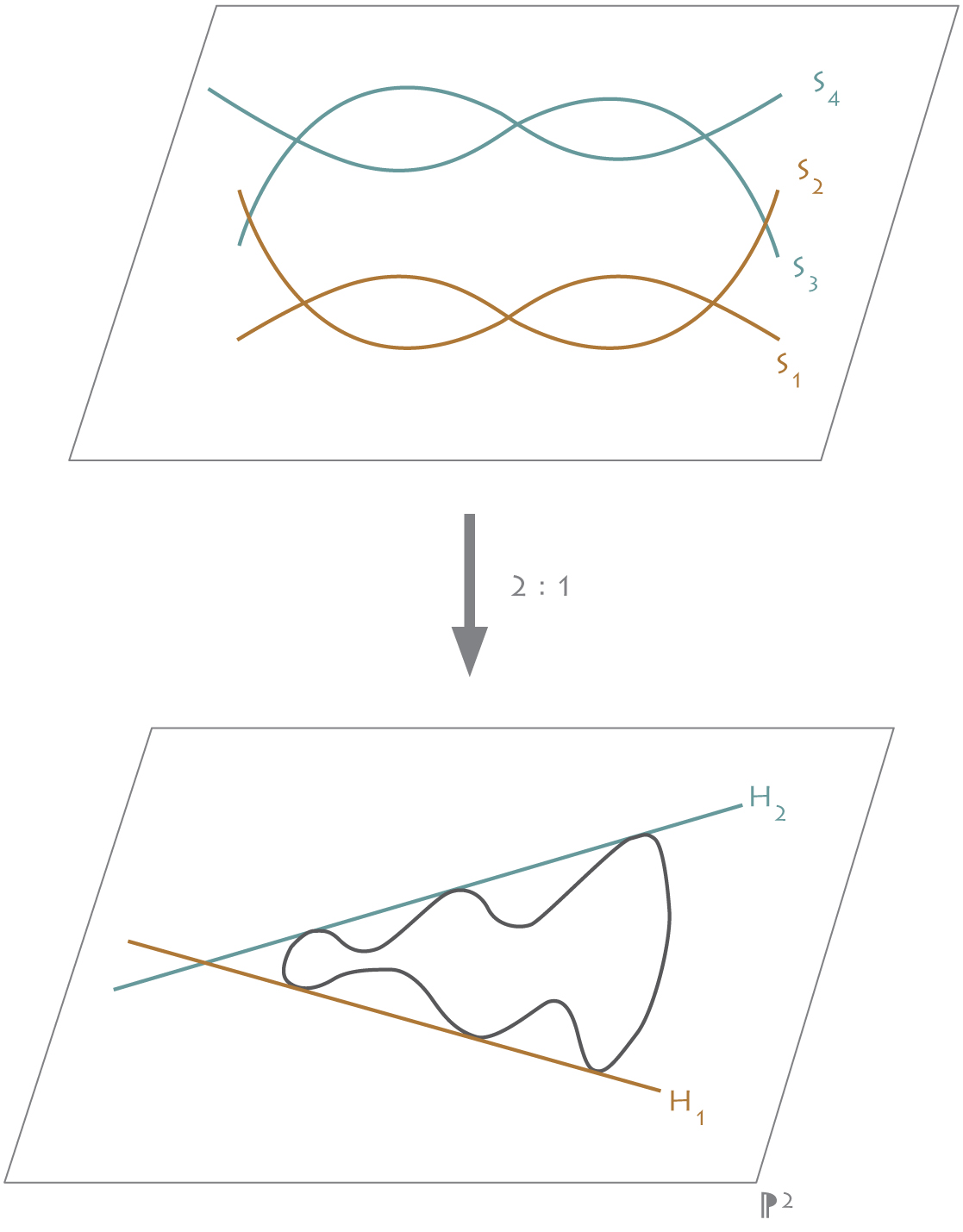}
\end{center}
\caption{Configuration of nodal curves on the K3 surface $Y$.}
\end{figure}

\begin{example}\label{orsol_mat}
Example \ref{s_2} is a specific instance of the following: for $n\in\{3,\ldots,18\}$, there exists a K3 surface $Y$ with Picard lattice isomorphic to $S=\mb{Z}^{n}=\langle s_1,\ldots,s_n\rangle$, the intersection form given by the $n\times n$ submatrix $M$ which is formed by the first $n$ rows and the first $n$ columns of

{\setlength\arraycolsep{2pt}
$$
Q = \left( \begin{array}{cccccccccccccccccc}
    -2   &  3  &   0   &  1  &   1  &   1  &   1  &   1  &   1   &  1  &   1   &  1  &   1   &  1  &   1  &   1  &   1  &   1\\
     3   & -2  &   1   &  0  &   0  &   0  &   0  &   0  &   0   &  0  &   0   &  0  &   0   &  0  &   0  &   0  &   0  &   0\\
     0   &  1  &  -2   &  1  &   0  &   0  &   0  &   0  &   0   &  0  &   0   &  0  &   0   &  0  &   0  &   0  &   0  &   0\\
     1   &  0  &   1   & -2  &   1  &   0  &   0  &   0  &   0   &  0  &   0   &  0  &   0   &  0  &   0  &   0  &   0  &   0\\
     1   &  0  &   0   &  1  &  -2  &   1  &   0  &   0  &   0   &  0  &   0   &  0  &   0   &  0  &   0  &   0  &   0  &   0\\
     1   &  0  &   0   &  0  &   1  &  -2  &   1  &   0  &   0   &  0  &   0   &  0  &   0   &  0  &   0  &   0  &   0  &   0\\
     1   &  0  &   0   &  0  &   0  &   1  &  -2  &   1  &   0   &  0  &   0   &  0  &   0   &  0  &   0  &   0  &   0  &   0\\
     1   &  0  &   0   &  0  &   0  &   0  &   1  &  -2  &   0   &  0  &   0   &  0  &   0   &  0  &   0  &   0  &   0  &   0\\
     1   &  0  &   0   &  0  &   0  &   0  &   0  &   0  &  -2   &  0  &   0   &  1  &   0   &  0  &   0  &   0  &   0  &   0\\
     1   &  0  &   0   &  0  &   0  &   0  &   0  &   0  &   0   & -2  &   1   &  0  &   0   &  0  &   0  &   0   &  0   &  0\\
     1   &  0  &   0   &  0  &   0  &   0  &   0  &   0  &   0   &  1  &  -2   &  1  &   0   &  0  &   0  &   0  &   0  &   0\\
     1   &  0  &   0   &  0  &   0  &   0  &   0  &   0  &   1   &  0  &   1   & -2  &   1   &  0  &   0  &   0  &   0  &   0\\
     1   &  0  &   0   &  0  &   0  &   0  &   0  &   0  &   0   &  0  &   0   &  1  &  -2   &  1  &   0  &   0  &   0  &   0\\
     1   &  0  &   0   &  0  &   0  &   0  &   0  &   0  &   0   &  0  &   0   &  0  &   1   & -2  &   1  &   0  &   0  &   0\\
     1   &  0  &   0   &  0  &   0  &   0  &   0  &   0  &   0   &  0  &   0   &  0  &   0   &  1  &  -2  &   1  &   0  &   0\\
     1   &  0  &   0   &  0  &   0  &   0  &   0  &   0  &   0   &  0  &   0   &  0  &   0   &  0  &   1  &  -2  &   0  &   0\\
     1   &  0  &   0   &  0  &   0  &   0  &   0  &   0  &   0   &  0  &   0   &  0  &   0   &  0  &   0  &   0  &  -2  &   0\\
     1   &  0  &   0   &  0  &   0  &   0  &   0  &   0  &   0   &  0  &   0   &  0  &   0   &  0  &   0  &   0  &   0  &  -2
\end{array}\right).
$$
}
We see this below in Theorem \ref{thawesome}. 
Now given a K3 surface with such a Picard lattice, we may assume that the all the $s_i$ are effective just as in Example \ref{s_2}. We now show that $s_1+s_2$ is ample: firstly, $(s_1+s_2)^2=2>0$ and
elementary computations yield that for all $i$, $(s_1+s_2)\cdot s_i=1$, $(s_1+s_2)\cdot (s_1+s_2-s_i)=1$ and $h^0(s-s_i)>0$. 
Thus $s_1+s_2$ satisfies the conditions of Lemma \ref{ample} and is consequently ample. The same argument as in Remark \ref{nodal} implies that the $s_i$ are all effective nodal classes, as are the $s_1+s_2-s_i$.
\end{example}

Given a surface $Y$ as in Example \ref{orsol_mat}, we would like to construct an automorphism of $Y$ by giving an isometry of $H^2(Y,\mb{Z})$ and thus the question we need to answer is: given an isometry $H^2(Y,\mb{Z})\lrw H^2(Y,\mb{Z}),$ how can we tell if it is induced by an automorphism $\s:Y\lrw Y?$ The Strong Torelli theorem aids us in this.
Before we introduce the Strong Torelli theorem, however, we need to define an effective Hodge isometry. To this end, recall that there is a Hodge decomposition of $H^2(Y,\mb{C})$:
\begin{eqnarray*}\label{hodge}
H^2(Y,\mb{C}) & \simeq & H^{0,2}(Y)\oplus H^{1,1}(Y)\oplus H^{2,0}(Y),
\end{eqnarray*}
where $H^{p,q}(Y)\simeq H^q(Y,\Omega^p)$\index{$H^{p,q}(Y)$}. Note that $\tmop{Pic}Y=H^2(Y,\mb{Z})\cap H^{1,1}(Y)$ (this is  a classical theorem of Lefschetz; see \cite[Sec. 1.3]{k3en} for details) and that $\omega_Y\in H^{2,0}(Y)\subset T_Y$.
\begin{definition} Let $Y,Y'$ be surfaces. An isometry of lattices
\begin{eqnarray*}
H^2(Y,\mb{Z}) & \lrw & H^2(Y',\mb{Z})
\end{eqnarray*}
is called an \textbf{ effective Hodge isometry}\index{effective Hodge isometry} if its $\mb{C}\mhyphen$linear extension
%
%
%
\begin{enumerate}
 \item sends $H^{2,0}(Y)$ to $H^{2,0}(Y')$ and
\item maps the class of some ample divisor on $Y$ to the class of an ample divisor on $Y'$.
\end{enumerate}

\end{definition}

This definition is not entirely standard but it is an equivalent formulation and can be found in \cite[Introduction]{torelli}. We use it because it suits our purposes nicely. We are now able to state the Strong Torelli theorem \cite[Ch. VIII, Thm 11.1]{bpv}.

\begin{theorem}[Strong Torelli theorem]\label{tor}\index{Strong Torelli theorem}
Let $Y$ and $Y'$ be two K3 surfaces and $\phi\colon H^2(Y,\mb{Z})\to H^2(Y',\mb{Z})$ an effective Hodge isometry. Then there exists a unique biholomorphic $\s:Y'\lrw Y$ such that $\phi=\s^*.$
\end{theorem}

We would also like to determine the quotients of K3 surfaces by certain involutions. To this end, we shall require the following definition and proposition.

\begin{definition}\index{involution! symplectic}\index{involution! anti-symplectic}
An involution $\psi$ on a K3 surface $Y$ is called {\bf symplectic} if ${\psi^*(\omega_Y)=\omega_Y}$. It is called {\bf anti-symplectic} if ${\psi^*(\omega_Y)=-\omega_Y}$.
\end{definition}

\begin{proposition}\cite[Prop.1.11]{k3sym}\label{k3sym}
 Let $\pi:Y\to Y/G$ be the quotient of a K3 surface by an anti-symplectic involution $\s.$ If Fix$_Y(\s)\neq\varnothing,$ then Fix$_Y(\s)$ is a disjoint union of smooth curves and $Y/G$ is a smooth, projective rational surface. Furthermore, Fix$_Y(\s)=\varnothing$ if and only if $Y/G$ is an Enriques surface.
\end{proposition}

\begin{remark}\label{irred}
 From \cite[Theorem 1.12]{k3sym}, if $Y/G$ in Proposition \ref{k3sym} is rational, then either (1) Fix$(\s)=\cup_i R_i\cup D_g$ (where the $R_i$ are smooth disjoint nodal curves and $D_g$ is a smooth curve of genus $g$) or (2) Fix$(\s)=D_1\cup D_2$ (where the $D_i$ are linearly equivalent elliptic curves). In case (2), $Y/G$ is both rational and elliptically fibred over $\p1$ (this is from the proof of \cite[Theorem 1.12]{k3sym})\footnote{See Chapter \ref{os}, Definition \ref{el} for definition of an elliptic fibration.}. This then implies that $\tmop{rk}(\tmop{Pic}(Y/G))= 10$ (see, for example, \cite[Introduction]{rat}).
\end{remark}

Finally, we would like to determine the structure of the Picard group of $Y/G$. We do so in the following lemma.

\begin{lemma}\label{inj}
 Let $\pi\colon Y\to Z$ be a double cover of a smooth, projective rational surface. Then ${\pi^*\colon\tmop{Pic}Z\to\tmop{Pic}Y}$ is an injection. Moreover, if the ramification locus $D\subset Y$ is irreducible, $\pi^*$ surjects onto $(\tmop{Pic}Y)^G$.
\begin{proof}
	Since $\pi$ is a double cover, it is cyclic and $\pi_*\os_{Y}\simeq \os_Z\oplus M$ for some ${M\in\tmop{Pic}Z}$. Now assume that $L\in\tmop{ker}(\pi^*)$. By the projection formula (see, for example, \cite[Chap. II, Ex. 5.1(d)]{ag}), ${\pi_*\os_{Y}\simeq\pi_*\os_{Y}\otimes L}$, implying $ \os_Z\oplus M \simeq L\oplus (L\otimes M)$. By Atiyah's Krull-Schmidt Theorem for vector bundles \cite{aks}, either $L\simeq \os_Z$ or $L\simeq M$. The latter would imply that $L^2\simeq\os_Z$. Since $Z$ is rational, $\tmop{Pic}Z$ has no torsion and this is impossible. Thus $L$ is necessarily trivial and $\pi^*$ is an injection.
We now demonstrate that when the ramification locus $D\subset Y$ is irreducible, $\pi^*$ surjects onto $(\tmop{Pic}Y)^G$ : we let $Y'=Y-D, Z'=Z-\pi(D)$. We then have a commutative diagram
$$
\xymatrix{
 Y' \ar[r]^{\pi'}\ar[d] & Z'\ar[d]\\
Y \ar[r]^\pi                  &  Z
}
$$
and a corresponding commutative diagram
\begin{eqnarray}\label{1}
\xymatrix{
 \tmop{Pic}Z \ar[r]^{\pi^*}\ar[d] & (\tmop{Pic}Y\ar[d])^G\\
\tmop{Pic}Z' \ar[r]^{\pi'^*}                  &  (\tmop{Pic}Y')^G
}
\end{eqnarray}
Since $\pi'\colon Y'\to Z'$ is \'etale, we may apply the Hochschild-Serre spectral sequence (as given in \cite[Chap. III, Theorem 2.20]{milnet})
\begin{eqnarray*}
 E^{pq}_2\assign H^p(G,H^q(Y',\os_{Y'}^*)) & \Longrightarrow H^{p+q}(Z',\os_{Z'}^*).
\end{eqnarray*}
The relevant part of $E^{pq}_2$ is
\begin{eqnarray*}
\begin{array}{lll}
H^0(G,\tmop{Pic}Y') & H^1(G,\tmop{Pic}Y') &\\
H^0(G,\mathcal{O}(Y')^*) & H^1(G,\mathcal{O}(Y')^*) & H^2(G,\mathcal{O}(Y')^*)
\end{array}
\end{eqnarray*}
From this, the following is exact
\begin{eqnarray}\label{2}
0\lrw H^1(G,\mathcal{O}(Y')^*)\lrw \tmop{Pic}Z'\lrw (\tmop{Pic}Y')^G\lrw H^2(G,\mathcal{O}(Y')^*)\lrw\ldots
\end{eqnarray}
Since $Y$ is projective and $D$ irreducible, $\mathcal{O}(Y')^*\simeq k^*$. Moreover, since $G$ acts trivially on the scalars $k^*$, $H^1(G,\mathcal{O}(Y')^*)\simeq\mb{Z}/2\mb{Z}$ and $H^2(G,\mathcal{O}(Y')^*)=0$. Then the exact sequence (\ref{2}) becomes
\begin{eqnarray}\label{2a}
0\lrw \mb{Z}/2\mb{Z} \lrw \tmop{Pic}Z'\lrw (\tmop{Pic}Y')^G\lrw 0.
\end{eqnarray}
We also know there are exact sequences and a commutative diagram

\begin{eqnarray}\label{3}
\xymatrix{
 0  \ar[r]  & \mb{Z}\pi(D)  \ar[r] \ar[d]^a   &   \tmop{Pic}Z  \ar[r] \ar[d]^{\pi^*}   &  \tmop{Pic}Z'  \ar[r] \ar[d]^{\pi'^*} & 0\\
0     \ar[r]            &       \mb{Z}D  \ar[r]     &   (\tmop{Pic}Y)^G    \ar[r]           &     (\tmop{Pic}Y')^G  \ar[r] &  0
}
\end{eqnarray}
where $a\pi(D)=2D$.
Collating the information contained in (\ref{1}), (\ref{2a}) and (\ref{3}), we retrieve the following commutative diagram
$$
\xymatrix{
         &                &                  0                         &             0                        &   \\
0 \ar[r] &\mb{Z}/2\mb{Z} \ar[r]   &   \tmop{Pic}Z' \ar[r]^{\pi'^*} \ar[u]      &  (\tmop{Pic}Y')^G  \ar[r] \ar[u]      &    0\\
   &   0        \ar[r]    &   \tmop{Pic}Z  \ar[r]^{\pi^*} \ar[u]      &  (\tmop{Pic}Y)^G   \ar[r]   \ar[u]    &    \tmop{coker}\pi^*\\
    &  0        \ar[r]    &    \mb{Z}\pi(D)      \ar[r]^a \ar[u]              &    \mb{Z}D          \ar[r]   \ar[u]    &     \mb{Z}/2\mb{Z}\\
     &                    &      0 \ar[u]                             &      0 \ar[u]                           & 
}
$$
Applying the snake lemma, we see that $\tmop{coker}\pi^*=0$ and this completes the proof.
%
\end{proof}
\end{lemma}

\begin{corollary}\label{intform}
 Let $\pi\colon Y\to Z$ be as above and $Y$ a K3 surface. Then there is an isometry of lattices $\tmop{Pic}Z\simeq \frac{1}{2}(\tmop{Pic}Y)^G$.
\begin{proof}
 Firstly, we know that $\tmop{Pic}Z\simeq (\tmop{Pic}Y)^G$. Then from \cite[Prop. I.8(ii)]{bo}, we see that for any $L_1,L_2\in\tmop{Pic}Z$, ${\pi^*L_1\cdot\pi^*L_2=2(L_1\cdot L_2)}$ and the result follows directly.
\end{proof}

\end{corollary}

%

%
We have now developed the theory necessary to construct our orders.
\begin{theorem}\label{thawesome}
Let $n\in\{3,\ldots,18\}$. Then 
\begin{enumerate}[i)]
 \item there exists a K3 surface $Y$ with Picard lattice isomorphic to $S\simeq\mb{Z}^{n}=\langle s_1,\ldots,s_n\rangle$, the intersection form given by the $n\times n$ submatrix $M$ 	which is formed by the first $n$ rows and the first $n$ columns of the matrix $Q$ in Example \ref{orsol_mat}.
	Further, there exists an involution $\s\colon Y\to Y$ such that the corresponding quotient morphism ${\pi\colon Y\to Z\assign Y/G}$ is the double cover of $\p2$ ramified on a smooth sextic $C$;
\item ${\hh\simeq (\mb{Z}/2\mb{Z})^{n-2}}$, generated by $L_i=s_1-s_i$, for $i\in\{3,\ldots,n\}$, and all relations satisfy the overlap condition.
%
Then, for $m_i\in\{0,1\}$ (not all zero), the corresponding $A(Y;(\otimes_{i=3}^nL_i^{m_i})_{\s})=\oy\oplus(\otimes_{i=3}^nL_i^{m_i})_{\s}$ are $2^{n-2}-1$ distinct maximal orders on $\p2$ ramified on $C$.
\end{enumerate}

\end{theorem}
\begin{proof}
\begin{enumerate}[i)]
 \item

Ignoring terms $s_j$ with $j>n$, we embed $S=\langle s_1,\ldots,s_{n}\rangle$ in $\Ll$ via
 \begin{equation*}
\begin{aligned}
\gamma\colon s_1 & \mapsto\la1+\mu_1, & s_2 & \mapsto\la2+3\mu_2, & s_3& \mapsto \la3,\\ 
s_4& \mapsto \la4, &s_5& \mapsto \la5+\mu_2, & s_6& \mapsto \la6+\mu_2,\\
s_7& \mapsto \la7+\mu_2, & s_8& \mapsto \la8+\mu_2, & s_9& \mapsto \la1'+\mu_2,\\
s_{10}& \mapsto \la2'+\mu_2, & s_{11}& \mapsto \la3'+\mu_2, &s_{12}& \mapsto \la4'+\mu_2,\\
s_{13}& \mapsto \la5'+\mu_2, & s_{14}& \mapsto \la6'+\mu_2, & s_{15}& \mapsto \la7'+\mu_2,\\
s_{16}& \mapsto \la8'+\mu_2, & s_{17}& \mapsto \mu_2+\mu_1'-\mu_2', & s_{18}& \mapsto \mu_2+\mu_1''-\mu_2''.\\
\end{aligned}
\end{equation*}
Since $\{\gamma(s_1),\ldots,\gamma(s_{18}),\mu_1,\mu_2,\mu_1',\mu_1''\}$ is a basis of $\Ll$, by Remark \ref{prim}, $\gamma$ is a primitive embedding. Also, $S$ has signature $(1,n-1)$ (this calculation was performed by Maple). By Proposition \ref{mor}, there is a K3 surface $Y$ and an isometry $\tmop{Pic}Y\simeq S$.  We define an isometry $\phi$ on $T_Y\oplus\tmop{Pic}Y$ as follows: for $t\in T_Y,\phi(t)=-t$; on $\tmop{Pic}Y, \phi(s_i)=s_1+s_2-s_i,i\in\{1,\ldots,n\}$. By Lemma \ref{pictembed}, there exists an $n\in\mb{Z}$ such that $T_Y\oplus\tmop{Pic}Y\supset n\Ll$ and this implies that $\phi$ extends to an isometry on $\Ll$ if and only if it preserves the integral lattice.
 MATLAB verifies that this is so (the MATLAB code for this can be found in Appendix \ref{code}), implying that $\phi$ extends to an isometry of $H^2(Y,\mb{Z})$, which we also denote $\phi$. Now we show that $\phi$ is an effective Hodge isometry: firstly, $H^{2,0}(Y)\subset T_Y$ and $\phi(t)=-t,$ for all $t\in T_Y$, imply that $\phi$ preserves $H^{2,0}(Y)$. Since $s_1+s_2$ is both fixed by $\phi$ and ample (the latter fact demonstrated in Example \ref{orsol_mat}), $\phi$ is an effective Hodge isometry and there exists an involution $\s\colon Y\to Y$ such that $\phi=\s^*$ by the Strong Torelli theorem.
Since $\omega_Y\in T_Y$ and $\s(t)=-t$ for all $t\in T_Y$, $\s$ is antisymplectic. Thus $Z=Y/G$ (here $G=\langle\s|\s^2=1\rangle$) is a smooth rational surface or an Enriques surface by Proposition \ref{k3sym}. Recalling that $S_i$ is the divisor such that $s_i\simeq\oy(S_i)$, $\pi_{|S_1\cup S_2}$ is the double cover of $\p1$ by two copies of $\p1$ intersecting in three points. Thus Fix$_Y(\s)\neq\varnothing$ and by Proposition \ref{k3sym}, $Z$ is a smooth rational surface. 
By Lemma \ref{inj}, $\pi^*\colon \tmop{Pic}Z \to (\tmop{Pic}Y)^G\simeq\mb{Z}(s_1+s_2)$ is injective, implying $\tmop{Pic}Z\simeq\mb{Z}$, from which we conclude that $Z\simeq\p2$. The Hurwitz formula ${\omega_Y=\pi^*(\omega_Z)\otimes\oy(R)}$, where $R$ is the ramification divisor, along with the fact that $\omega_Y\simeq\oy$ imply that $\pi$ is ramified on a sextic $C$. By Proposition \ref{k3sym}, $C$ is smooth and thus irreducible.
\item
We now compute $H^1(G,\tmop{Pic}Y)$. The kernel of $1+\s$ is generated by $s_1-s_i$, for $i\in\{2,\ldots,n\}$, while the image of $1-\s$ is generated by $s_1-s_2$ and $2(s_1-s_i)$, for $i\in\{3,\ldots,n\}$. Thus $H^1(G,\tmop{Pic}Y)\simeq(\mb{Z}/2\mb{Z})^{n-2}$. Moreover, all relations satisfy the overlap condition by Proposition \ref{ol}. By Remark \ref{inbr}, for $m_i\in\{0,1\}$ not all zero, the corresponding $A\assign\oy\oplus(\otimes_{i=3}^nL_i^{m_i})_{\s}$ are distinct maximal orders on $\p2$ ramified on $C$.
\end{enumerate}
\end{proof}
\begin{figure}[H]
\begin{center}
\includegraphics*[trim= 0 5 0 0]{thesis3v3.jpg}
\end{center}
\caption{The double cover $\pi\colon Y\to\p2$ in the $n=3$ case.}
\end{figure}
\begin{remark}\label{bita}
Since $C$ is smooth and irreducible, by Lemma \ref{inj}, $\pi^*$ is an isomorphism onto $(\tmop{Pic}Y)^G\simeq\mb{Z}(s_1+s_2)$, implying $\pi^*H=s_1+s_2$, where $H$ is a line on $Z$. Then for all $i$,  $S_i+\phi(S_i)\sim S_1+S_2$ is the inverse image of a line on $\p2$. Since $S_i$ and $\phi(S_i)$ are distinct rational curves, there are $(n-1)$ lines $H_i\subset\p2$ such that $\pi^{-1}(H_i)=S_i+\phi(S_i)$.
Noting that for all $i$, $S_i\cdot\phi(S_i)=3$, this then implies that $C$ has $n-1$ tritangents $H_1,\ldots,H_{n-1}$.
\end{remark}

\section{The construction of orders on ruled surfaces}
We now  construct orders on ruled surfaces, to which end we make the following definition.
\begin{definition}\index{ruled surface}
 A surface $Z$ is {\bf (geometrically) ruled} if there exists a smooth curve $C$ and a morphism $p\colon Z\to C$ such that, for all $c\in C$, the fibre $Z_c\simeq\p1$.
\end{definition}

Given a ruled surface $\rho\colon Z\to C$, we know that $Z=\mb{P}_C(E)$\index{$\mb{P}_C(E)$} the projectivisation of a rank $2$ vector bundle $E$ on $C$ (see \cite[Prop. III.7]{bo}).

\begin{proposition}\label{ruledpic}\cite[Proposition III.18]{bo}
 The Picard group of a ruled surface ${p:Z\to C}$ is given by
\begin{eqnarray*}
\tmop{Pic}Z & = & p^*\tmop{Pic}C\oplus\mb{Z}C_0.
\end{eqnarray*}
Moreover, $C_0^2=\tmop{deg}(E),F^2=0,C_0\cdot F=1$ and $K_Z\equiv-2C_0-(\tmop{deg}(E)+2g(C)-2)F,$ where $F$ is a fibre of $p.$ 
\end{proposition}

\begin{example}\index{Hirzerbruch surface}
The $n$th Hirzebruch surface, ${\mb{F}_n\assign\mb{P}_{\mb{P}^1}(\os_{\mb{P}^1}\oplus\os_{\mb{P}^1}(-n))}$\index{$\mb{F}_n$}, is ruled over $\p1$. In fact, these are all the ruled surfaces over the projective line (up to isomorphism) \cite[Prop.III.7]{bo}. 
\end{example}
\begin{remark}\label{hirz}

The Picard lattice of $\fn$ is given by $\tmop{Pic}\fn\simeq\mb{Z}^2$ with intersection form
$$
\left( \begin{array}{cc}
-n & 1 \\
1 & 0 
\end{array} \right).
$$
If $n$ is even, this Picard lattice is isometric to
$$\mb{H}=
\left( \begin{array}{cc}
0 & 1 \\
1 & 0 
\end{array} \right)
$$
with respect to the generators $C_0+\frac{n}{2}F, F$ of $\tmop{Pic}\fn$. If $n$ is odd, $\pic\fn$ is not isometric to $\mb{H}$ since $C_0^2=-n$ and $\mb{H}$ is an even lattice.
\end{remark}
\begin{figure}[H]
\begin{center}
\includegraphics[scale=0.5]{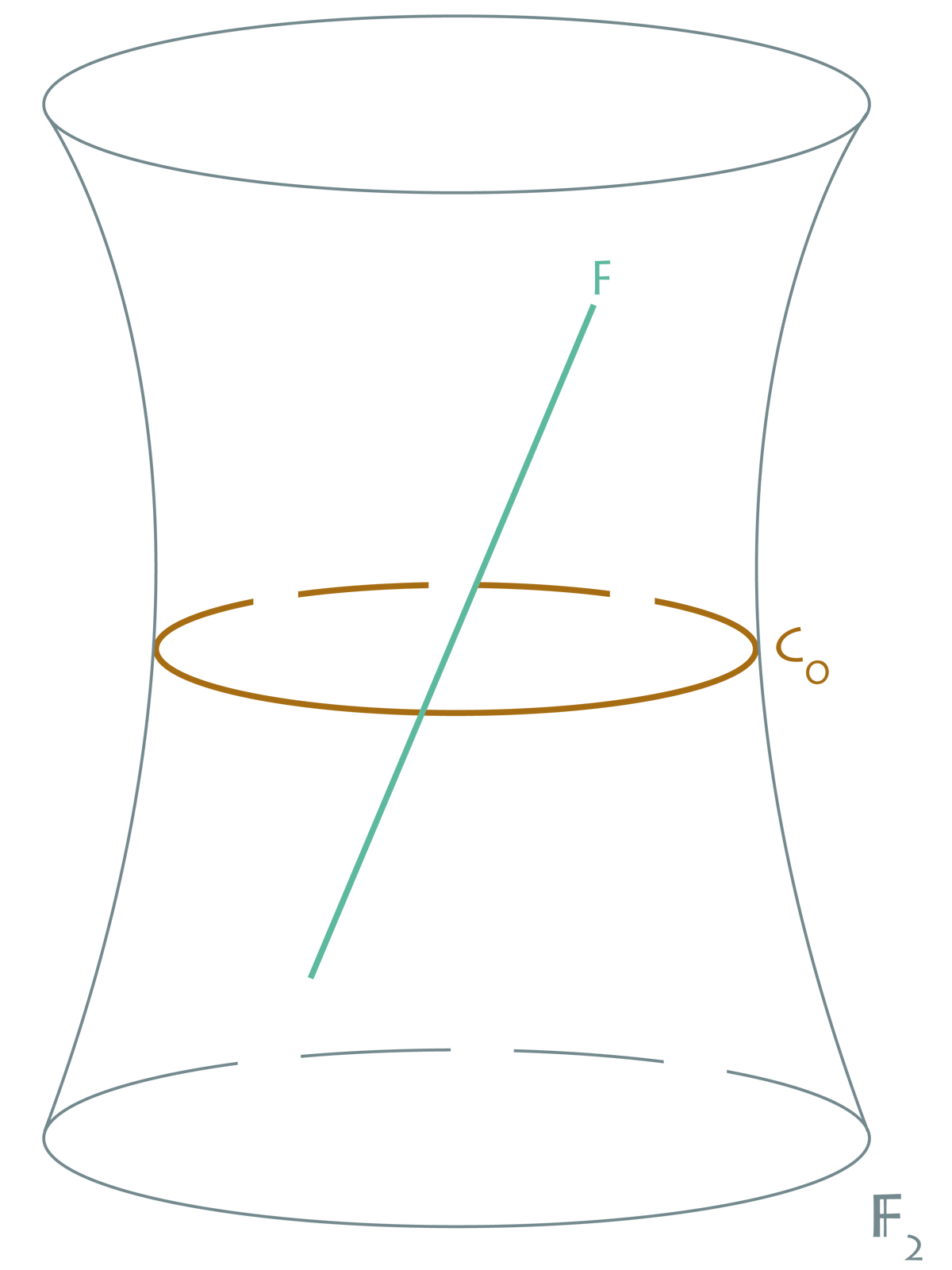}
\end{center}
\caption{The Hirzebruch surface $\mathbb{F}_2$.}
\end{figure}

\begin{lemma}\label{negdiv}
 Let $n>0$. The only Hirzebruch surface with an irreducible divisor $C_0$ such that $C_0^2=-n$ is $\fn$.
\begin{proof}
Assume there exists an irreducible curve $C\subset\mb{F}_m$ such that $C^2=-n$, where $n\neq m$. We know  that there exists a $C_0\subset\mathbb{F}_m$ such that $C_0^2=-m$. Since $C$ is linearly equivalent to neither $C_0$ nor $F$, by \cite[Chap. V, Cor. 2.18]{ag}, $C\sim aC_0+bF$,where $a,b>0$. However, since $C$ is irreducible, $C\cdot C_0\geq 0$ and $C\cdot F\geq 0$, implying $C^2=C\cdot(aC_0+bF)\geq 0$, contradicting the existence of such a curve $C$.
\end{proof}

\end{lemma}

\subsection{Orders on $\p1\times\p1$}
To demonstrate the versatility of the construction introduced in Section \ref{sextic}, we now perform the same trick to construct orders on the quadric surface $\p1\times\p1$. As before, we begin by constructing a K3 surface $Y$ which we shall eventually show to be a double cover of the quadric.
\begin{proposition}\label{p1xp1}
 Let $S=\mb{Z}^4=\langle s_1,s_2,s_3,s_4\rangle$ with bilinear form given by 
\begin{eqnarray*}
\left( \begin{array}{cccc}
0 & 1 & 1 & 1\\
1 &-2 & 2 & 0\\
1 & 2 &-2 & 0\\
1 & 0 & 0 &-2\\
\end{array} \right).
\end{eqnarray*}
Then there exists a K3 surface $Y$ such that $\tmop{Pic}Y\simeq S.$ Moreover, the generic member of $|s_1|$ is irreducible.

\begin{proof}
We embed $S$ in $\Ll$ via
 \begin{equation*}
\begin{aligned}
\gamma\colon s_1 & \mapsto\mu_1+\mu_1', & s_2 & \mapsto\la1+\mu_2+\mu_1'',\\
s_3& \mapsto \la4+\mu_2+\mu_2'', & s_4& \mapsto	\la2+\mu_2.
\end{aligned}
\end{equation*}
Since $\{\gamma(s_1),\ldots,\gamma(s_4),\la1,\ldots,\la8,\la1',\ldots,\la8',\mu_1,\mu_2'\}$ is a basis for $\Ll$, by Remark \ref{prim} $\gamma$ is a primitive embedding. Since $S$ also has signature $(1,3)$ (this calculation was performed by Maple), there exists a K3 surface $Y$ such that $\tmop{Pic}Y\simeq S$ by Proposition \ref{mor}. We may once again assume that the $s_i$ and $s_5\assign s_2+s_3-s_4$ are effective classes. Then explicit elementary computations demonstrate that $s=s_1+s_2+s_3$ satisfies the conditions of Lemma \ref{ample}, implying that $s$ is an ample class. For $i\in\{2,3,4,5\}$, $s\cdot s_i=1$ implying that each $S_i$, for $i\in\{2,3,4,5\}$, is irreducible and thus a nodal curve. 

We now show that the generic member of $|s_1|$ is irreducible:
since $s$ is ample and $s\cdot s_1=2$, any member of $|s_1|$ has at most two components. By \cite[Prop. 2.6]{projk31}, if $|s_1|$ has no fixed components, then every member of $|s_1|$ can be written as a finite sum $E_1+\ldots +E_n$ where $E_i\sim E$ for all $i$ and $E$ an irreducible curve of arithmetic genus 1. Then $s_1\cdot s_2$ is a multiple of $n$. This, along with the fact that  $s_1\cdot s_2=1$, implies that $n=1$. From \cite[Discussion (2.7.3)]{projk31}, $|s_1|$ has fixed components if and only if the generic member of $|s_1|$ is $E\cup R$, where $E$ is an irreducible curve of arithmetic genus 1 and $R$ is both a nodal curve and a fixed component of $|s_1|$. 

We know the following: $E+R\in|s_1|$, $R^2=-2$, $E^2=0$, and $s_1^2=0$; it follows that $E\cdot R=1$ and $R\cdot S_1=-1$. 
Then the intersection theory on $Y$ yields that $R\not\sim S_i$ for $i\in\{2,3,4\}$. 
Then $R\cdot S_4\geq 0$ and $E\cdot S_4\geq 0$. From \cite[discussion preceding Prop. 2.6]{projk31}, the curve $E$ defines a base-point free pencil of genus 1 curves on $Y$ and $S_4$ is not a component of any fibres, implying that $E\cdot S_4=1$ and thus $R\cdot S_4=0$. Similarly we see that $R\cdot S_2=R\cdot S_3=0$.  Letting $R\sim\sum_{i=1}^4a_iS_i$,
\begin{eqnarray}\label{s2}
 R\cdot S_2 & = & a_1-2a_2+2a_3\\ \label{s3}
 R\cdot S_3 & = & a_1+2a_2-2a_3\\ \label{s4}
 R\cdot S_4 & = & a_1-2a_4.
\end{eqnarray}
Since $R\cdot S_i=0$, for $i\in\{2,3,4\}$, (\ref{s2})$+$(\ref{s3}) yields $a_1=0$ and it follows that $a_2=a_3$. This, in conjunction with (\ref{s4}), tells us that $a_4=0$. Then $R\sim a_2(S_2+S_3)$, implying $R^2=0$ and thus $R$ cannot possibly be a nodal curve, yielding a contradiction.
We conclude that $|s_1|$ has no fixed components and its general member is an irreducible curve of arithmetic genus 1.
\end{proof}

\end{proposition}

\begin{figure}[H]
\begin{center}
\includegraphics*[trim= 0 280 0 0]{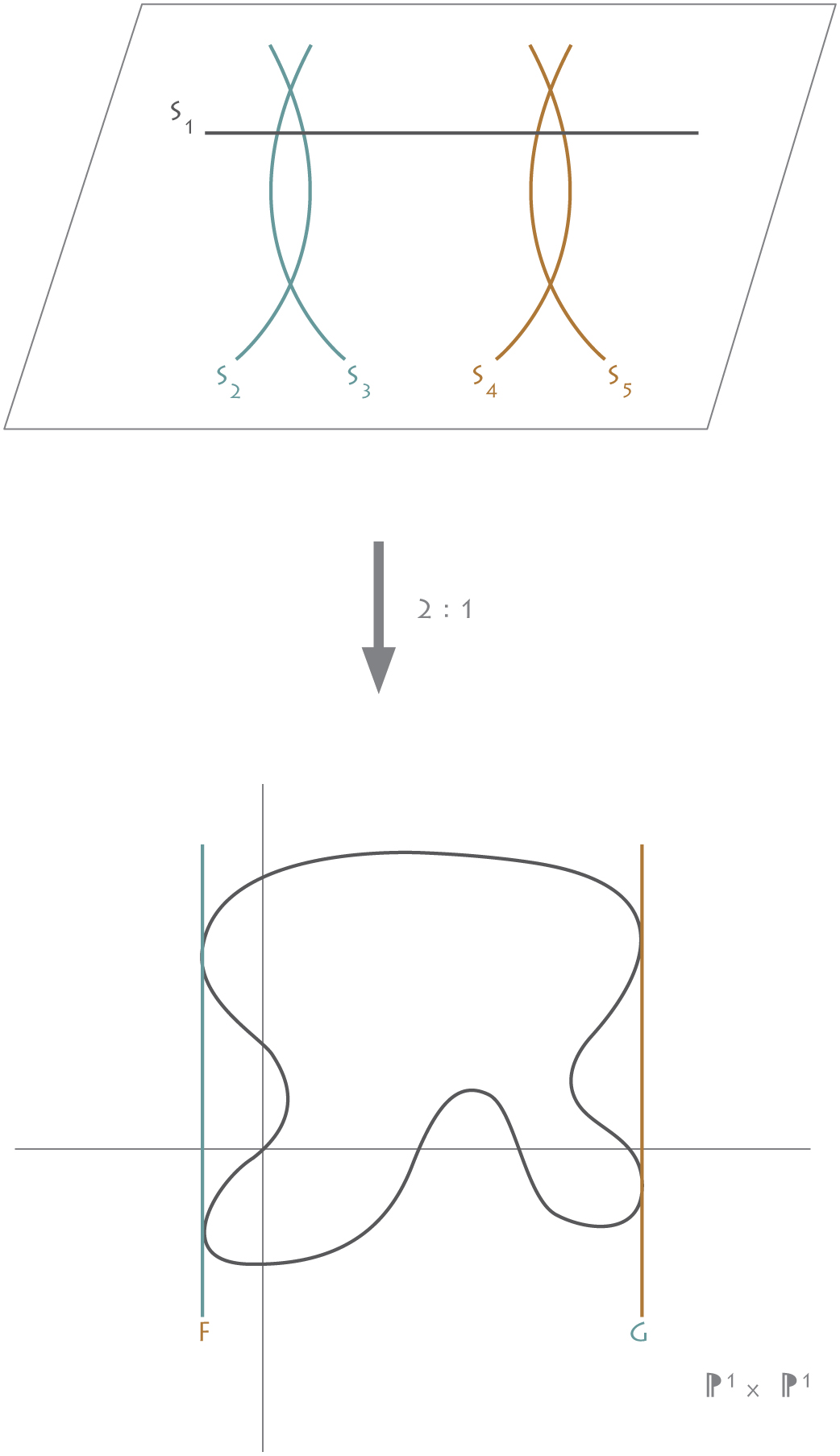}
\end{center}
\caption{Configuration of curves on $Y$.}
\end{figure}
\begin{proposition}\label{p1}
Let $Y$ be as in Proposition \ref{p1xp1}. Then
\begin{enumerate}[i)]
 \item
there exists an automorphism $\s$ of $Y$ such that $Z\assign Y/G$ is $\p1\times\p1$ (where $G=\langle\s|\s^2=1\rangle$) and $\pi:Y\to Z$ is a double cover ramified on a $(4,4)-$divisor $D$;
\item
${H^1(G,\tmop{Pic}Y)\simeq\mb{Z}/2\mb{Z}}$, generated by $L\simeq s_2-s_4$. Then A$\assign \oy\oplus\ls$ is a maximal order on $Z$ ramified on $D.$
\end{enumerate}
\begin{proof}
\begin{enumerate}[i)]
\item
As before, we first define an involution $\phi$ on $\tmop{Pic}Y\oplus T_Y.$ The action of $\phi$ on $\tmop{Pic}Y$ is given by the matrix
 \begin{eqnarray*}
\left( \begin{array}{cccc}
1 & 0 & 0 & 0\\
0 & 0 & 1 & 1\\
0 & 1 & 0 & 1\\
0 & 0 & 0 & -1\\
\end{array} \right)
\end{eqnarray*}
and $\phi(t)=-t$, for all $t\in T_Y.$ Once again, MATLAB verifies that $\phi$ extends to an isometry on $H^2(Y,\mb{Z})$, also denoted $\phi$ (see Appendix \ref{code}, Section \ref{b2}). We now show that $\phi$ is an effective Hodge isometry. Since $H^{2,0}(Y)\subset T_Y$, $\phi$ preserves $H^{2,0}(Y)$. The ample class $s=s_1+s_2+s_3$ (ampleness of $s$ was demonstrated in Propostion \ref{p1xp1}) is preserved by $\phi$ and thus $\phi$ is an effective Hodge isometry. We conclude from the Strong Torelli theorem that $\phi$ is induced by a unique involution $\s\colon Y\to Y$. We denote by $\pi\colon Y\to Z\assign Y/G$ the corresponding quotient morphism, where $G=\langle\s|\s^2=1\rangle$. 
Since $\pi^*\colon \tmop{Pic}Z \to (\tmop{Pic}Y)^G=\langle s_1,s_2+s_3\rangle$ is an injection by Lemma \ref{inj}, $\tmop{rk}\tmop{Pic}Z\leq 2$. The K3 double cover of an Enriques surface has Picard rank $\geq 10$, implying by Proposition \ref{k3sym} that $Z$ is rational and $\pi$ is ramified on the disjoint union of smooth curves. 
\\
 We now show that the ramification locus $D'\subset Y$ is irreducible: firstly, any component of $D'$ is necessarily fixed by $\s$ and thus linearly equivalent to $aS_1+b(S_2+S_3)$, implying its self-intersection is $4ab$. Thus there are no components of $D'$ which are nodal curves.  Moreover, since $\tmop{rk}\tmop{Pic}Z\leq 2$, we conclude from Remark \ref{irred} that $D'$ is irreducible.
Then by Lemma \ref{inj}, $\pi^*\colon \tmop{Pic}Z \to (\tmop{Pic}Y)^G\simeq\mb{Z}s_1\oplus\mb{Z}(s_2+s_3)$ is an isomorphism, implying $\tmop{Pic}Z\simeq \langle t_1,t_2\rangle$ (where $\pi^*(t_1)=s_1, \pi^*(t_2)=s_2+s_3$). By Corollary \ref{intform}, $\tmop{Pic}Z$ has intersection form given by
$$\mb{H} = 
\left( \begin{array}{cc}
0 & 1 \\
1 & 0 
\end{array} \right).
$$
Since the rational surfaces with Picard rank 2 are precisely the Hirzebruch surfaces $\mb{F}_n,$  by Remark \ref{hirz}, $Z\simeq\mb{F}_{2n}$, $n\geq0 $. Assuming $Z\simeq\mb{F}_{2n}, n>0$, there exists an effective divisor $C$ such that $C^2=-2n$. We now show that this is impossible: such a $C$ is necessarily linearly equivalent $at_1+bt_2$, where $ab=-2n$. Thus there exists an effective divisor $C$ linearly equivalent to $aS_1+b(S_2+S_3)$. Since $|s_1|$ defines a base-point free pencil on $Y$, $C\cdot S_1\geq 0$. However, $C\cdot S_1=2b$ and we conclude that $b\geq 0$. Moreover, since $(s_2+s_3)^2=0$, $|s_2+s_3|$ is a pencil and any fixed component is either $S_2$ or $S_3$, which isn't possible since then $S_2$ or $S_3$ would give a pencil of curves with no fixed component by \cite[discussion (2.7)]{projk31}. Thus $|s_2+s_3|$ defines a base-point free pencil and $C\cdot(S_2+S_3)\geq 0$, implying that $a\geq 0$. Thus $ab\geq 0\neq-2n$, yielding a contradiction.
Thus $n=0$ and $Z\simeq\mb{F}_0$, that is, isomorphic to $\p1\times\p1$. Use of the Hurwitz formula once again demonstrates that $\pi$ is ramified on a $(4,4)$-divisor $D$.
\item 
The kernel of $1+\s$ is generated by $s_2-s_3$ and $s_2-s_4$ while the image of $1-\s$ is generated by $s_2-s_3$ and $2(s_2-s_4)$.
Thus ${H^1(G,\tmop{Pic}Y)\simeq\mb{Z}/2\mb{Z}}$, generated by $L=s_2-s_4$. By Proposition \ref{ol}, the nontrivial relation satisfies overlap and by Remark \ref{inbr}, $A\assign\oy\oplus L_\s$ is a nontrivial order  on $Z$ ramified on $D$ and thus maximal.
\end{enumerate}
\end{proof}
\end{proposition}
\begin{remark}\label{bit}
$S_2+\phi(S_2),S_4+\phi(S_4)$ are inverse images of fibres of the same projection $p_2\colon Z\to \p1$. Thus there are two fibres $F,G$ of $p_2$ such that $\pi^{-1}(F)=S_2+\phi(S_2)$ and $\pi^{-1}(G)=S_4+\phi(S_4)$ and $A\simeq\oy\oplus\oy(S_2-S_4)_\s$. Thus $F$ and $G$ are both bitangent to the ramification curve $D$.
 \end{remark}

\begin{figure}[H]
\begin{center}
\includegraphics*{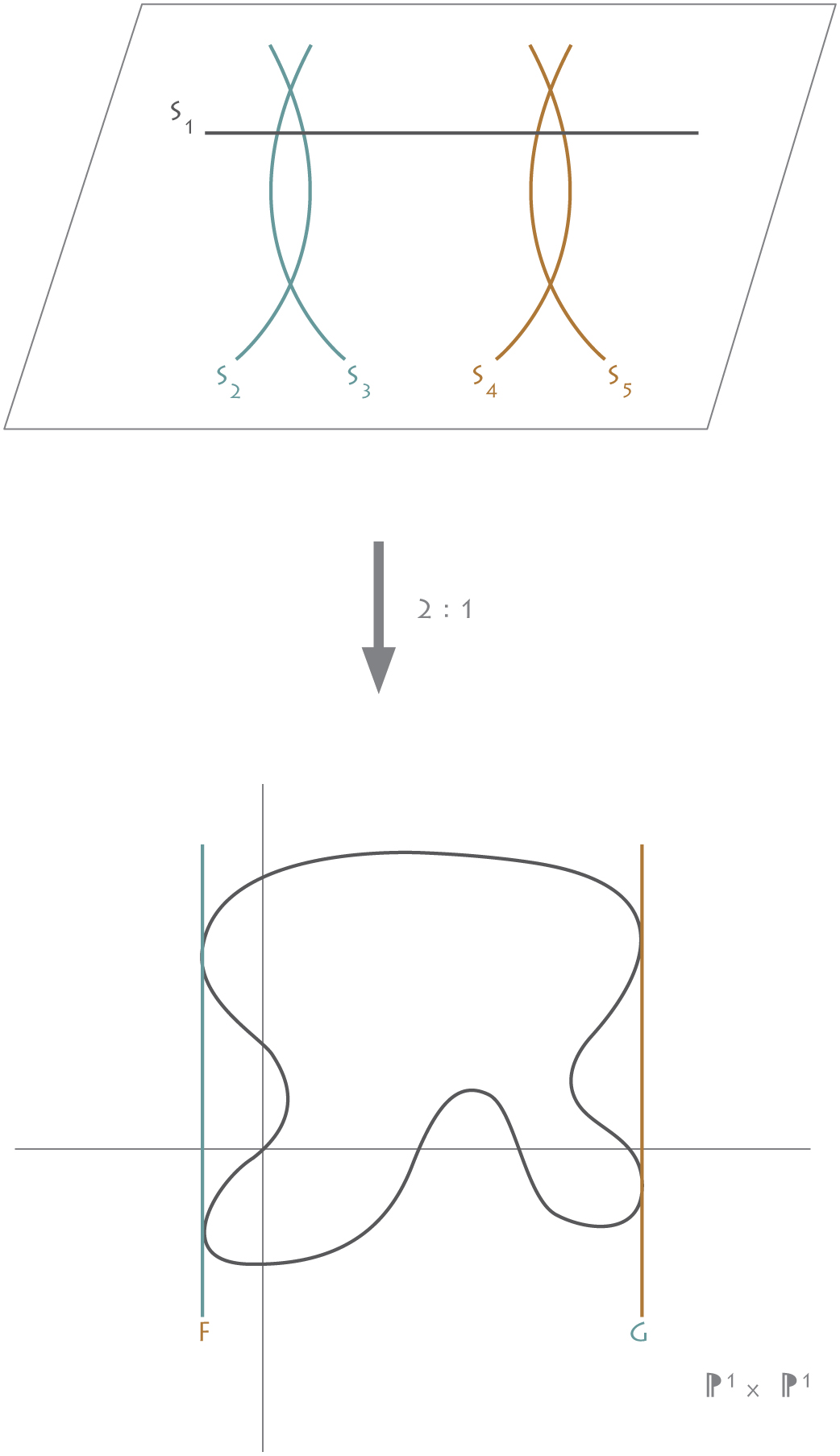}
\end{center}
\caption{The double cover $\pi\colon Y\to\p1\times\p1$.}
\end{figure}
Since the centre of $A$ is the quadric $\p1\times\p1$ and hence birational to $\p2$, it makes sense to ask whether $A$ is itself birational to any of the orders on $\p2$ constructed above.
One may expect $A$ to be birational to an order on $\p2$ ramified on a sextic. This, however, is not the case. In fact, it is birational to an order $A'$ ramified on a singular octic, as we discover below. First let us define what it means for two orders to be birational.

\begin{definition}\index{order!birational equivalence of orders}
 Let $A$ be an order on a scheme $Z$, $A'$ an order on a scheme $Z'$. We say that $A$ and $A'$ are {\bf birational}\footnote{This definition will suffice for our purposes. More generally, one may wish to define two such orders to be birational if $f_{|U}^*(A')\sim_M A_{|U}$.}
 if there exists a birational map $f\colon Z\dashrightarrow Z'$ and an open set $U\subset Z$ on which $f$ is regular such that there is an isomorphism $f_{|U}^*(A')\simeq A_{|U}$.
\end{definition}

Now let $A$ be the order on $Z:=\p1\times\p1$ ramified on a $(4,4)-$divisor $D$ constructed above in Proposition \ref{p1}. We now show that $A$ is birational to an order $A'$ on $Z'\simeq\p2$ ramified on an octic with 2 quadruple points. There exists a birational map $\mu\colon Z \dra Z'$ which is obtained by blowing up a point $p$ on $Z$ (where $p\not\in D, F$ or $G$ ($F,G$ as defined in Remark \ref{bit})) and then blowing down the horizontal and vertical fibres through that point. For any curve $C\subset Z$, we denote its strict transform on $Z'$ by  $C'$. 

The strict transform $D'\subset Z'$ of $D$ is an octic with two 4-fold points $q_1,q_2$ and there is a double cover $\pi'\colon Y'\to Z'$ ramified on $D'$ and a commutative diagram
\begin{displaymath}
 \xymatrix{
Y \ar@{-->}[r]^{\mu_Y} \ar[d]_{\pi} & Y' \ar[d]^{\pi'}\\
Z \ar@{-->}[r]^\mu        & Z'}
\end{displaymath}
%

%

\begin{proposition}
 There exists an order $A'$ on $Z'$ ramified on $D'$ which is birational to $A$. Moreover, $A'$ is constructible using the noncommutative cyclic covering trick.
\begin{proof}
 Recall from Remark \ref{bit} that $D$ has two bitangents $F,G$. Then $F',G'$ are both bitangent to $D'$ and thus each splits into two components on $Y'$, say $F_1',F_2'$ and $G_1',G_2'$. Then $A'=\os_{Y'}\oplus\os_{Y'}(F_1'-G_1')$ is a nontrivial order ramified on $D'$. Letting $U\subset Z$ be an open subset of $Z$ such that $\mu_{|U}$ is an isomorphism, we see that $A_{|U}\simeq\mu_{|U}^*(A')$.
\end{proof}

\end{proposition}

\subsection{Orders on $\mb{F}_2$}

Using the same procedure, we now demonstrate ways to construct numerically Calabi-Yau orders on the 2nd Hirzebruch surface $\mathbb{F}_2\assign\mb{P}_{\mb{P}^1}(\os_{\mb{P}^1}\oplus\os_{\mb{P}^1}(-2))$.

\begin{proposition}\label{f2yo}
 Let $S=\langle s_1,s_2,s_3,s_4,s_5\rangle$ with bilinear form given by 
\begin{eqnarray*}
A=\left( \begin{array}{ccccc}
-2 & 0 & 1 & 0 & 1\\
0 & -2 & 0 & 1 & 0\\
1 & 0 & -2 & 2 & 0\\
0 & 1 & 2 & -2 & 0\\
1 & 0 & 0 & 0 & -2
\end{array} \right).
\end{eqnarray*}
Then there exists a K3 surface $Y$ such that $\tmop{Pic}Y\simeq S.$
\end{proposition}

\begin{proof}
We embed $S$ in $\Ll$ via
 \begin{equation*}
\begin{aligned}
\gamma\colon s_1 & \mapsto\la4, & s_2 & \mapsto\la2+\mu_1, &s_3& \mapsto \la1+2\mu_1,\\
 s_4& \mapsto\la7+\mu_2, & s_5 & \mapsto \la5.
\end{aligned}
\end{equation*}
Since $\{\gamma(s_1),\ldots,\gamma(s_5),\la3,\la6,\la7,\la8,\la1',\ldots,\la8',\mu_1,\mu_1',\mu_2',\mu_1'',\mu_2''\}$ is a basis for $\Ll$, $S\hookrightarrow\Ll$ is a primitive sublattice by Remark \ref{prim}. Since $S$ also has signature $(1,4)$ (this calculation was performed by Maple), there exists a K3 surface $Y$ such that $\tmop{Pic}Y\simeq S$ by Proposition \ref{mor}. By Proposition \ref{rr}, we may assume that $s_i$ are effective classes, as is $s_6\simeq s_3+s_4-s_5$. A simple calculation verifies that $s=s_1+s_2+3s_3+3s_4$ satisfies the conditions of Lemma \ref{ample}, implying $s$ ample and since $s\cdot s_i=1$ for all $i$, each $s_i$ is an irreducible class and hence an effective nodal class.
\end{proof}
\begin{figure}[H]
\begin{center}
\includegraphics*[trim= 0 280 0 0]{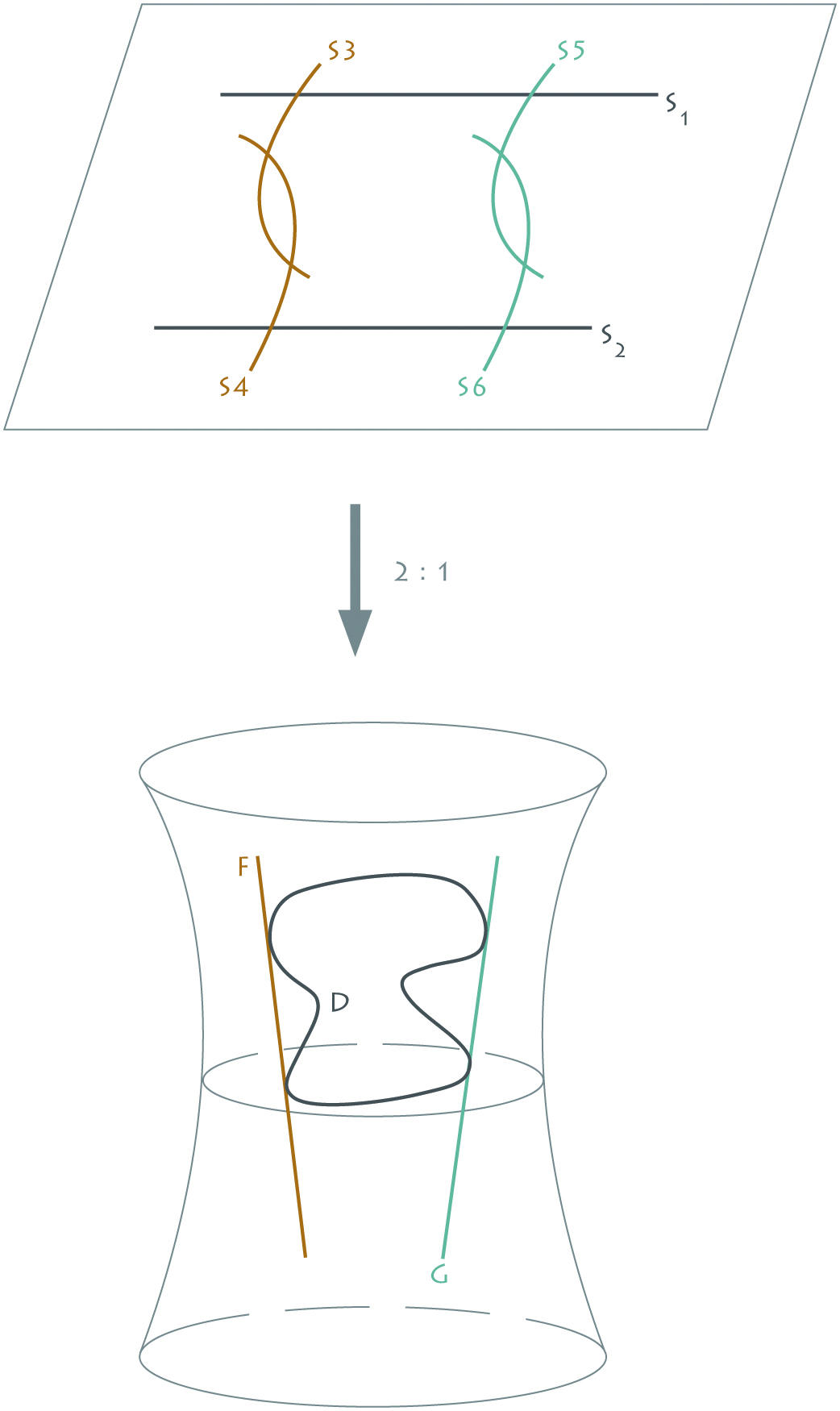}
\end{center}
\caption{Configuration of curves on $Y$.}
\end{figure}
\begin{proposition}\label{f2}
\begin{enumerate}[i)]
Let $Y$ be as in Proposition \ref{f2yo}. Then 
 \item 
there exists an automorphism $\s$ of $Y$ such that $Z\assign Y/G$ is $\mb{F}_2$ (where $G=\langle\s|\s^2=1\rangle$) and $\pi\colon Y\to Z$ is a double cover ramified on a divisor ${D\sim 4C_0+8F }$;
\item
${H^1(G,\tmop{Pic}Y)\simeq\mb{Z}/2\mb{Z}}$, generated by $L=s_3-s_5$ and $A=\oy\oplus\ls$ is a maximal order on $Z$ ramified on $D.$
\end{enumerate}

\begin{proof}
\begin{enumerate}[i)]
\item
As before, we first define an involution $\phi$ on $\tmop{Pic}Y\oplus T_Y.$ The action of $\phi$ on $\tmop{Pic}Y$ is given by the matrix
 \begin{eqnarray*}
\left( \begin{array}{ccccc}
0 & 1 & 0 & 0 & 0\\
1 & 0 & 0 & 0 & 0\\
0 & 0 & 0 & 1 & 1\\
0 & 0 & 1 & 0 & 1\\
0 & 0 & 0 & 0 & -1
\end{array} \right)
\end{eqnarray*}
and $\phi(t)=-t$, for all $t\in T_Y.$ Once again, MATLAB verifies that $\phi$ extends to an isometry on $H^2(Y,\mb{Z})$, also denoted $\phi$ (see Appendix \ref{code}, Section \ref{b3}). We now show that $\phi$ is an effective Hodge isometry. Since $H^{2,0}(Y)\subset T_Y$, $\phi$ preserves $H^{2,0}(Y)$. Moreover, since $\phi$ preserves the ample class $s=s_1+s_2+3(s_3+s_4)$ (demonstrated to be ample in Proposition \ref{f2yo}), $\phi$ is an effective Hodge isometry. \\
Since $\pi^*\colon \tmop{Pic}Z \to (\tmop{Pic}Y)^G=\langle s_1+s_2,s_3+s_4\rangle$ is an injection by Lemma \ref{inj}, $\tmop{rank}\tmop{Pic}Z\leq 2$. The K3 double cover of an Enriques surface has Picard rank $\geq 10$, implying by Proposition \ref{k3sym} that $Z$ is rational and $\pi$ is ramified on the disjoint union of smooth curves. 
We now show that the ramification $D'$ locus is irreducible: firstly, any component of $D'$ is necessarily fixed by $\s$ and thus linearly equivalent to $a(S_1+S_2)+b(S_3+S_4)$, implying its self-intersection is $4a(b-a)$. Thus there are no components of $D'$ which are nodal curves. Moreover, since $\pi^*\colon \tmop{Pic}Z \to (\tmop{Pic}Y)^G$ is an injection by Lemma \ref{inj}, $\tmop{rank}\tmop{Pic}Z\leq 2$. We conclude from Remark \ref{irred} that $D'$ is irreducible.

Thus by Corollary \ref{intform}, $\tmop{Pic}Z\simeq(\tmop{Pic}Y)^G\simeq\langle s_1+s_2,s_3+s_4\rangle$ with intersection product given by
$$
\left( \begin{array}{cc}
-2 & 1 \\
1 & 0 
\end{array} \right).
$$
Thus $Z$ is a Hirzebruch surface $\mathbb{F}_{2n}$, $n\geq 0$. Letting ${T_1=\pi(S_1)}$, $T_1$ is an irreducible divisor such that $T_1^2=-2$, implying by Lemma \ref{negdiv} that  $Z\simeq\mb{F}_2$. Since $Y$ is a K3 surface, $\pi$ is ramified on a divisor $D\sim -2K_{\mb{F}_2}$ and thus $D\sim 4C_0+8F$. 
\item
We now compute $H^1(G,\tmop{Pic}Y)$: the kernel of $1+\s$ is generated by ${s_1-s_2},{s_3-s_4}$ and $s_3-s_5$ while the image of $1-\s$ is generated by $s_1-s_2,s_3-s_4$ and $2(s_3-s_5)$. Thus $H^1(G,\tmop{Pic}Y)\simeq\mb{Z}/2\mb{Z}$, generated by $s_3-s_5$.
The construction of the maximal order follows as in previous examples: $A=\oy\oplus\ls$, where $L$ is a representative of $s_3-s_5$.
\end{enumerate}
\begin{remark}
$S_3+\phi(S_3),S_5+\phi(S_5)$ are inverse images of fibres of the projection $p\colon\mb{F}_2\to \p1$. Thus there are two fibres $F,G$ of $p_2$ such that $\pi^{-1}(F)=S_3+\phi(S_3)$ and $\pi^{-1}(G)=S_5+\phi(S_5)$ and $A\simeq\oy\oplus\oy(S_3-S_5)_\s$. Thus $F$ and $G$ are both bitangent to the ramification curve $D$.
 \end{remark}

\begin{figure}[H]
\begin{center}
\includegraphics*[trim= 0 0 0 0]{thesis06.jpg}
\end{center}
\caption{The double cover $\pi\colon Y\to\mathbb{F}_2$.}
\end{figure}

\end{proof}
\end{proposition}

\subsection{When the surface is ruled over an elliptic curve}\label{ellruled}
Up until now, our use of the Torelli theorem has required that we work over $\mb{C}$. This is no longer required and as such, for the remainder of this chapter, the base field is an arbitrary algebraically closed field $k$ of characterstic $0$.
We now include a brief section on the construction of numerically Calabi-Yau orders on surfaces ruled over elliptic curves.
These orders are of particular interest since there are such limited possibilities for ramification data. Indeed, for this reason Chan and Kulkarni remark that they "are inclined to think that these orders are somehow special"\cite{ncy}.

Let $A$ be a numerically Calabi-Yau order on a surface $Z$ which is ruled over an elliptic curve:
\begin{eqnarray*}
p: Z & \lrw C.
\end{eqnarray*}
We know that $Z=\mb{P}_C(E),$ the projectivisation of a rank $2$ vector bundle $E$ on $C$. Since $Z$ is the centre of a numerically Calabi-Yau order, the number of permissible $E$'s is finite and easy to enumerate. In \cite{ncy}, we discover the following: 
\begin{enumerate}
\item if $E$ is split, then $E=\os_C\oplus N,$ where $N$ is $a-$torsion for $a\in\{1,2,3,4\}.$ 
\item If $E$ is not split, it is indecomposable of degree one, that is, the non-split extension of a degree one line bundle $L$ by $\os_C:$
\begin{eqnarray*}
0\lrw \os_C\lrw E\lrw L \lrw 0.
\end{eqnarray*}
\end{enumerate}

\subsubsection{Case 1: $E$ splits}
We would like to describe the possible ramification of such orders and for this reason we first wish to describe the Picard group of $Z$: recall from Proposition \ref{ruledpic} that $\tmop{Pic}Z = p^*\tmop{Pic}C\oplus\mb{Z}C_0$ and $C_0^2=\tmop{deg}(E),F^2=0,C_0\cdot F=1$ and $K_Z\equiv-2C_0+\tmop{deg}(E)F$.

\begin{remark}
 In our case, $E=\os_C\oplus N,$ where $N$ is $a-$torsion, implying ${\tmop{deg}(E)=0}, C_0^2=0$ and $K_Z\sim-2C_0$.
\end{remark}

%

We see in \cite[discussion following Lemma 2.4]{ncy} that we can record ramification data using ramification vectors: $(e_1,e_2,\ldots)$ where the $e_i$ are repeated with multiplicity. In the case of an order on a ruled surface, given ramification curves $D_i$, the multiplicity is of the form $D_i\cdot F,$ where $F$ is a fibre of the ruling. In the following, by {\bf $n-$section}\index{$n-$section} we shall mean an effective divisor $D$ on $Z$ (not necessarily irreducible) which does not contain fibres of $p$ as components and such that $D\cdot F=n$ for all fibres $F$ of $p:Z\lrw C.$
From \cite[Prop. 5.4]{ncy}, the possible ramification vectors in the $E$ is split case are $(2,2,2,2),(3,3,3),(2,4,4) \ \tmop{and} \ (2,3,6).$ We shall be dealing with the first case, that of $(2,2,2,2),$ until mentioned otherwise.

We wish to construct examples of orders ramified on $4-$sections of $p,$ which are disjoint unions of elliptic curves by \cite[Prop. 4.2]{ncy}. To this end, we let $D$ be such a $4-$section. In order to use Chan's noncommutative cyclic covering trick to construct such an order, we would like to have a double cover of $Z$ ramified on $D$:
\begin{proposition}\label{cover}
Let $D$ be a $4-$section of $p.$ Then there exists a surface $Y$ and a $2:1$ cover $\pi:Y\lrw Z$ ramified on $D.$
\begin{proof}
From the proof \cite[Prop. 4.3]{ncy}, we see that $D$ is numerically equivalent to $4C_0.$ This means that $D\sim 4C_0+p^*(L'),$ where $L'\in\tmop{Pic}^0C.$ Since the $2-$multiplication map $[2]:C\lrw C$ is surjective, any element of $\tmop{Pic}^0C$ is $2-$divisible. Hence $D\in\tmop{Pic}Z$ is $2-$divisible and by the cyclic covering trick, there exists a surface $Y$ and a $2:1$ cover $\pi:Y\lrw Z$ ramified on $D$.
\end{proof}
\end{proposition}

Letting $G=\langle\s|\s^2=1\rangle$ be the Galois group of $\pi,$ the following proposition tells us some of the line bundles $L\in\tmop{Pic}Y$ which satisfy $L_{\s}^{\otimes 2}\simeq\oy$ and thus give us noncommutative cyclic covers. We shall require the following lemma:

%
%


%
%
\begin{proposition}\label{tors}
Let $M\in (\tmop{Pic}C)_2$. Then $L\assign (p\pi)^*M\in\tmop{ker}(1+\s).$ Moreover, unless $D$ is irreducible such the the corresponding cover $D\to C$ has Galois group $\mb{Z}/2\mb{Z}\oplus\mb{Z}/2\mb{Z}$, there exists an $M\in (\tmop{Pic}C)_2$ such that ${A_M=\oy\oplus\ls}$ is a maximal numerically Calabi-Yau order on $Z$ with ramification vector $(2,2,2,2)$.
\begin{proof}
First we note that $(p\pi)^*\tmop{Pic}C$ is $G-$invariant since any element of $\pi^*\tmop{Pic}Z$ is $G-$invariant. Thus $L+\s(L)=2L\simeq\oy\in\tmop{Pic}Y$ and $L\in\tmop{ker}(1+\s).$ By Proposition \ref{ol}, all relations arising from elements of $H^1(G,\tmop{Pic}Y)$ satisfy the overlap condition. We now show that there exists at least one 2-torsion line bundle $M$ such that $A_M=\oy\oplus L_{\s}$ is maximal. We do so by showing that $A_M$ has nontrivial ramification. Let $D_i$ be the curves where $\pi$ (and thus $A_M$) is ramified, that is, the $D_i$ are the components of $D$. Since each $D_i$ is an $n$-section for some $n$, we have a corresponding finite map $\pi_i:D_i\to C.$ Theorem \ref{ram} states that the ramification along $\pi(D_i)$ is isomorphic to the cover given by the $2-$torsion line bundle $L_{|D_i},$ which is exactly $\pi_i^*(M).$ 
Thus $A_M$ has nontrivial ramification on all $D_i$ unless $\pi_i^*(M)\simeq\os_{D_i}$ for some ramification curve $D_i$.

If $D_i$ is a section or a trisection, then all nontrivial $2-$torsion line bundles on $C$ pull back nontrivially. If $D$ is the disjoint union of two bisections $D_1$ and $D_2$ and for each $i$, $M_i$ the unique nontrivial line bundle on $C$ such that $\pi_i^*(M_i)\simeq \os_{D_i},$ then $\pi_i^*(M_1\otimes M_2)$ is nontrivial if $D_1$ and $D_2$ are nonisomorphic. If $D_1\simeq D_2,$ then $\pi_i^*(M')$ is nontrivial for all $M'\not\simeq M_1.$ In the case that $D=D_1$ is an irreducible $4-$section, $\pi_1:D\to C$ is cyclic, corresponding to a $4$-torsion $N\in\tmop{Pic}C$.
%
Then any $M'\in(\tmop{Pic}C)_2$ such that $M'$ is not a power of $N$ pulls back nontrivially to $D$. By Lemma \ref{max}, $A_M$ is maximal and we are done.
\end{proof}
\end{proposition}

%
The same trick can be performed for the other cases enumerated above: $(3,3,3),(2,4,4),(2,3,6)$. We give the flavour of this by doing so for $(3,3,3)$ and $(2,3,6)$, beginning with the former:

\begin{proposition}
 Let $T$ be a trisection of $p\colon Z\to C$. Then there exists a triple cover $\pi\colon Y\to Z$ totally ramified on $T$. Moreover, there exists an $M\in(\tmop{Pic}C)_3$ such that $A_M\assign\oy\oplus L_\s\oplus L_\s^2$ is a maximal numerically Calabi-Yau order on $Z$ ramified on $T$, where $L=(p\pi)^*(M)$.
\begin{proof}
From the proof of \cite[Prop. 4.3]{cyc}, $T\simeq 3C_0+p^*(L')$, for $L'\in\tmop{Pic}^0C$. Since the 3-multiplication map $[3]\colon C\to C$ is surjective, any element of $\tmop{Pic}^0C$ is 3-divisible and thus there exists a triple cover of $\pi\colon Y\to Z$ totally ramified on $T$ and unramified elsewhere. The covering group is $G=\langle\s|\s^3=1\rangle.$ Now for any $M\in(\tmop{Pic}Y)_3$, $L\assign(p\pi)^*M\in\tmop{ker}(1+\s+\s^2)$ and by Proposition \ref{ol}, all relations arising from elements of $H^1(G,\tmop{Pic}Y)$ satisfy the overlap condition. We now show that for any $T$, there exists at least one $M\in(\tmop{Pic}C)_3$ such that $A_M\assign\oy\oplus L_\s$ is a maximal numerically Calabi-Yau order. We do this by showing $A_M$ has nontrivial ramification. From Theorem \ref{ram}, the ramification of $A_M$ along each component $T_i$ of $T$ is $p_i^*(M)$, where $p_i=p_{|T_i}$. We know that $(\tmop{Pic}C)_3\simeq(\mb{Z}/3\mb{Z})^2$, generated by $N_1,N_2$. If $T=T_1$ is irreducible, then $p_1^*(N_j)\sim 0$ for only one of the $N_i$ and if $T$ is reducible, then $N_j\not\in\tmop{ker}p_i^*$, for $j\in\{1,2\}$. Thus by Lemma \ref{max} there always exists an $M\in(\tmop{Pic}C)_3$ such that $A_M$ is maximal.
\end{proof}

\end{proposition}

If $A$ is a numerically Calabi-Yau order with ramification vector $(2,3,6)$, then $Z\simeq C\times\p1$ by \cite[Prop. 4.3]{ncy}. We now construct such an order.

\begin{proposition}
 Let  $Z\simeq C\times\p1$ and $p_i\in\p1, i\in\{1,2,3\}$ be three points on the line. Then there exists a $6\colon 1$ cyclic cover $\pi\colon Y\to Z$ ramified with index 2 over $Z_{p_1}$ index $3$ over $Z_{p_2}$ and index 6 over $Z_{p_3}$. Then $H^1(G,\tmop{Pic}Y)\simeq (\mb{Z}/6\mb{Z})^2$, and for any $L\in\hh$ which is neither 2-torsion nor 3-torsion, $A\assign\oy\oplus \ls\oplus\ldots\oplus \ls^5$ is a maximal numerically Calabi-Yau order with ramification vector $(2,3,6)$.
\begin{proof}
Letting $E$ be the elliptic curve with $j(E)=0$, there exists a $6:1$ cyclic cover $\pi_E\colon E\to\p1$ ramified at $p_1$ with index 2, $p_2$ with index 3 and $p_3$ with index 6 (see the statement and proof of \cite[Chap. III, theorem 10.1]{silverman} for details). We form the fibred product
 \begin{displaymath}
\xymatrix{
Y \ar[r] \ar[d] & C\times\p1 \ar[d]\\
E \ar[r]^{\pi_E} &        \p1
}
\end{displaymath}
Then $Y=C\times E$ and $H^1(G,\tmop{Pic}Y)\simeq H^1(G,\tmop{Pic}C)\oplus H^1(G,\tmop{Pic}E)$.
Since all elements of $\tmop{Pic}C$ are fixed by $G$, $H^1(G,\tmop{Pic}C)\simeq(\tmop{Pic}C)_6\simeq(\mb{Z}/6\mb{Z})^2$. We now show that $H^1(G,\tmop{Pic}E)\simeq 0$: since $H^1(G,\tmop{Pic}E)\simeq\tmop{ker}(1+\s+\ldots+\s^5)/\tmop{im}(1-\s)$ and $\tmop{ker}(1+\s+\ldots+\s^5)\subset\tmop{Pic}^0E$, it suffices to show that $\tmop{im}(1-\s)=\tmop{Pic}^0E$. In fact, we now show that $(1-\s)\colon\tmop{Pic}^1E\to\tmop{Pic}^0E$ is surjective. The group homomorphism $1-\s\colon\tmop{Pic}^1E\to\tmop{Pic}^0E$ corresponds to a scheme morphism on the relevant components of the Picard scheme of $E$,  $\mathbf{Pic}^1E$ and $\mathbf{Pic}^0E$, which are both isomorphic to $E$. Since $\pi$ is totally ramified at $p_3\in\p1$, $q=\pi^{-1}_E(p_3)$ is fixed by $\s$ and $(1-\s)(q)=e_0$, the zero point of $E$. Letting $q'\in E$ be any point not fixed by $\s$, $(1-\s)(q')\neq e_0$. Then $(1-\s)\colon\mathbf{Pic}^1E\to\mathbf{Pic}^0E$ is non-constant implying it is surjective  by \cite[Chap. II, Prop. 6.8]{ag}.
%
%
%
Thus $H^1(G,\tmop{Pic}E)\simeq 0$ and it follows that $H^1(G,\tmop{Pic}Y)\simeq(\tmop{Pic}C)_6\simeq(\mb{Z}/6\mb{Z})^2$.
 By Proposition \ref{ol}, all relations satisfy the overlap condition. Letting $L=p_1^*(M)$, for $M\in(\tmop{Pic}C)_6$, the ramification vector of $A\assign\oy\oplus \ls\oplus\ldots\oplus \ls^5$ is $(2,3,6)$. By Lemma \ref{untot}, the ramfication is given by $M^3$ over $Z_{p_1}$, $M^2$ over $Z_{p_2}$, and $M$ over $Z_{p_3}$.
By Lemma \ref{max}, $A$ is maximal precisely when $M$ is neither $2$-torsion nor $3$-torsion and the result follows.
\end{proof}
\end{proposition}

We now look at the case when $Z$ arises from an indecomposable vector bundle $E$ on $C$.

\subsubsection{Case 2: $E$ indecomposable}
Recall that in this case $Z=\mb{P}_C(E),$ where $E$ is the non-split extension of a degree one line bundle $L$ by $\os_C:$
\begin{eqnarray*}
0\lrw \os_C\lrw E\lrw L \lrw 0.
\end{eqnarray*}
Here
\begin{eqnarray*}
\tmop{Pic}Z & = & p^*\tmop{Pic}C\oplus\mb{Z}C_0,
\end{eqnarray*}
$C_0^2=1$ and $K_Z\equiv-2C_0+F$ (from Proposition \ref{ruledpic}).

\begin{theorem}[\cite{ncy}, Theorem 5.6]\label{enum}
Let $A$ be a numerically Calabi-Yau order on $Z.$ Then the ramification indices are all $2$ and either
\begin{enumerate}
\item the ramification locus $D=D_1$ is irreducible and $D_1\equiv-2K$, or
\item the ramification locus $D=D_1\cup D_2$ splits such that $D_i\equiv-K.$
\end{enumerate}
\end{theorem}
\begin{remark}
In case 1, the divisor is an irreducible $4-$section and in case 2, the disjoint union of 2 irreducible bisections \cite[proof of Thm 4.5]{ncy}.
\end{remark}
\begin{proposition}
Let $D$ be a divisor of the form specified in Theorem \ref{enum}. Then we can construct a maximal order $A=\oy\oplus\ls$ on $Z$ with nontrivial ramification on each component of $D$ unless $D$ is an irreducible $4-$section such that the covering $D\to C$ is not cyclic. Here $\pi\colon Y\to Z$ is the double cover ramified on $D$ and $L$ is the pullback to $Y$ of a $2$-torsion line bundle on the base curve $C$.
\begin{proof}
Let $D$ be an effective divisor numerically equivalent to $-2K.$ Then there exists a surface $Y$ and a $2:1$ cover $\pi:Y\lrw Z$ ramified on $D$ with Galois group $\{\s|\s^2=1\}$. The proof of this is analogous to that of Proposition \ref{cover}.
As in Proposition \ref{tors} above,
\begin{eqnarray*}
(p\pi)^*(\tmop{Pic}C)_2 & \subset & \tmop{ker}(1+\s)
\end{eqnarray*}
and all relations satisfy the overlap condition by Proposition \ref{ol}. As in Proposition \ref{tors} letting $p_i:D_i\lrw C$ denote the projections of each component of $D,$ the ramification of $A=\oy\oplus((p\pi)^*M)_{\s}$ over $D_i$ is given by $p_i^*(M),$ for $M\in(\tmop{Pic}C)_2.$ Now we need to verify that there exists a $2-$torsion line bundle $M$ such that $p_i^*(M)$ is non-trivial in $\tmop{Pic}D_i$, for all $i$. We demonstrate this for each of the possible cases. If $D=D_1$ is irreducible, then  $p_1:D_1\lrw C$ is a $4:1$ cyclic cover. Thus the kernel of
\begin{eqnarray*}
p_1^*: \tmop{Pic}C & \lrw \tmop{Pic}D_1
\end{eqnarray*}
is generated by some $4-$torsion $M_1\in\tmop{Pic}C.$ Then $p_1^*$ trivialises only one of the three nontrivial $2-$torsion line bundles on $C.$ If, on the other hand, $D=D_1\cup D_2$ is the disjoint union of two bisections, then from the proof of \cite[Theorem 4.5]{ncy}, $D_1$ and $D_2$ are two non-isomorphic double covers of $C.$ Letting $L_i\in\tmop{Pic}C$ be the line bundle trivialised by $p_i^*\colon\tmop{Pic}C\to\tmop{Pic}D_i,$ $p_i^*(L_1+L_2)$ is nontrivial in $\tmop{Pic}D_i$ for each $i$, implying maximality by Lemma \ref{max}.
\end{proof}
\end{proposition}

%% file: c3.tex
 \chapter{Constructing orders on elliptically fibred surfaces}\label{os}

\section{Elliptic fibrations on surfaces}
All the constructions of orders so far have been on surfaces of Kodaira dimension $-\infty.$ Here we are interested in constructing examples on surfaces of higher Kodaira dimension. There is a rich theory concerning surfaces of Kodaira dimension zero, many of which are equipped with elliptic fibrations while all surfaces of Kodaira dimension $1$ are equipped with such a fibration.
%
The following definition is from  \cite[Chap. V, Section 6]{bpv}.

\begin{definition}\index{elliptic fibration}\label{el}
An {\bf elliptic fibration} on a surface $Z$ is a smooth curve $C$ and a flat, projective morphism of schemes ${\phi:Z\lrw C}$ such that the general fibre $Z_c$ is a smooth curve of genus one over $K(c)$, for $c\in C$. A {\bf section} of $\phi$ is a morphism $s_0:C\lrw Z$ such that $\phi\circ s_0=\tmop{id}_C.$ If $\phi$ has a section, we call it a {\bf Jacobian elliptic fibration}\index{elliptic fibration!Jacobian elliptic fibration}.
\end{definition}
\begin{remark}
Specifying a section is equivalent to specifying an irreducible curve $S_0$ on $Z$ such that $S_0\cdot Z_c=1,$ for all closed $c\in C.$ The correspondence is as follows: the curve corresponding to $s_0\colon C\to Z$ is $S_0=\tmop{im}(s_0)$ while given such a curve $S_0$, the corresponding morphism $s_0\colon C\to Z$ is given by $s_0(c)=S_0\cap Z_c$. Such a curve we shall also refer to as a section of $\phi:Z\to C.$ We make this explicit since when dealing with sections it will suit our purposes to think of them as morphisms in some cases and curves in others.
\end{remark}
All but a finite number of the fibres of an elliptic fibration will be smooth curves of genus 1 and Kodaira classified all possible fibres in \cite{fibres}. For any given elliptic fibration, a fibre $Z_c$ is one of the following:
\begin{enumerate}
\item Irreducible, in which case it is (a) smooth of genus 1, (b) rational with a node or (c) rational with a cusp;
\item Reducible but not multiple. In this case $Z_c=\sum n_i C_i,$ where the $C_i$ are nodal curves, that is, rational $(-2)-$curves and gcd$(n_i)=1$;
\item A multiple fibre. These are classified (see \cite[Chap. V, Section 7]{bpv}, for example).
\end{enumerate} 
Any fibre which is not a smooth curve of genus 1 is called a {\bf degenerate fibre}.
%

If $\phi$ is a Jacobian elliptic fibration the sections of $\phi$ form a group, denoted $\Phi.$ We see this by noting that specifying a section is equivalent to specifying a point on the genus one curve $Z_{\eta},$ the generic fibre of $\phi,$ which is a group. Alternatively, we can view the group law as follows: given two sections $s_1,s_2\in\tmop{Hom}_C(C,Z)$, and $c\in C$ such that $Z_c$ is nondegenerate, $(s_1+s_2)\colon c\mapsto s_1(c)\oplus s_2(c)$, where $\oplus$ is the addition on the group scheme $Z_c$ (see \cite[10.3]{algeom2} for details); this defines a morphism $(s_1+s_2)\colon C-\{c_1,\ldots,c_n\}\to Z$, where $Z_{c_i}$ are the degenerate fibres of $\phi$. Since any rational map from a smooth curve is regular \cite[Chap. II, Prop. 2.1]{silverman}, this defines a section $(s_1+s_2)\colon C\to Z$. The group $\Phi$ is known as the {\bf group of sections}\index{elliptic fibration!group of sections} of $\phi$.

%
%

{\bf From here on in, unless stated otherwise, by elliptic fibration we shall mean a smooth elliptic fibration, that is, one without degenerate fibres.} In Section \ref{ogg}, we shall deal with cases in which there are degenerate fibres.

Since the structure of the Picard group of a surface is essential to our constructing noncommutative cyclic covers, we wish to describe the Picard group of a surface possessing an elliptic fibration. To do so, we first need to define the relative Picard functor (see \cite{picard} for details):
\begin{definition}
Let $f:X\lrw S$ be a separated map of finite type. The {\bf relative Picard functor}\index{Picard functor} $\tmop{Pic}_{X/S}$ is defined by
\begin{eqnarray*}
\tmop{Pic}_{X/S}(T) & = & \tmop{Pic}(X_T)/f_T^*\tmop{Pic}(T),
\end{eqnarray*}
for any $S-$scheme $T$ where $X_T\assign X\times_ST$\index{$X_T$} and $f_T\colon X_T\to T$\index{$f_T$}.
\end{definition}
The Picard functor is representable in many cases (for further details, see \cite{picard} and \cite{neron}) and the following theorem gives us representability in the case in which we are currently interested, that of an elliptic fibration. 

\begin{theorem}\cite[Theorem 4.8]{picard}
 Assume $f\colon X\to S$ is projective Zariski locally over $S$ and is flat with integral geometric fibres. Then $\tmop{Pic}_{X/S}$ is representable by a scheme $\mathbf{Pic}_{X/S},$\index{Picard scheme} the relative Picard scheme of $X/S$.
\end{theorem}

\begin{remark}\label{jac}
$\mathbf{Pic}_{X/S}$ parametrises pairs $(s,M_s)$, where $s\in S$ and $M_s\in\tmop{Pic}X_s$. In the case of elliptic fibrations on surfaces, where the fibres of $f\colon X\to S$ are curves, $\mathbf{Pic}_{X/S}=\coprod_{i\in\mb{Z}}\mathbf{Pic}^i_{X/S},$ where $\mathbf{Pic}^i_{X/S}$ is the connected component of $\mathbf{Pic}_{X/S}$ parametrising pairs $(s,M_s)$ such that $\tmop{deg}(M_s)=i$. Then $\mathbf{Pic}^0_{X/S}$ is the connected component of the identity.
\end{remark}
\begin{example}\label{J}
Let $\phi:Z\lrw C$ be an elliptic fibration. Then the morphism $\mathbf{Pic}^0_{Z/C}\to C$ is also elliptically fibred, possesses a section and has the same fibres as $\phi\colon Z\to C$ over closed points $c\in C$ (for details, see\cite[Prop. 4.2.2]{cald}). 
\end{example}

\begin{definition}\label{jaco}
Setting $J\assign\mathbf{Pic}^0_{Z/C}$, we call $\phi_J:J\lrw C$ the {\bf (relative) Jacobian fibration} of $\phi:Z\lrw C.$
\end{definition}

\begin{remark}\label{sect}
If $\phi$ itself possesses a section, then $Z$ and $J$ are isomorphic as elliptic fibrations over $C$. This means there exists an isomorphism $f:Z\to J$ such that
\begin{displaymath}
 \xymatrix{
Z \ar[dr]_\phi \ar[r]^f & J \ar[d]^{\phi_J}\\
                        & C}
\end{displaymath}
commutes.
\end{remark}
Now we can describe the Picard group of a surface $Z$ equipped with a Jacobian elliptic fibration:
\begin{proposition}\label{fibadd}
Let $\phi:Z\lrw C$ be a Jacobian elliptic fibration. Then
\begin{eqnarray*}
\tmop{Pic}Z & \simeq & \tmop{Pic}C\oplus\tmop{Hom}_C(C,\mathbf{Pic}_{Z/C}).
\end{eqnarray*}
\begin{proof}
We first note that $\tmop{Pic}Z/\phi^*\tmop{Pic}C\simeq \tmop{Pic}_{Z/C}(C)$. Since $\tmop{Pic}_{Z/C}$ is representable, $\tmop{Pic}Z/\phi^*\tmop{Pic}C\simeq\tmop{Hom}_C(C,\mathbf{Pic}_{Z/C})$, implying that
\begin{eqnarray}\label{split}
0\lrw\tmop{Pic}C\longrightarrowlim^{\phi^*}\tmop{Pic}Z\longrightarrowlim\tmop{Hom}_C(C,\mathbf{Pic}_{Z/C})\lrw 0
\end{eqnarray}
is exact.
Since $\phi\colon Z\lrw C$ has a section $s_0\colon C\lrw Z,$  $\phi^*\colon\tmop{Pic}C\lrw\tmop{Pic}Z$ also has a section $s_0^*,$ implying the above exact sequence splits.
\end{proof}
\end{proposition}

Note that implicit in the splitting $\tmop{Pic}Z \simeq \tmop{Pic}C\oplus\tmop{Hom}_C(C,\mathbf{Pic}_{Z/C})$ is the choice of a section $s_0$ of $\phi$.

%
%
\begin{remark}\label{section}
We wish to determine the embedding 
\begin{eqnarray*}
 \psi\colon\tmop{Hom}_C(C,\mathbf{Pic}_{Z/C})& \hookrightarrow & \tmop{Pic}Z
\end{eqnarray*}
which splits the exact sequence (\ref{split}) above. To this end, we let ${f\in\tmop{Hom}_C(C,\mathbf{Pic}_{Z/C})}$. Since $\mathbf{Pic}_{Z/C}$ parametrises pairs $(c,M_c)$, where $c\in C $ and $M_c\in\tmop{Pic}Z_c$, $f(c)$ will correspond to a line bundle on $Z_c$, for all $c\in C$. We shall abuse notation by writing $f(c)$ for this line bundle. Then $\psi(f)$ is the unique line bundle $L$ on $Z$ such that $L_{|Z_c}\simeq f(c)$, for all $c\in C$, and $L_{|S_0}\simeq \os_{S_0}$.
\end{remark}
%
\begin{lemma}\label{sectionf}
 There is an embedding 
\begin{eqnarray*}\label{sectiona}
 \alpha\colon\Phi & \longrightarrow & \tmop{Pic}Z
\end{eqnarray*}
given by $\alpha(S_0)=\os_Z$ and $\alpha(S)=\os_Z(S-S_0-D_S)$,  when $S\neq S_0$. Here ${D_S= \sum_{c\in C}a_cZ_c}$\index{$D_S$} such that $S\cap S_0=\sum_{c\in C}a_cZ_{c|S_0}$.

\begin{proof}
 Firstly, $\Phi$ is isomorphic to $\tmop{Hom}_C(C,J)$, the isomorphism given by
\begin{eqnarray*}
 S & \mapsto & (c\mapsto (c,\os_{Z_c}(S_{|Z_c}-S_{0|Z_c})))
\end{eqnarray*}
Composing this with $\psi$ of Remark \ref{section} yields the desired group homomorphism $\alpha\colon\Phi \lrw \tmop{Pic}Z$. It is clear that $\alpha(S_0)=\os_Z$. When $S\neq S_0$, we set $D_S= \sum_{c\in C}a_cZ_c$ such that $S\cap S_0=\sum_{c\in C}a_c Z_{c|S_0}$. Then for all $c\in C$, $\os_Z(S-S_0-D_S)_{|Z_c}=\os_{Z_c} (S_{|Z_c}-S_{0|Z_c})$ and $\os_Z(S-S_0-D_S)_{|S_0}=\os_{S_0}$, implying $\alpha(S)=\os_Z(S-S_0-D_S)$.

\end{proof}
\end{lemma}

\begin{remark}\label{section2}
 For any section $S$ of $\phi\colon Z\to C$, we denote its image in $\tmop{Pic}Z$ as $L_S$. Similarly, if we consider the section as a morphism $s\colon C\to Z$, we denote its image by $L_s\in\tmop{Pic}Z$.
\end{remark}

\begin{corollary}\index{elliptic fibration!trivial elliptic fibration}\label{corcor}
 Let $E$ be an elliptic curve with zero $e_0\in E$, $C$ an arbitrary curve and $Z=E\times C$. The fibration $p_2:Z\lrw C$ is called a {\bf trivial elliptic fibration} and $\Phi  \simeq  \tmop{Hom}(C,E)$. For any $\varphi\colon C\to E$,
\begin{eqnarray*}
L_{\varphi} & = & \os_Z(\Gamma_{\varphi}-(\{e_0\}\times C+\sum_{c_i\in\varphi^{-1}(e_0)}E\times\{c_i\})),
\end{eqnarray*}
unless $\varphi(c)=e_0$ is the constant morphism, in which case $L_{\varphi}\simeq \os_Z$.
\begin{proof}
From Lemma \ref{sectionf}, $\Phi\simeq\tmop{Hom}_C(C,J)\simeq\tmop{Hom}_C(C,C\times E)$,
implying ${\Phi\simeq\tmop{Hom}(C,E)}$. The zero section $S_0$ is $\{e_0\}\times C$ and if $\varphi$ is not the constant morphism $\varphi(c)=e_0$ for all $c\in C$, $\Gamma_\phi\cap S_0=\phi^{-1}(e_0)$. Thus $D_{\Gamma_\phi}=\sum_{c_i\in\varphi^{-1}(e_0)}E\times\{c_i\})$ and the result follows.
\end{proof}

\end{corollary}

\begin{example}\label{trivfib}

Let $E,C,Z$ be as above. If there are no nonconstant morphisms $\psi:C\lrw E,$ then $\tmop{Hom}(C,\mathbf{Pic}^0E)\simeq\mathbf{Pic}^0E$ which is isomorphic, as a group, to $E.$ Thus
\begin{eqnarray*}
\Phi & \simeq & \{\{e\}\times C|e\in E\}
\end{eqnarray*}
with $(\{e_1\}\times C)\oplus(\{e_2\}\times C)=\{e_1\oplus_E e_2\}\times C.$ For any section $S\in\Phi$, ${L_S=\os_Z(S-\{e_0\}\times C)}$.

Now we wish to consider cases where there exists nonconstant morphisms $\psi:C\lrw E$.
For example, let $Z:=E\times E,$ where $E$ is an elliptic curve without complex multiplication. 
By \cite[Chap. III, Remark 4.3]{silverman}, 
\begin{eqnarray*}
 \tmop{Hom}(E,E)& \simeq & \tmop{Pic}^0E\oplus\mb{Z}\tmop{id}.
\end{eqnarray*}
%
%
Thus, by Corollary \ref{corcor},
\begin{eqnarray*}
\Phi & \simeq & \tmop{Pic}^0E\oplus\mb{Z}S,
\end{eqnarray*}
where $S=\Gamma_{id}-(\{e_0\}\times E)$ and $L_{S}=\os_Z(\Gamma_{id}-(\{e_0\}\times E)- (E\times\{e_0\}))$.
Now let $Z=E\times E$ where $j(E)=1728$, that is, $E$ possesses a complex multiplication $\tau$ such that $\tau^2=[-1]_E$. By \cite[Example 4.4 and Theorem 10.1]{silverman}, 
\begin{eqnarray*}
 \tmop{Hom}(E,E)& \simeq & \tmop{Pic}^0E\oplus\mb{Z}\tmop{id}\oplus\mb{Z}\tau.
\end{eqnarray*}
Then, by Corollary \ref{corcor},
\begin{eqnarray*}
\Phi & = & \tmop{Pic}^0E\oplus\mb{Z}S\oplus\mb{Z}T
\end{eqnarray*}
where $T=\Gamma_{id}-(\{e_0\}\times E)$ and $L_T=\os_Z(\Gamma_{\tau}-(E\times\{e_0\}))$.
\end{example}
 
\begin{defex}\label{fibre}\index{elliptic fibration!fibre bundles} {\bf Fibre bundles}:
Let $E$ be an elliptic curve, $C$ an arbitrary curve of genus $\geq 1$, $\tmop{Aut}E$ the group of automorphisms of $E$ (as an abelian variety). Let $\psi:C'\lrw C$ be an \'etale covering with Galois group $G$ isomorphic to some finite subgroup of $\tmop{Aut}E$ and $f\colon G\lrw\tmop{Aut}E$ the corresponding injection. We define an action of $G$ on $E\times C'$ via
\begin{eqnarray*}
g: E\times C' & \lrw & E\times C'\\
     (e,c')     & \longmapsto & (f(g)(e),g(c'))
\end{eqnarray*}
and form the corresponding quotient $Z=(E\times C')/G$ which is elliptically fibred over $C.$ We call $\varphi \colon Z\lrw C$ a {\bf fibre bundle}. Note that a trivial elliptic fibration is a special case of this construction.
\end{defex}
The importance of fibre bundles is demonstrated by the following proposition:
\begin{proposition}\cite{algsurf}\label{bund}
Let $\varphi\colon Z\to C$ be a Jacobian elliptic fibration without degenerate fibres. Then $\phi$ is a fibre bundle.
\end{proposition}
\begin{remark}\label{fibgrp}
Let our notation remain as in Example/Definition \ref{fibre}. Any section of a fibre bundle $\phi\colon Z\to C$ is necessarily the image under $\pi\colon E\times C'\to Z$ of a $G$-invariant section $S'$ of $\psi\colon E\times C' \to C'$. To see this, we let $S\subset Z$ be such a section, $F$ a fibre of $\phi$. Since $S\cdot F=1$, $\pi^*S\cdot \pi^*F=n$, where $n=|G|$. However, $\pi^*F$ is the union of $n$ fibres of $\psi\colon E\times C' \to C'$, from which we see $\pi^*S$, which is $G$-invariant, is a section of $\psi$ as claimed. Conversely, given a $G$-invariant section $S'\subset E\times C'$, $ \pi(S')$ is a section of $\phi\colon Z\to C$. Thus the group of sections $\Phi_Z$ of $\varphi:Z\to C$, is isomorphic to $E^G$, the group of fixed points of $E$ under the action of $\tmop{im}G\subset\tmop{Aut}(E)$. For example, if $G$ is cyclic of order 2 and acts on $E$ by sending $e$ to $-e$, then $\Phi_Z\simeq(\mb{Z}/2\mb{Z})^2$.
\end{remark}

\section{The Tate-Shafarevich group and Ogg-Shafarevich theory}

The following is classical: given a Jacobian elliptic fibration ${\phi:Z\to C}$ on a surface $Z$, the Brauer group of $Z$, $\tmop{Br}(Z)$ (see Appendix \ref{br} for details concerning the Brauer group of a scheme) is isomorphic to the Tate-Shafarevich group $\Sha_C(Z)$ of $\phi$, described below. Since we are interested in explicit constructions of Azumaya algebras on $Z$, it is of interest to find an analogue of the noncommutative cyclic covering trick for elements of $\Sha_C(Z)$. In this section we do exactly this and provide some examples of the analogous construction.

Example \ref{J}, together with Remark \ref{sect}, tell us that given an arbitrary elliptic fibration ${\phi:Z\lrw C,}$ there exists a unique surface $J$ with an elliptic fibration $\phi_J:J\lrw C$ with the same fibres as $\phi$ over the closed points of $C$ such that  $\phi_J$ possesses a section ($\phi_J$ is the Jacobian fibration of $\phi$). Moreover, given a Jacobian elliptic fibration ${\phi:Z\lrw C}$, $\Sha_C(Z)$, which we describe here, is a classifying group for all fibrations with $\phi$ as their Jacobian. Let $\phi:Z\lrw C$ be a Jacobian elliptic fibration. We now define the Tate-Shafarevich group of $\phi$:\index{Tate-Shafarevich group, $\Sha_C(Z)$}\index{$\Sha_C(Z)$}
\begin{eqnarray*}
\Sha_C(Z) & := & H^1_{\acute{et}}(C,Z^\sharp),
\end{eqnarray*}
where $Z^\sharp$ is the sheaf of abelian groups on $C$ defined by
\begin{eqnarray*}
Z^\sharp(U) & = & \{\mbox{the group of sections of } Z_U\lrw U\}
\end{eqnarray*}
for all \'etale $U\lrw C.$ There is a $1-1$ correspondence between elements of $\Sha_C(Z)$ and isomorphism classes of elliptic fibrations $X\lrw C$ with $Z\lrw C$ as their Jacobian (see \cite{dol94} and \cite[Section 4.4]{cald}).
From here on in, $\phi:Z\lrw C$ will always denote a Jacobian elliptic fibration. The following result is well-known:
\begin{theorem}
If $C$ is a smooth curve, then
\begin{eqnarray*}
\tmop{Br}(Z) & \simeq & \Sha_C(Z).
\end{eqnarray*}
\begin{proof}
From \cite[Cor. 1.17]{dol94}, $\Sha_C(Z)\simeq\tmop{coker}(\tmop{Br}(C)\lrw\tmop{Br}(Z)).$ Since $C$ is a smooth curve, $\tmop{Br}(C)\simeq 0$ and the result follows.
\end{proof}
\end{theorem}

 As stated above, due to the isomorphism $\tmop{Br}(Z)\simeq\Sha_C(Z),$ we would hope for a construction for elements of $\Sha_C(Z)$ analogous to the noncommutative cyclic covering trick and we do indeed have one which we describe here now.

\begin{construction}\label{cyctate}\cite{algsurf}
Let $\pi_C:C'\lrw C$ be an $n:1$ cyclic \'etale covering with Galois group $G=\{\s|\s^n=1\}.$ Form the fibred product:
\begin{eqnarray}
  Y & \longrightarrowlim^{\pi} & Z \nonumber\\
  \phi_Y \downarrow & \square & \downarrow \phi \nonumber\\
  C' & \longrightarrowlim^{\pi_C} & C \nonumber
\end{eqnarray}
We denote by $\Phi_Y$ the group of sections of $\phi_Y:Y\lrw C'$ and note that $G$ acts on $Y,$ sending sections to sections, and thus acts on $\Phi_Y$ via $\s^i(s)\colon c'\mapsto \s^{i}(s(\s^{-i}(c')))$. Let $s$ be a representative 1-cocycle of a class in $H^1(G,\Phi_Y)$, that is, a section $s$ such that $s+\s(s)+\ldots+\s^{n-1}(s)=0$ in $\Phi_Y$. We use such an $s$ to define a new $G-$action on $Y$, which we denote $G_s$ (and this is a bonafide action by Lemma \ref{constr} below), via
\begin{eqnarray}
  \sigma_s : Y & \longrightarrow & Y \nonumber\\
  ( c', z ) & \longmapsto & ( \sigma ( c' ), \s(z) \oplus  s  ( \s(c' ) )).
  \nonumber
\end{eqnarray}
 We define a new quotient of $Y$ as $Z_{C',s} \assign Y / G_s$\index{$Z_{C',s}$}. Then $Z_{C', s}\in\Sha_C(Z).$ This construction is drawn from \cite{algsurf}.
\end{construction}

\begin{remark}
 In the following lemma, we also demonstrate that using cohomologous 1-cocycles results in the same quotient surface, implying it is actually an element of $H^1(G,\Phi_Y)$ we are using to twist the $G$-action.
\end{remark}

\begin{lemma}\label{constr}
 This $\s_s\colon Y\to Y$ defines a fixed point free group action of $G_s$ on $Y$ and cohomologous $1\mhyphen$cocycles result in the same action.
 \begin{proof}
Since $\pi_C\colon C'\to C$ is \'etale, $\s_s$ is fixed point free and so for the first part of the lemma, we need to verify is that $\s_s^n$ acts trivially on $Y:$
\begin{eqnarray*}
\s_s^n(c',z) & = & \s_s^{n-1}(\s(c'),\s(z)\oplus s(\s(c')))\\
             & = & (\s^n(c'),\s^n(z)\oplus \s^{n-1}s(\s(c'))\oplus \s^{n-2}s(\s^2(c'))\oplus \ldots\oplus s(c')))\\
	     & = & (c',z\oplus \s^{-1}s(\s(c'))\oplus \s^{-2}s(\s^2(c'))\oplus \ldots\oplus \s(s(\s^{n-1}(c')))\oplus s(c'))\\
	     & = & (c',z),
\end{eqnarray*}
the last equality arising from the fact that $s+\s(s)+\ldots+\s^{n-1}(s)=0$ in $\Phi_Y$ and that $\s^i(s)(c')=\s^{i}(s(\s^{-i}(c')))$.
Now to show cohomologous $1\mhyphen$cocycles result in the same quotient surface, by Remark \ref{sect}, we need only show that if $s$ is a $1$-coboundary then $Y/G_s$ possesses a section. To this  end, we let $s\in\tmop{im}(1-\s)$, i.e. $s=t-\s(t)$, for some section $t\in\Phi$. Then $Y/G_s$ possesses a section if and only if there exists a section of $\Phi_Y$ which is $\s_s-$invariant (this is analogous to Remark \ref{fibgrp}). 
We now show that $t\colon C'\to Y$ is $\s_s-$invariant:

\begin{eqnarray*}
\s_s(c',t(c')) & = & \s_{t-\s(t)}(c',t(c'))\\
           & = & (\s(c'),\s(t(c'))\oplus(t-\s(t))(\s(c')))\\
           & = & (\s(c'),\s(t(c')) \oplus t(\s(c'))\ominus \s(t(\s^{-1}\s(c')))\\
           & = & (\s(c'),t(\s(c'))).
\end{eqnarray*}
Thus we see that the section $t\colon C'\to Y$ is $\s_s-$invariant and $Y/G_s\simeq Z$, as required.
\end{proof}
\end{lemma}

By construction, this element $Z_{C',s}$ of $\Sha_C(Z)$ lies in the kernel of ${\pi_C^*:\Sha_C(Z)\lrw\Sha_{C'}(Y).}$ We would like to know if all elements trivialised by $\pi_C^*$ can be constructed in this manner and the following tells us that they can.
\begin{proposition}\label{ker}
$H^1(G,\Phi_Y)\simeq\tmop{ker}(\Sha_C(Z)\lrw\Sha_{C'}(Y)).$
\begin{proof}
Let $\iota:\eta\hookrightarrow C,\iota':\eta'\hookrightarrow C'$ be the generic points of $C,C'$ respectively. Also recall that $Z_{\eta},Y_{\eta'}$ are group schemes over $K(C),K(C')$ respectively and, following \cite{dol94}, we shall identify $Z_{\eta},Y_{\eta'}$ with the sheaves in the \'etale topology which they respectively represent. Then from \cite[Section 1]{dol94}, we see that ${\Sha_C(Z)\simeq H^1_{\acute{et}}(C,\iota_*Z_{\eta})}$ and ${\Sha_{C'}(Y)=H^1_{\acute{et}}(C',\iota_*Y_{\eta'}).}$ 

We first show that $\iota'_*Y_{\eta'}$ is the pullback of $\iota_*Z_{\eta}$ under $\pi_C\colon C'\to C$.  We form the fibred diagram
\begin{eqnarray}
  \eta' & \longrightarrowlim^{\pi_{\eta}} & \eta \nonumber\\
  \iota' \downarrow & \square & \downarrow \iota \nonumber\\
  C' & \longrightarrowlim^{\pi_C} & C \nonumber
\end{eqnarray}
First note that $Y_{\eta'}=\pi_{\eta}^*Z_\eta$. Since $\iota, \iota'$ are inclusions of points, 
\begin{eqnarray*}
\pi_C^*\iota_*Z_\eta & \simeq & \iota'_*\pi^*_{\eta}Z_\eta\\
                     & \simeq & \iota'_*Y_{\eta}.
\end{eqnarray*}

We now use the Hochschild-Serre spectral sequence:
\begin{eqnarray*}
E^{pq}_2\assign H^p(G,H^q_{\acute{et}}(C',\iota'_*Y_{\eta'})) & \implies & H^{p+q}_{\acute{et}}(C,\iota_*Z_{\eta}).
\end{eqnarray*}
The relevant part of $E^{pq}_2$ is
$$
\begin{array}{lll}
H^0(G,\Sha_{C'}(Y)) & H^1(G,\Sha_{C'}(Y)) &\\
H^0(G,\Phi_Y) & H^1(G,\Phi_Y) & H^2(G,\Phi_Y)
\end{array}
$$
From this, the following is exact
\begin{eqnarray*}
0\lrw H^1(G,\Phi_Y)\lrw H^1_{\acute{et}}(C,\iota_*Z_{\eta})\lrw \Sha_{C'}(Y)^G\lrw H^2(G,\Phi_Y)\lrw\ldots
\end{eqnarray*}
Since $H^1_{\acute{et}}(C,\iota_*Z_{\eta})\simeq\Sha_C(Z)$ and $\Sha_{C'}(Y)^G$ is a subgroup of $\Sha_{C'}(Y),$
\begin{eqnarray*}
 0\lrw H^1(G,\Phi_Y)\lrw \Sha_C(Z)\lrw \Sha_{C'}(Y) 
\end{eqnarray*}
is exact and $H^1(G,\Phi_Y)\simeq\tmop{ker}(\Sha_C(Z)\lrw\Sha_{C'}(Y)).$
\end{proof}
\end{proposition}
Here we provide a simple example of the construction:
\begin{example}\label{cyctriv}
Let $\phi:Z\lrw C$ be a trivial elliptic fibration, as in Example \ref{trivfib}. That is, $Z=E\times C,$ for some elliptic curve $E.$ Assume $g(C)\geq 1.$ Let $\pi_C:C'\lrw C$ be an \'etale $n:1$ cyclic cover such that there are no nonconstant morphisms $C'\to E$ and set ${Y\assign Z\times_CC'}$ as above. Then the following diagram is fibred
\begin{eqnarray}
  Y & \longrightarrowlim^{\pi} & Z \nonumber\\
  \phi_Y \downarrow & \square & \downarrow \phi \nonumber\\
  C' & \longrightarrowlim^{\pi_C} & C \nonumber
\end{eqnarray}
As we saw in Example \ref{trivfib}, the sections of $\phi_Y$ are all of the form $\{e\}\times C',$ for $e\in E$, and $\Phi_Y\simeq\tmop{Pic}^0E.$ Since  $\s(e,c'_1)=(e,\s(c_1')),$ any element of $\Phi_Y$ is fixed by $\s,$ implying  
\begin{eqnarray*}
H^1(G,\Phi_Y) & \simeq & (\Phi_Y)_n\\
             & \simeq &(\tmop{Pic}^0E)_n.
\end{eqnarray*}
 For any $n-$torsion $\varepsilon-e_0\in\tmop{Pic}^0E,$ we define the new $G-$action $G_{\varepsilon}$ by
\begin{eqnarray}
  \sigma_{\varepsilon} : Y & \longrightarrow & Y \nonumber\\
  ( c', e ) & \longmapsto & ( \sigma ( c' ), e \oplus \varepsilon ).
  \nonumber
\end{eqnarray}
Letting $Z':=Y/G_{\varepsilon},$ we see that $Z'$ is an elliptic surface fibred over $C$. As in Remark \ref{fibgrp}, $\Phi_Z\simeq E^G$. Here, however, $E^G$ is empty, implying $p\colon Z'\to C$ is elliptically fibred without a section. It does, however, possess an $n-$section which is the image of $\{e_0\}\times C$ under $\pi'\colon Y\lrw Z'.$
\end{example}

\section{Relationship with the noncommutative cyclic covering trick}

This section is devoted to exploring the relationship between the above construction and Chan's noncommutative cyclic covering trick. The questions we answer are the following: given an element of $H^1(G,\Phi_Y)$, how do we recover an element of $H^1(G,\tmop{Pic}Y)$ and a corresponding relation? Does this relation satisfy the overlap condition? What is the corresponding cyclic cover? We then provide examples of how the constructions and correlations work.

Let the setup be as in Construction \ref{cyctate}, that is, an \'etale cover $\pi_C$ and a fibred diagram
\begin{eqnarray}
  Y & \longrightarrowlim^{\pi} & Z \nonumber\\
  \phi_Y \downarrow & \square & \downarrow \phi \nonumber\\
  C' & \longrightarrowlim^{\pi_C} & C \nonumber
\end{eqnarray}
Recall that we use elements of $Rel_{io}$ to construct noncommutative cyclic covers, where $Rel_{io}$ is the subgroup of invertible relations satisfying the overlap condition.
\begin{lemma}
We have an isomorphism
\begin{eqnarray*}
Rel_{io}/E & \simeq & H^1(G,\Phi_Y).
\end{eqnarray*}
\begin{proof}
Since $\tmop{Br}(Z)\simeq\Sha_C(Z)$ and $\tmop{Br}(Y)\simeq\Sha_{C'}(Y),$
\begin{eqnarray*}
\tmop{Br}(Y/Z) & \simeq & \tmop{ker}(\Sha_C(Z)\lrw\Sha_{C'}(Y)).
\end{eqnarray*} 
However, $\tmop{Br}(Y/Z)$ is also isomorphic to $Rel_{io}/E$ by Proposition \ref{prop} and $\tmop{ker}(\Sha_C(Z)\lrw\Sha_{C'}(Y))$ is isomorphic to $H^1(G,\Phi)$ by Proposition \ref{ker}, proving the lemma.
%
\end{proof}

\end{lemma}

 Since our interest lies in constructing noncommutative cyclic algebras using elements of $Rel_{io}/E,$ we would like to see how this isomorphism occurs constructively, that is, given an element of $H^1(G,\Phi),$ how do we recover the corresponding relation satisfying the overlap condition? The following proposition tells us how to retrieve an element of $H^1(G,\tmop{Pic}Y)$ from one of $H^1(G,\Phi_Y).$

 \begin{proposition}\label{split1}
 If $S\in H^1(G,\Phi_Y),$ then $L_S\in H^1(G,\tmop{Pic}Y),$ where ${L_S=\os_Z(S-S_0-D_S)}$ is as defined in Remark \ref{section2}.
 \begin{proof}
 Recall that $\tmop{Pic}Y \simeq \tmop{Pic}C'\oplus\tmop{Hom}_{C'}(C',\mathbf{Pic}_{Y/C'})$ with $\psi$ the group homomorphism $\tmop{Hom}_{C'}(C',\mathbf{Pic}_{Y/C'})\lrw \tmop{Pic}Y.$ Since this splitting of $\tmop{Pic}Y$ is a splitting of $G-$modules,
%
 \begin{eqnarray*}
H^1(G,\tmop{Pic}Y) & \simeq & H^1(G,\tmop{Pic}C')\oplus H^1(G,\tmop{Hom}_{C'}(C',\mathbf{Pic}_{Y/C'})).
\end{eqnarray*}
By Lemma \ref{split2} below, $H^1(G,\tmop{Hom}_{C'}(C',\mathbf{Pic}_{Y/C'}))\simeq H^1(G,\Phi_Y)$ from which we see that 
\begin{eqnarray*}
H^1(G,\tmop{Pic}Y) & \simeq & H^1(G,\tmop{Pic}C')\oplus H^1(G,\Phi_Y)
\end{eqnarray*}
and we thus have an embedding $H^1(G,\Phi_Y)\longhookrightarrow H^1(G,\tmop{Pic}Y)$ which is given by $S\mapsto L_S$ as defined in Remark \ref{section2}.
\end{proof}
\end{proposition}

\begin{lemma}\label{split2}
There is an isomorphism $H^1(G,\tmop{Hom}_{C'}(C',\mathbf{Pic}_{Y/C'}))\simeq H^1(G,\Phi_Y).$
\begin{proof}
Since ${\Phi_Y\simeq\tmop{Hom}_{C'}(C',\mathbf{Pic}^0_{Y/C'}})$ (as given in the proof of Lemma \ref{sectionf}), we are required to show that $H^1(G,\tmop{Hom}_{C'}(C',\mathbf{Pic}_{Y/C'}))\simeq H^1(G,\tmop{Hom}_{C'}(C',\mathbf{Pic}^0_{Y/C'})).$
Given $\sum_ia_iS_i\in\tmop{Hom}_{C'}(C',\mathbf{Pic}_{Y/C'})$, $\sum_ia_i(S_i-S_0)\in\tmop{Hom}_{C'}(C',\mathbf{Pic}^0_{Y/C'})$. This means that
\begin{eqnarray}\label{okffool}
\tmop{Hom}_{C'}(C',\mathbf{Pic}_{Y/C'}) &\simeq  &\tmop{Hom}_{C'}(C',\mathbf{Pic}^0_{Y/C'})\oplus\mb{Z}S_0.
\end{eqnarray}
Moreover, since 
$S_0$ is the inverse image under $\pi$ of the zero section of $\phi\colon Z\to C$, $S_0$ is $G$-invariant. This, along with the fact that $\tmop{Hom}_{C'}(C',\mathbf{Pic}^0_{Y/C'})$ is $G$-invariant, implies that
(\ref{okffool}) is a splitting of $G$-modules. It follows that
\begin{eqnarray*}
H^1(G,\tmop{Hom}_{C'}(C',\mathbf{Pic}_{Y/C'})) & \simeq & H^1(G,\tmop{Hom}_{C'}(C',\mathbf{Pic}^0_{Y/C'}))\oplus H^1(G,\mb{Z}S_0) \\
                                                & \simeq  & H^1(G,\tmop{Hom}_{C'}(C',\mathbf{Pic}^0_{Y/C'}))
\end{eqnarray*}

%
%
%
 as required.
\end{proof}
\end{lemma}

%
%

Given $S\in H^1(G,\Phi_Y)$ and its corresponding $L_S\in H^1(G,\tmop{Pic}Y),$ we would like to figure out if the corresponding relation $\alpha:(L_S)_{\s}^n\simeq\oy$ satisfies the overlap condition and thus can be used to construct an Azumaya algebra on $Z.$

In order to solve the overlap question, we are accustomed to invoking Proposition \ref{ol} but since in our current case $\pi$ is \'etale, we are not able to do so. Since $H^1(G,\tmop{Pic}Y) \simeq H^1(G,\tmop{Pic}C')\oplus H^1(G,\Phi_Y)$, our first step is to identify $H^1(G,\tmop{Pic}C').$ We do so in the following lemmata and proposition:

\begin{lemma}\label{lem}
Let $g$ be the genus of $C, p_0\in C$ a distinguished point. Then the cyclic cover $\pi_C:C'\lrw C$ can be constructed using the cyclic covering trick, that is, using an $n-$torsion line bundle $M\in\tmop{Pic}{C}.$ Moreover, there exist $p_i\in C$ such that
\begin{eqnarray*}
M & \simeq & \os_C\left(\sum_{i=1}^{r}p_i-rp_0\right), 
\end{eqnarray*}
where $r=2g-1.$
\begin{proof}
Recall the Kummer sequence:
\begin{eqnarray*}
1\lrw\mu_n\lrw\mb{G}_m\longrightarrowlim^{\times n}\mb{G}_m\lrw 1.
\end{eqnarray*}
The associated long exact sequence in \'etale cohomology is
\begin{eqnarray*}
0\lrw H^1(C,\mu_n)\lrw\tmop{Pic}C\longrightarrowlim^{\times n}\tmop{Pic}C\lrw\ldots
\end{eqnarray*}
and we conclude that $H^1(C,\mu_n)\simeq(\tmop{Pic}C)_n,$ and since $H^1(C,\mu_n)$ classifies $n:1$ cyclic \'etale covers of $C,$ any such cover can be constructed from an $n-$torsion line bundle on $C.$ In particular, there exists an $n-$torsion line bundle $M\in\tmop{Pic}C$ such that $\pi_{C*}(\os_{C'})=\os_C\oplus M\oplus\ldots\oplus M^{n-1}.$  We now show that $M$ is of the form specified in the statement of the lemma: firstly, since $M$ is $n-$torsion, it is in $\tmop{Pic}^0C.$ Then from \cite[Ch. III, Sec. 1]{milne-abelian}, for all $s>2g-2,$ ${C^s\twoheadrightarrow\tmop{Pic}^s C}$, where $C^s=C\times\ldots\times C$($s$ copies) and there is a bijective correspondence $\tmop{Pic}^s C \longleftrightarrow \tmop{Pic}^0 C$ given by $D\mapsto D-sp_0.$ Thus there exist $p_i\in C$ such that 
\begin{eqnarray*}
M & \simeq & \os_C\left(\sum_{i=1}^{r}p_i-rp_0\right), p_i\in C,
\end{eqnarray*}
where $r=2g-1.$
\end{proof}
\end{lemma}

\begin{lemma}\label{exact}
There is an exact sequence
\begin{eqnarray}
0\lrw Rel_{io}/E\lrw H^1(G,\tmop{Pic}C') \lrw H^3(G,\os(C')^*)
\end{eqnarray}
 and $H^3(G,\os(C')^*)\simeq\mb{Z}/n\mb{Z}.$
\begin{proof}
From Proposition \ref{prop} there is an exact sequence
\begin{eqnarray}\label{ok}
H^2(G,\os(C')^*)\lrw Rel_{io}/E\lrw H^1(G,\tmop{Pic}C') \lrw H^3(G,\os(C')^*)
\end{eqnarray}
and $G$ acts trivially on $\os(C')^*\simeq k^*.$ Letting $\s$ be a generator for $G$, for any $G$-module $M$, $H^1(G,M)\simeq\frac{\tmop{ker}(1+\s+\ldots\s^{n-1})}{\tmop{im}(1-\s)}$. Since $G$ acts trivially on $k^*$, $H^2(G,\os(C')^*)\simeq 0$ and $H^3(G,\os(C')^*)\simeq \mu_n$.
By (\ref{ok}),
\begin{eqnarray*}
 0\lrw Rel_{io}/E\lrw H^1(G,\tmop{Pic}C') \lrw H^3(G,\os(C')^*)
\end{eqnarray*}
is exact and $H^3(G,\os(C')^*)\simeq\mb{Z}/n\mb{Z}$ which completes the proof.
\end{proof}
\end{lemma}
Our goal is to compute $H^1(G,\tmop{Pic}C')$ and with the aid of the above lemmata we are now able to do so.
\begin{proposition}\label{curve}
 Using the notation of Lemma \ref{lem}, $H^1(G,\tmop{Pic}C')\simeq \mb{Z}/n\mb{Z}$ and is generated by 
\begin{eqnarray*}
N & := & \os_{C'}\left(\sum_{i=1}^rp_i'-rp_0'\right)
\end{eqnarray*}
where $p_i'\in\pi^{-1}(p_i)$.
\begin{proof}
From Lemma \ref{exact} the following is exact
\begin{eqnarray*}
0 \lrw Rel_{io}/E\lrw H^1(G,\tmop{Pic}C')\longrightarrowlim^{d_2} H^3(G,\os(C')^*)\lrw\ldots
\end{eqnarray*}
From Proposition \ref{prop}, we know that $Rel_{io}/E\simeq\tmop{Br}(C'/C),$ which is trivial since Brauer groups of curves are trivial by Tsen's theorem. Thus $d_2$ is injective. We now demonstrate that it is also surjective.
Since $\pi_C\colon C'\to C$ is given by $M\in\tmop{Pic}C$ with $M$ as given in Lemma \ref{lem}, we know that there exists an $f\in K(C)$ such that $K(C')=K(C)[d]/(d^n-f)$ and
\begin{eqnarray*}
(f) & = & n\left(\sum_{i=1}^rp_i-rp_0\right).
\end{eqnarray*}
Then
\begin{eqnarray*}
\pi_C^*((f)) & = & n\left(\sum_{i=1}^rp'_i+\s\sum_{i=1}^rp'_i+\ldots+\s^{n-1}\sum_{i=1}^rp'_i\right)-nr\left(p'_0+\s p'_0+\ldots\s^{n-1}p'_0\right).
\end{eqnarray*}
Thus
\begin{eqnarray*}
(d) & = & \left(\sum_{i=1}^rp'_i+\s\sum_{i=1}^rp'_i+\ldots+\s^{n-1}\sum_{i=1}^rp'_i\right)-r\left(p'_0+\s p'_0+\ldots\s^{n-1}p'_0\right),
\end{eqnarray*}
from which we conclude that $N=\os_{C'}\left(\sum_{i=1}^rp_i'-rp_0'\right)\in H^1(G,\tmop{Pic}C').$ However, for all $i\in\{1,\ldots,n-1\}$,
\begin{eqnarray*}
N^i(\s^*N^i)\ldots\s^{*(n-1)}N^i & = & d^i\os_{C'}
\end{eqnarray*}
and $d^i\not\in K(C)$. This then
implies that $d_2(N^i)\neq 0\in H^3(G,\os(C')^*)$, for all $i\in\{1,\ldots,n-1\}$. Moreover, from Lemma \ref{exact}, $H^3(G,\os(C')^*)\simeq\mb{Z}/n\mb{Z}$.
Thus $d_2:H^1(G,\tmop{Pic}C')\lrw H^3(G,\os(C')^*)$ is in fact surjective and $H^1(G,\tmop{Pic}C')\simeq\mb{Z}/n\mb{Z}$.
\end{proof}
\end{proposition}

We are now able to see how to retrieve relations satisfying the overlap condition from elements of $H^1(G,\Phi).$
\begin{theorem}\label{seclap}
Let $S\in H^1(G,\Phi_Y).$ Then $L_S\in H^1(G,\tmop{Pic}Y)$ and the corresponding relation satisfies the overlap condition. Moreover $=A(Y;(L_S)_\s)$ is an Azumaya algebra on $Z$.
\begin{proof}
The proof of Lemma \ref{exact} holds for the morphism $\pi\colon Y\to Z$, implying that the following sequence is exact:
\begin{eqnarray}\label{done}
0 \lrw Rel_{io}/E\longrightarrowlim^f H^1(G,\tmop{Pic}Y)\longrightarrowlim^{d_2^Y}\mb{Z}/n\mb{Z}.
\end{eqnarray}
We also know that 
\begin{eqnarray*}
 H^1(G,\tmop{Pic}Y) & \simeq & H^1(G,\tmop{Pic}C')\oplus H^1(G,\Phi_Y)
\end{eqnarray*}
and ${H^1(G,\tmop{Pic}C')\simeq\mb{Z}/n\mb{Z}.}$ Then
\begin{eqnarray*}
0\lrw Rel_{io}/E\lrw \mb{Z}/n\mb{Z}\oplus H^1(G,\Phi_Y)\lrw\mb{Z}/n\mb{Z}
\end{eqnarray*}
is exact. Corresponding to $s_0:C'\lrw Y$ we have a commutative diagram:
\begin{eqnarray*}
H^1(G,\tmop{Pic}Y) & \longrightarrowlim^{s_0^*} & H^1(G,\tmop{Pic}C')\\
d_2^Y\downarrow      &      &  \downarrow d_2\\
H^3(G,\os(Y)^*)    & \longrightarrowlim^{\beta} & H^3(G,\os(C')^*)
\end{eqnarray*}
From the proof of Proposition \ref{curve}, $d_2$ is surjective. Since $s_0^*$ also surjects and the above diagram commutes, $\beta\circ d_2^Y$ surjects, from which we conclude $\beta:H^3(G,\os(Y)^*)\to H^3(G,\os(C')^*)$ does also. However, by Lemma \ref{exact}, both $H^3(G,\os(Y)^*)$ and $H^3(G,\os(C')^*)$ are isomorphic to $\mb{Z}/n\mb{Z}$ implying $\beta$ is an isomorphism.
Now $s_0^*(L_S)\simeq \os_{C'}$ and as such is trivial in $H^1(G,\tmop{Pic}C')$, implying that $d_2s_0^*(L_S)=0$ and thus $\beta d_2^Y(L_S)=0$. Since $\beta$ is an isomorphism, $L_S\in\tmop{ker}d_2^Y$ and from this, along with exactness of (\ref{done}), we deduce that $L_S$ is in the image of $f\colon Rel_{io}/E \to H^1(G,\tmop{Pic}Y)$ . Thus the corresponding relation satisfies the overlap condition and, by Theorem \ref{normet}, the corresponding cyclic algebra is Azumaya. Since $Rel_{io}/E\simeq\tmop{Br}(Y/Z)$ by Proposition \ref{prop}, $A$ is nontrivial.
\end{proof}
\end{theorem}

\begin{remark}
 In the following examples, we identify $\Phi_Y$ with its image in $\mathbf{Pic}^0_{Z/C}$. That is, we think of the section $s\colon C\to Z$ as the  divisor $S-S_0$, where $S=\tmop{im}(s)$.
\end{remark}

\begin{example}{\bf (Constructing Azumaya algebras on trivial elliptic fibrations)}\label{fib}
 Let the setup be exactly the same as in Example \ref{cyctriv}. That is, $Z=E\times C,$ $E$ elliptic, $g(C)\geq 1, \pi_C:C'\lrw C$ \'etale and $Y$ given by the following fibred diagram:
\begin{eqnarray}
  Y & \longrightarrowlim^{\pi} & Z \nonumber\\
  \phi_Y \downarrow & \square & \downarrow \phi \nonumber\\
  C' & \longrightarrow & C \nonumber
\end{eqnarray}
Let $L\in H^1(G,\Phi_Y)$ be the $n-$torsion section $p_1^*(e_1-e_0).$ By Proposition \ref{split1}, the corresponding element of $H^1(G,\tmop{Pic}Y)$ is $L$ itself and the corresponding relation $\alpha:L_{\s}^{\otimes n}\simeq\oy$ satisfies overlap by Theorem \ref{seclap}. 
Thus ${A:=\oy\oplus\ls\oplus\ldots\oplus L_{\s}^{n-1}}$ is an Azumaya algebra over $Z.$ This is a classical example of Azumaya algebra, usually constructed in the fashion presented here now.
\end{example}

\begin{construction}
 Let $Z= E\times C$ be the fibred product of two irrational smooth curves, $L\in(\tmop{Pic}E)_n, M \in(\tmop{Pic}C)_n$ and $p_1\colon Z\to E, p_2\colon Z\to C$ the two projections. We define the box product $L^i\boxtimes M^j=p_1^*L^i\otimes p_2^*M^j$ and set
\begin{eqnarray*}
 A' & \assign & \bigoplus_{i,j=0}^{n-1}L^i\boxtimes M^j
\end{eqnarray*}
Given sections $l\in p_1^*L, m\in p_2^*M$ we set $(m\otimes 1)(1\otimes l) \assign \zeta(l\otimes m)$ for some primitive nth root of unity $\zeta$. 
\end{construction}

Let us now look more generally at constructing Azumaya algebras on Jacobian elliptic fibrations without degenerate fibres, that is, on fibre bundles.

\begin{example}
 Let $Z$ be a fibre bundle as in Example/Definition \ref{fibre} such that ${G=\langle\s|\s^2=1\rangle}$. That is $Z=(E\times C')/G$ where $G$ acts on $C'$ without fixed points and ${\s(e)=-e}$. Then $\phi\colon Z\to C\assign C'/G$ is an elliptic fibration with group of sections $\Phi_Z$ isomorphic to $E^G\simeq(\mb{Z}/2\mb{Z})^2$ by Remark \ref{fibgrp}.
 Now take a double cover $\pi_C:D\lrw C$ and once again form the fibred diagram 
\begin{eqnarray}
  Y & \longrightarrowlim^{\pi} & Z \nonumber\\
  \phi_Y \downarrow & \square & \downarrow \phi \nonumber\\
  D & \longrightarrowlim^{\pi_C} & C \nonumber
\end{eqnarray}
Here $G$ acts on $Y$ via its Galois action on $D$. We require that $D$ is not birational to a curve $D'$ in $Z$ such that $\pi_C=\phi_{|D'}$, that is, there are no curves $D'\subset Z$ such that $D\times_C D'$ splits into the union of sections of $\phi_Y$.
Then $\Phi_Y\simeq\Phi_Z\simeq(\mb{Z}/2\mb{Z})^2$ is fixed by $G,$ implying $H^1(G,\Phi)\simeq(\Phi_Y)_2\simeq(\mb{Z}/2\mb{Z})^2.$ The construction of noncommutative cyclic covers here is analogous to those in Example \ref{fib}.
\end{example}
The condition that there are no curves $D\subset Z$ such that $D\times_{C'}D'$ splits is integral to the constructions above. We demonstrate this by showing that if $D\simeq C'$, then $H^1(G,\Phi)$ is trivial.
\begin{example}
As above, let $Y=E\times C'$ and $\s\colon Y\to Y$ such that $\s(e,c')=(-e,\s(c'))$ and $C'\to C:=C'/G$ is Galois \'etale, where $G=\langle\s|\s^2=1\rangle$. Then letting  $Z:=Y/G$, $Z\to C$ is a fibre bundle. Now if $\Phi_Y\simeq\tmop{Pic}^0E,$ (that is, if all morphisms $C'\to E$ are constant), then
%
%
%
\begin{eqnarray*}
 \s\colon \Phi_Y & \lrw & \Phi_Y\\
    e-e_0& \mapsto & e_0-e
\end{eqnarray*}
and thus
\begin{equation}
\begin{aligned}
1+\s\colon  \Phi_Y & \lrw \Phi_Y, & 1-\s\colon  \Phi_Y & \lrw \Phi_Y\\
e  -e_0     & \mapsto 0 &  e  -e_0        & \mapsto 2(e-e_0).
\end{aligned}
\end{equation}
Since $[2]\colon E\to E$ is surjective, $\tmop{im}(1-\s)=\Phi_Y$ and $H^1(G,\Phi_Y)$ is trivial. In this case we cannot construct any nontrivial Azumaya algebras using the noncommutative cyclic covering trick.
\end{example}
The story is different, however, when there is a nonconstant morphism $\psi\colon C'\to E$, as we now see:

\begin{proposition}
Let $Y,Z$ be as above. Assume there exists a nonconstant morphism $\psi:C'\to E$ such that $\Phi_Y\simeq\tmop{Pic}^0E\oplus\mb{Z}(\Gamma_{\psi}-(\{e_0\}\times C'))$ and $\psi$ factors through $\pi_C:C'\to C$:
\begin{displaymath}
\xymatrix{
C' \ar[dr]_\psi \ar[r]^{\pi_C} & C \ar[d]\\
                               &  E}
\end{displaymath}
 Then $H^1(G,\Phi_Y)\simeq\mb{Z}/2\mb{Z},$ the only nontrivial class being $S=\Gamma_{\psi}-(\{e_0\}\times C')$.  Moreover,
\begin{eqnarray*}
 L_S & = & \oy(\Gamma_{\psi}-(\{e_0\}\times C'+\sum_{c_i\in\psi^{-1}(e_0)}E\times \{c_i\})).
\end{eqnarray*}
Then $A_S=\oy\oplus (L_S)_\s$ is a cyclic algebra satisfying the overlap condition and thus an Azumaya algebra.
\begin{proof}
We know that $\Phi_Y\simeq\tmop{Pic}^0E\oplus\mb{Z}S,$ where $S=\Gamma_\psi-(\{e_0\}\times C')$ and $\s$ acts on $\Phi_Y$ by sending $(\{e\}\times C')-(\{e_0\}\times C')\in\Phi_Y$ to $(\{e_0\}\times C')-(\{e\}\times C')$ and $\Gamma_{\psi}-(\{e_0\}\times C')$ to $\Gamma_{-\psi}-(\{e_0\}\times C'),$ where $-\psi=[-1]_E\circ\psi.$ The former is obvious while the latter is due the fact that $\psi(c')=\psi(\s(c')).$ Looking fibrewise, it is clear that $\Gamma_{\psi}-(\{e_0\}\times C')+\Gamma_{-\psi}-(\{e_0\}\times C')=0\in\Phi_Y.$ Thus
\begin{equation}
\begin{aligned}
1+\s\colon  \Phi_Y & \lrw \Phi_Y, & 1-\s\colon  \Phi_Y & \lrw \Phi_Y\\
(e-e_0,S)        & \mapsto (0,0)&  (e-e_0,S)          & \mapsto (2(e-e_0), 2S)
\end{aligned}
\end{equation}
Since the multiplication by $2$ map on $\tmop{Pic}^0E$ is surjective, $\mbox{im}(1-\s)=\tmop{Pic}^0E\oplus 2S$ and since $S\in\tmop{ker}(1+\s),$ $H^1(G,\tmop{Pic}Y)\simeq \mb{Z}/2\mb{Z}.$ By Lemma \ref{section2}, 
\begin{eqnarray*}
 L_S & = & \oy(\Gamma_{\psi}-(\{e_0\}\times C'+\sum_{c_i\in\psi^{-1}(e_0)}E\times \{c_i\})),
\end{eqnarray*}
and by Theorem \ref{seclap}, the corresponding relation satisfies the overlap condition.
\end{proof}
\end{proposition}


\subsection{Constructing Azumaya algebras on non-Jacobian elliptic fibrations}
Ogg-Shafarevich theory provides a recipe for constructing Azumaya algebras on surfaces possessing Jacobian elliptic fibrations. We now construct some Azumaya algebras on surfaces which are elliptically fibred over curves, the fibrations of which have no sections. Recall the first examples of non-Jacobian elliptic fibrations presented in Example \ref{cyctriv}: $Y=E\times C'$, where $E$ is elliptic, $C'$ an \'etale $n\colon 1$ cyclic cover of a curve $C$ with Galois group $G=\langle\s|\s^n=1\rangle$. Letting $\varepsilon$ be an $n-$torsion point on $E$, we extend the action of $\s$ on $C'$ to $Y$ via
\begin{eqnarray*}
 \s\colon Y & \lrw & Y\\
    (e,c')  & \mapsto & (e\oplus\varepsilon,\s(c')).
\end{eqnarray*}
By setting $Z\assign Y/G$, we retrieve a non-Jacobian elliptic fibration $\phi\colon Z\to C$.
\begin{proposition}
 $H^1(G,\tmop{Pic}Y)\simeq(\mb{Z}/n\mb{Z})^2$, generated by bundles $L,M$ such that, for $i\in\{1,\ldots,n-1\}$,
\begin{eqnarray*}
A_i & \assign & \oy \oplus (L^i\otimes M^{-i})_\s\oplus \ldots \oplus (L^i\otimes M^{-i})_\s^{n-1}
\end{eqnarray*}
 are nontrivial non-Morita equivalent Azumaya algebras on $Z$. 
\begin{proof}
Once again, the splitting $\tmop{Pic}Y\simeq \tmop{Pic}E \oplus \tmop{Pic}C'$ is a splitting of $G$-modules, implying
 \begin{eqnarray*}
  \hh & \simeq & H^1(G,\tmop{Pic}E)\oplus H^1(G,\tmop{Pic}C')\\
      & \simeq & \mb{Z}/n\mb{Z} \oplus \mb{Z}/n\mb{Z}.
 \end{eqnarray*}
The second isomorphism is a result of Proposition \ref{curve}. Now recalling that $H^3(G,\os(Y)^*))\simeq\mb{Z}/n\mb{Z}$ and that 
\begin{eqnarray*}
 Rel_{io}/E & \simeq & \tmop{ker}(\hh \longrightarrowlim^{d_2} H^3(G,\os(Y)^*)),
\end{eqnarray*}
we wish to decipher which elements of $\hh$ lie in $\tmop{ker}d_2$. We do this now:
The cyclic cover $\pi_C\colon C'\to C$ corresponds to an $n$-torsion line bundle $L'\in\tmop{Pic}C$.
Now there exists an elliptic fibration $\varphi_1\colon Z \to C$ and a fibred diagram
$$
\xymatrix{
Y \ar[r]^\pi \ar[d] & Z \ar[d]^{\varphi_1}\\
C'        \ar[r]^{\pi_C}&   C
}
$$
This implies that $Y$ is the cyclic cover corresponding to $\varphi_1^*(L')$. 
By Lemma \ref{lem}, $L'\simeq\os_C\left(\sum_{i=1}^{r}p_i-rp_0\right), p_i\in C$, where $r=2g(C)-1$. Then  $\varphi_1^*(L')\simeq\os_Z\left(\sum_{i=1}^{r}Z_{p_i}-rZ_{p_0}\right)$.
As in the proof of Proposition \ref{curve}, letting $p_i'\in \pi_C^{-1}(p_0)$, 
\begin{eqnarray*}
 L & \assign & \oy\left(\sum_{i=1}^{r}Y_{p_i'}-rY_{p_0'}\right)\in\hh,
\end{eqnarray*}
generating the second copy of $\mb{Z}/n\mb{Z}$ in $\hh$. Moreover, just as in the proof of Proposition \ref{curve}, $d_2(L)$ generates $H^3(G,\os(Y)^*)$.
Similarly, there exist ${q_i'\in E}$ such that $N=\left(\sum_{i=1}^{r'}Y_{q_i'}-r'Y_{q_0'}\right)\in\hh$, where $r'=2g(E/G)-1$. Moreover, $N$ generates the first copy of $\mb{Z}/n\mb{Z}\in\hh$ and, just as with $L$, $d_2(N)$ generates $H^3(G,\os(Y)^*)$. Thus there exists some power $M$ of $N$ which also generates the first copy of $\mb{Z}/n\mb{Z}\in\hh$ such that
$d_2(L^i\otimes M^{-i})=0$, and the corresponding relations satisfy the overlap condition. The result follows immediately.
\end{proof}
\end{proposition}

\section{Ogg-Shafarevich theory for orders}\label{ogg}

As seen in the previous section, Ogg-Shafarevich theory gives us a way of constructing Azumaya algebras on elliptic surfaces using \'etale covers of the base curve $C.$ Our interests, however, lie in constructing the more general objects orders. For this reason, we delve into the case where $\pi_C:C'\lrw C$ is a ramified $n\colon1$ cyclic cover. The setup is the same as before with two alterations. Firstly, we no longer require that all fibres be smooth and thus from now on, by elliptic fibration we shall mean an elliptic fibration as defined in Definition \ref{el} (we are now able to drop this requirement since we do not need all the theory developed above for smooth elliptic fibrations). Secondly, now $\pi_C$ will be ramified at points $c_i'$ such that $Z_{\pi_C(c_i')}$ is irreducible and smooth:
\begin{eqnarray*}
  Y & \longrightarrowlim^{\pi} & Z \nonumber\\
  \phi_Y \downarrow & \square & \downarrow \phi \nonumber\\
  C' & \longrightarrowlim^{\pi_C} & C \nonumber
\end{eqnarray*}
Letting $\pi_C$ be ramified at $c_i'\in C'$ with ramification index $r_i$, then $\pi$ is ramified solely at the fibres $Y_{c'_i}$ with index $r_i.$ 
%
Using the noncommutative cyclic covering trick in this case will result in the orders being ramified on $Z_{\pi(c_i')}$ with ramification index $r_i$.
%
%
%
%
%
%
In the following, we provide a collection of examples:

\subsection{Ogg-Shafarevich for orders on rational elliptic fibrations}

We first describe the construction of a rational elliptic fibration:

\begin{example}\label{ellrat}\index{elliptic fibration!rational elliptic fibration} {\bf Rational elliptic fibrations}: Let $C_1, C_2$ be two elliptic curves on $\p2$ in general position, that is, intersecting in $9$ distinct points, $p_i,i\in\{1,\ldots,9\}.$ The linear system $\lambda C_1+\mu C_2$ defines a rational map $\p2\dashrightarrow\p1$ with the $p_i$'s as base points. Blowing up these base points results in a morphism $p:\tmop{Bl}_9\p2\lrw\p1$ with general fibre a genus 1 curve and thus $p$ is an elliptic fibration. We proceed to show that $p$ has sections, in fact that any exceptional curve (of which there are infinitely many!) on $Z\assign\tmop{Bl}_9\p2$ is a section of $p$. Now $\tmop{Pic}Z\simeq\mb{Z}H\oplus\bigoplus_{i=1}^9\mb{Z}E_i,$ where $H$ is given by the pullback of the hyperplane via $\s:Z\lrw\p2$ and the $E_i$ are the inverse images of the $p_i.$ We have $K_Z\sim-3H+\sum_{i=1}^9E_i.$ Let $c\in\p1$ be a closed point. Then $F_c:=p^{-1}(c)$ is linearly equivalent to $3H-\sum_{i=1}^9E_i\sim-K_Z.$ Letting $E$ be any exceptional curve (that is, a rational curve with self-intersection $-1$) on $Z$, the adjunction formula states that
\begin{eqnarray*}
2(g(E)-1) & = & E\cdot(E+K_Z)\\
          & = & E^2-E\cdot F_c.
\end{eqnarray*}
Rearranging, we see that $E\cdot F_c=E^2-2(g(E)-1)=1.$ Since this holds for all closed $c\in\p1,$ $E$ is a section. 
In this case, for $S_1,S_2\in\Phi$, 
\begin{eqnarray*}
 S_1\oplus S_2 & \sim & S_1+S_2-S_0+\alpha F
\end{eqnarray*}
where $\alpha=(S_1+S_2)\cdot S_0-S_1\cdot S_2+1$ (this is \cite[Lemma 16]{morrison1983group}).
\end{example}

We wish to construct orders on $Z$ ramified on $C_1\cup C_2,$ where we have abused notation by writing $C_i\subset Z$ for the strict transform of $C_i\subset\p2$. These orders are examples of minimal orders (see \cite[Example 3.4]{ncy}) on non-minimal surfaces. We first construct the double cover of $\pi:Y\lrw Z$ ramified on $C_1$ and $C_2$:
\begin{lemma}\label{y}
There exists a K3 surface $Y$ and a double cover $\pi:Y\lrw Z$ ramified on $C_1\cup C_2$ with ramification index $2.$ Moreover, there is a Jacobian elliptic fibration $\phi_Y\colon Y \to \p1$. 
\begin{proof}
Firstly, we note that for $i\in\{1,2\}$, $C_i$ is the fibre above a point $p_i\in C$. Since $\os_C(p_1+p_2)$ is $2$-divisible in $\tmop{Pic}C$, there is a double cover $\pi_C\colon C'\to C$ ramified solely at $p_1$ and $p_2$. By the Riemann-Hurwitz formula, $C'\simeq\p1$. We form the fibred product
\begin{eqnarray}
  Y & \longrightarrowlim^{\pi} & Z \nonumber\\
  \phi_Y \downarrow & \square & \downarrow \phi \nonumber\\
  C' & \longrightarrowlim^{\pi_C} & C \nonumber
\end{eqnarray}
and it follows that $\pi\colon Y\to Z$ is the double cover of $Z$ ramified solely on $C_1\cup C_2$. Now $Y$ is a resolution of a double sextic $Y'\to\p2$, the sextic possessing only simple singularities. By \cite[ Sect. 1, discussion preceding Proposition A]{doub}, $Y$ is a K3 surface. Moreover, by construction, $Y$ is elliptically fibred over $C'$. We now show that the pullback under $\pi$ of any section of $\phi$ is a section of $\phi_Y$, implying $\phi_Y$ is a Jacobian elliptic fibration. To see this, let $S$ be a section of $\phi$, $F$ a fibre. Then $\pi^*S\cdot\pi^*F=2$. However, recalling $C\simeq\p1$, we see that $\pi^*F\sim 2F'$, where $F'$ is a fibre of $\phi_Y$, implying $\pi^*S\cdot F'=1$.

\end{proof}
\end{lemma}

In order to construct orders ramified on $C_1\cup C_2,$ we require the existence of nontrivial elements of $H^1(G,\tmop{Pic}Y).$ In general, however, it is extremely difficult to even compute $\tmop{Pic}Y.$ There are two special cases in which we can find nontrivial $1-$cocycles and thus construct nontrivial orders.

\subsubsection{When $\Phi$ has 2-torsion}

The first is when $\Phi$, the group of sections of $\phi\colon Z\to C$, has $2-$torsion. There are many examples of such rational elliptic fibrations and we refer the reader to \cite[Section ``Torsion groups of rational elliptic surfaces'']{rat} for explicit examples.  

\begin{proposition}
If $\Phi$ has a nontrivial $2-$torsion section $S$ such that $S_{|C_i}\nsim S_{0|C_i}$, $i\in\{1,2\}$, then there exists a nontrivial $L\in H^1(G,\tmop{Pic}Y)$ and the corresponding relation satisfies the overlap condition. The cyclic algebra $A=\oy\oplus\ls$ is a numerically Calabi-Yau order ramified on $C_1\cup C_2$.

\begin{proof}
Since $S$ is a $2$-torsion section, $2(S-S_0)\sim \alpha F,$ where $F$ is the class of a fibre. Let $S'=\pi^*S, S_0'=\pi^*S_0,$ and $F'$ be the class of a fibre of $\phi_Y\colon Y\to C$'. Then $\pi^*F=2F'$ and
\begin{eqnarray*}
\s^*: \tmop{Pic}Y & \lrw  & \tmop{Pic}Y\\
       S'         & \longmapsto &  S'\\
       S_0'          & \longmapsto &  S_0'\\
       F'          & \longmapsto & F'
\end{eqnarray*}
and 
\begin{eqnarray*}
S'-S_0'-\alpha F'+\s^*(S'-S_0'-\alpha F') & \sim & 2(S'-S_0')-2\alpha F'\\
                              & \sim & \pi^*(2(S-S_0)-\alpha F)\\
			      & \sim & 0.
\end{eqnarray*}
Thus $L=\oy(S'-S_0'-\alpha F')\in H^1(G,\tmop{Pic}Y).$ The corresponding relation satisfies the overlap condition due to Proposition \ref{ol}. Moreover, $L$ is nontrivial in $H^1(G,\tmop{Pic}Y)$ since the ramification of $A=\oy\oplus L_\s$ above $C_i$ is given by the $2$-torsion line bundle $\os_{C_i}(S'_{|Y_{c_i'}}-S'_{0|Y_{c_i'}})$, which is nontrivial by assumption. Thus each $\ti{C_i}$ is irreducible and $A$ is maximal by Lemma \ref{max}.
\end{proof}
\end{proposition}
\begin{remark}
Torsion sections of elliptic fibrations seem to provide a substantial number of examples here. We can perform the same trick for elliptically fibred K3 surfaces: such sections are abundant and have been studied extensively by Persson and Miranda \cite{k3tors}.
\end{remark}

\subsubsection{The second case in which we can find nontrivial 1-cocycles}

The second case in which we may construct such orders is more subtle and involved, requiring us to study the geometry of a general rational elliptic fibration $\phi:Z\to\p1$ further.

First notice that there is an involution $\tau$ of $Z$ which sends $z\mapsto-z$ on fibres (for details, see \cite[Introduction, p.3]{rat}). Setting $G=\langle\tau|\tau^2=1\rangle$, we let ${\psi:Z\to\ti{X}\assign Z/G}$ be the corresponding quotient morphism and note that $\ti{X}$ is smooth. Blowing down exceptional curves disjoint from $\psi(S_0)$ yields a birational morphism $\mu\colon\ti{X}\to X$ and $X$ is ruled over $C\simeq\p1$. This is because  the quotient of an elliptic curve by the $\mb{Z}_2$-action $z\mapsto -z$ is $\p1$.
\begin{displaymath}
\xymatrix{
	Z  \ar@/^2pc/[rr]^{\rho}\ar[drr]_\phi \ar[r]^\psi & \ti{X} \ar[dr] \ar[r]^\mu & X \ar[d]^p\\
		   &  & \p1}
\end{displaymath}
Letting $\rho\colon Z\to X$ denote the composite morphism $\mu\circ\psi$, $S_0^2=-1$ implies that $\rho(S_0)^2=-2$. Thus $X\simeq\mb{F}_2$ and $\psi$ is ramified on $\ti{C_0}\cup\ti{T}$, where $\ti{C_0}$ is the strict transform of the special section $C_0$ on $\mb{F}_2$ and $\ti{T}$ the strict transform of a trisection $T$ disjoint from $C_0$. Thus we see that classifying rational elliptic fibrations is equivalent to classifying trisections $T$ on $\mb{F}_2$ disjoint from $C_0$ with at most simple singularities and such that $\ti{C_0}+\ti{T}$ is $2$-divisible in $\tmop{Pic}(\ti{X})$. For any effective divisor $D$ on $X$, we let $\ti{D}$ denote its strict transform on $\ti{X}$ with respect to the birational morphism $\mu$.

\begin{theorem}\label{ncyrat}
If $\phi:Z\lrw\p1$ corresponds to a trisection $T\subset\mathbb{F}_2$ with 2 nodes, then we can construct a numerically Calabi-Yau order $A$ ramified on 2 fibres of $\phi.$ The explicit construction of $A$ is given in the proof of the theorem.
\end{theorem}
Before we prove this theorem we require the following lemma.

 \begin{lemma}\label{lemlem}
Assume there exists a section $S$ of $p\colon X\to\p1$ such that $S\cdot C_0=0$ and ${S'=\psi^{-1}(\ti{S})}$ is irreducible. Then there exists an $L\in\tmop{Pic}Y$ (where $Y=Z\times_CS'$) such that $A=\oy\oplus L_\s$ is an order on $Z$ ramified on $Z_{c_i}$, where the $c_i$ are the ramification points of $\pi_C\colon S'\to C$.
\begin{proof}
We first note that since $S$ is a section of $p\colon X\to C$, $S'=\rho^{-1}(S)$ is a bisection of $\phi:Z\to C$ and there is a corresponding irreducible double cover $\pi_C:S'\lrw C$.
We form
\begin{eqnarray}
  Y & \longrightarrowlim^{\pi} & Z \nonumber\\
  \varphi' \downarrow & \square & \downarrow \varphi \nonumber\\
  S' & \longrightarrowlim^{\pi_C} & C \nonumber
\end{eqnarray}
Now $\pi^{-1}(S')=S'\times_CS'$ splits into $2$ copies of $S',$ say, $S_1$ and $S_2.$ Recalling that
\begin{eqnarray*}
 \tmop{Pic}X & \simeq \mb{Z}C_0\oplus\mb{Z}F,
\end{eqnarray*}
$S\sim C_0+nF\in\tmop{Pic}X$, implying 
\begin{eqnarray}\label{one}
S' & \assign & \psi^{-1}(\ti{S})\sim2S_0+nF',
\end{eqnarray}
where $F'$ is the class of a fibre of $\phi.$ Let $S_0'=\pi^{-1}S_0$, the reduced inverse image of $S_0$. Then $\s(S_0')=S_0'$ , $\s(S_1)=S_2$ and $\s(F')=F'$.
%
%
%
%
Then
\begin{eqnarray*}
 (1+\s)(S_1-S_0'-nF')& \sim & S_1+S_2-2S_0'-2nF'\\
                   &  \sim & \pi^*(S'-2S_0-nF)\\
                   &  \sim &   0.
\end{eqnarray*}
The last linear equivalence is a result of (\ref{one}). Thus $L\simeq\oy(S_1-S_0+nF') \in H^1(G,\tmop{Pic}Y)$. The corresponding relation satisfies the overlap condition due to Proposition \ref{ol}. By assumption $S\cdot C_0=0$ and this implies $S_1\cdot S_0'=0$. We conclude that ${(S_1-S_0)_{|Y_{c_i'}}\not\simeq\os_{Y_{c_i'}}}$. By theorem \ref{ram}, the ramification of $A_S=\oy\oplus L_{\s}$ along $Z_{\pi(c_i')}$ is the cyclic cover corresponding to $\os_{Y_{c_i'}}(S_{1|Y_{c_i'}}-S_{0|Z_{c_i'}})$, which is nontrivial. By Lemma \ref{max}, $A$ is maximal.
\end{proof}
\end{lemma}
\begin{proof}(of theorem \ref{ncyrat})
To prove the theorem, first note that if $S'$ in Lemma \ref{lemlem} is rational, then $\pi_C:S'\lrw C$ is ramified above 2 points and the order $A_S$ is ramified on $2$ two fibres $Z_{c_i}$ and thus numerically Calabi-Yau (this is precisely the example Chan and Kulkarni deduce exists via the Artin-Mumford sequence \cite[Example 3.4]{ncy}). We are thus required to show that if $\phi:Z\lrw C$ corresponds to a trisection with 2 nodes, there exists a section $S$ of $p\colon X\to C$ such that $S'\assign\rho^{-1}(S)$ is an irreducible rational bisection of $\phi\colon Z\to C$.
This is equivalent to showing there exists a section $S$ such that $\ti{S}\cdot(\ti{T}+\ti{C_0})=2$ since then ${S'\lrw S}$ is ramified above 2 points, implying $S'$ rational.
Now since $T$ is a trisection of $\phi$ disjoint from $C_0$, the intersection theory on $\mb{F}_2$ implies that $T\sim3C_0+6F$. We assume $T$ is a $2-$nodal trisection with nodes at $q_1,q_2.$ Since $h^0(\os_X(aC_0+bF))=1-a^2+ab+b$,  $h^0(\os_X(C_0+2F))=4$ and there exists a divisor $S\sim C_0+2F$ through each $q_i$ with direction different to $T$. Now $S\cdot(T+C_0)=6$ implies that $\ti{S}\cdot(\ti{T}+\ti{C_0})=2$ and we are done.
\end{proof}

\begin{figure}[H]\label{tnode}
\begin{center}
\includegraphics[scale=0.7]{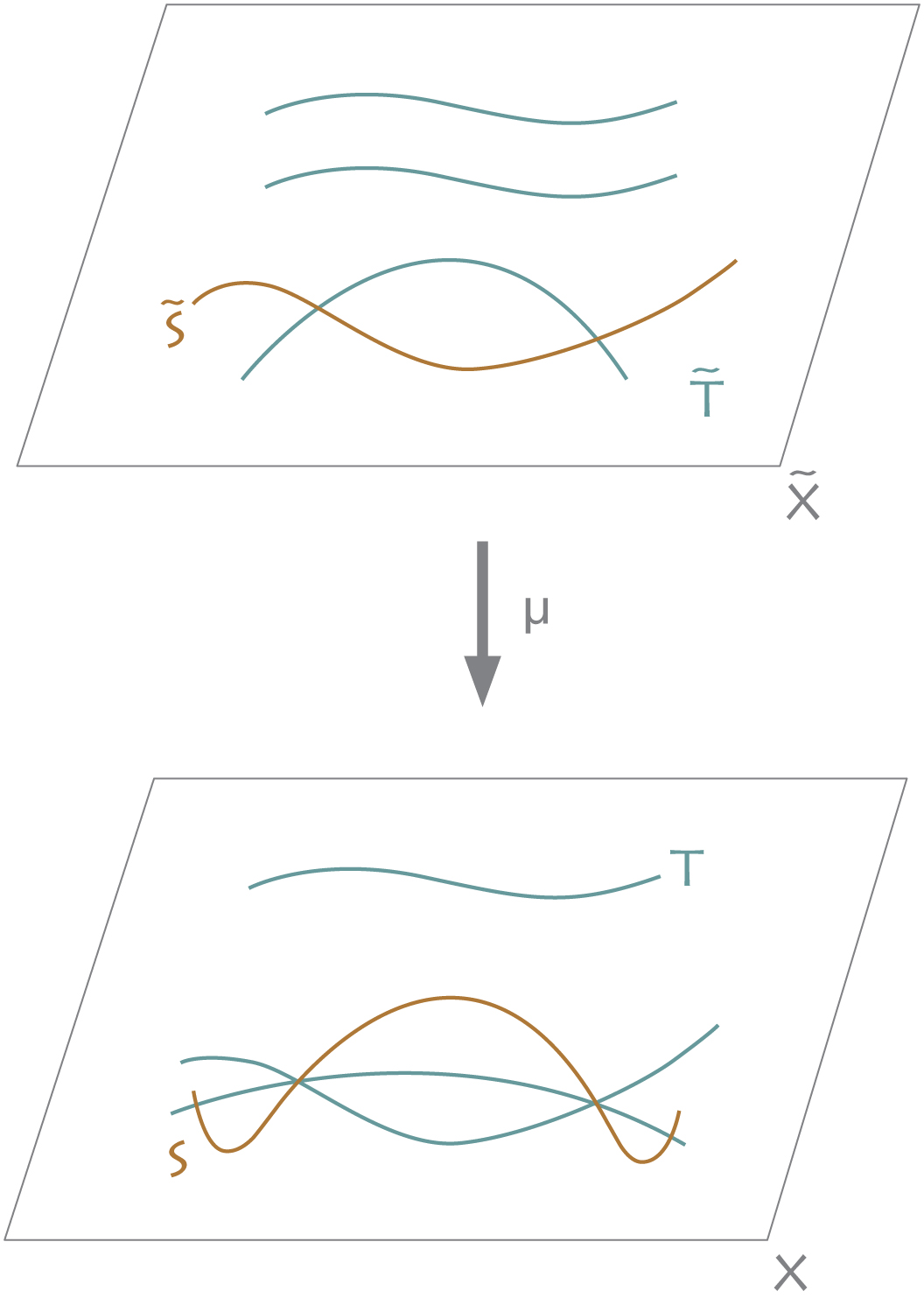}
\end{center}
\caption{The morphism $\mu\colon\ti{X}\to X$.}
\end{figure}

\section{Bielliptic orders}
\index{Bielliptic surface}
In this section we construct orders on $E\times\p1,$ (where $E$ is an elliptic curve), orders which are analogous to bielliptic surfaces, which we now define: 
\begin{definition}(Bielliptic surface)
A surface $X$ is bielliptic if $p_g(X)=0, q(X)=1$.
\end{definition}

Bielliptic surfaces play an important role in the classification of surfaces of Kodaira dimension $0$ and their constructions are realised in the following proposition.

\begin{proposition}\cite{serr}
Let $X$ be a bielliptic surface. Then there exists two elliptic curves $E,F$ and an abelian group $G$ acting on $E$ and $F$ such that:

(i) $E/G$ is an elliptic curve and $F/G\simeq \mathbb{P}^1$;

(ii) $X\simeq (E\times F)/G$.
\end{proposition}

\begin{proposition}\cite{serr}
There are seven types of bielliptic surfaces, described in the following table (in which $F\simeq \mb{C}/(\mb{Z}\oplus\omega\mb{Z}), \rho:=e^{2\pi i/3}$):

\begin{tabular}{|l|l|l|l|}
\hline
Type & $\omega$ & $G$ & Action of $G$ on $F\simeq\mathbb{C}/(\mathbb{Z}\omega\oplus\mathbb{Z})$\\
\hline
1 & any & $\mathbb{Z}_2$ &                   $x\mapsto-x$\\
2 & any & $\mathbb{Z}_2\times\mathbb{Z}_2$ & $x\mapsto-x, x\mapsto x+\epsilon \tmop{with} 2\epsilon=0$\\
3 & $i$ & $\mathbb{Z}_4$ &                   $x\mapsto i x$\\
4 & $i$ & $\mathbb{Z}_4\times\mathbb{Z}_2$ & $x\mapsto i x, x\mapsto x+(1+i)/2$\\
5 & $\rho$ & $\mathbb{Z}_3$ &                $x\mapsto\rho x$\\
6 & $\rho$ & $\mathbb{Z}_3\times\mathbb{Z}_3$ & $x\mapsto \rho x, x\mapsto x+(1-\rho)/3$\\
7 & $\rho$ & $\mathbb{Z}_6$ &                $x\mapsto-\rho x$\\
\hline
\end{tabular}

\begin{remark}\label{org}
 Bielliptic surfaces can be constructed in a manner analogous to Construction \ref{cyctate}. We provide examples in the cyclic case: for $G$ cyclic, we take the cover $\pi\colon E\times F\to E\times \p1$ with Galois group $G'=\langle\s|\s^n=1\rangle$. Letting $M\assign\os_E(\varepsilon-e_0)$, a line bundle of order $n$ in $\tmop{Pic}E$, we see that $p_1^*(M)\in H^1(G',\Phi)$, where $p_1\colon E\times F\to E$ is the first projection. We twist the $G'$-action using $p_1^*(M)$ in the following manner:
\begin{eqnarray*}
\s_{\varepsilon}: E\times F & \lrw & E\times F\\
   (e,f)     & \mapsto & (e\oplus\varepsilon,\s(f)).
\end{eqnarray*}
\end{remark}
\end{proposition}
In the following Proposition, we set $Y=E\times F$ and $M$ as in Remark \ref{org} above. 
\begin{proposition}(Bielliptic orders)
For each bielliptic surface $X=(E\times F)/G$ such that $G$ is cyclic, $A_X=\oy\oplus\ls\oplus\ldots\oplus L_\s^{n-1}$ is a maximal order on $E\times\p1$ corresponding to $X$, where $n=|G|$.
\begin{proof}
Let $X=(E\times F)/G$ be a bielliptic surface, $G$ cyclic, $Y=E\times F.$ Let $n$ be the order of $G$, generated by $\s_{\varepsilon}$.
As seen in Remark \ref{org}, such a surface is given by a covering $\pi:Y\lrw E\times\p1$ and $L\assign p_1^*(\varepsilon-e_0)\in H^1(G',\Phi_Y),$ where $G'$ is the Galois group of $\pi$ and $p_1\colon Y\to E$ is the first projection.
 We see that  $L\assign p_1^*(\os_E(\varepsilon-e_0))\in H^1(G',\tmop{Pic}Y)$ and the corresponding relation satisfies the overlap condition by Lemma \ref{ol}. We form the noncommutative cyclic algebra ${A_X=\oy\oplus\ls\oplus\ldots\oplus L_\s^{n-1}}$. Implementations of Theorem \ref{ram}, Lemma \ref{untot} and Lemma \ref{max} verify that each $A_X$ is a maximal order on $Z=Y/G'=E\times\p1.$ 
\end{proof}
\end{proposition}

%% file: app1.tex
\appendix

%% file: c5.tex
\chapter{Azumaya algebras and Brauer groups}\label{br}

\section{Central simple algebras and Brauer groups}

Our initial objects of interest in noncommutative algebra are central simple algebras over a base field $K$.
\begin{definition}
 A $K$-algebra $A$ is called {\bf simple} if the only two-sided ideals of $A$ are $(0)$ and $A$ itself. A {\bf central simple} $K$-algebra is a simple finite $K$-algebra $A$ with centre $Z(A)=K$.\index{central simple agebra}
\end{definition}

\begin{example}
 The Quaternion algebra $\mb{H}=\mb{R}\oplus\mb{R}i\oplus\mb{R}j\oplus\mb{R}k$ is a central simple  $\mb{R}$-algebra.
\end{example}

\begin{theorem}[Artin-Wedderburn Theorem]\cite[(3.5)]{lam}\label{aw}
 A central simple $K-$algebra $A$ is isomorphic to $M_n(D)$, the ring of $n\times n$ matrices over a $K$-division ring $D$.
\end{theorem}

\begin{definition}
 We say two central simple algebras $A\simeq M_n(D)$,  $A'\simeq M_m(D')$ are {\bf Morita equivalent}\index{Morita equivalent}\index{$\sim_M$} if $D\simeq D'$ and we write $A\me A'$.
\end{definition}

\begin{remark}
 Equivalently, $A\me A'$ if and only if there is a categorical equivalence $f\colon\tmop{Mod}A\to \tmop{Mod} A'$.
\end{remark}

We denote by $\br(K)$\index{$\br(K)$} the set of Morita equivalence classes of central simple $K$-algebras. Then
\begin{enumerate}
 \item Tensoring over $K$ yields a well-defined operation on $\br(K)$, since 
\begin{align*}
 M_n(D)\otimes_K M_m(D') & \simeq M_{mn}(D\otimes_K D')\\
                         &  \me     D\otimes_K D'.
\end{align*}
\item Setting $A^\circ$\index{$A^\circ$} to be the opposite algebra of $A$, $A\otimes A^\circ\simeq M_n(K)$, where ${n=\tmop{deg}_K A}$.
\end{enumerate}

Thus $\br(K)$\index{Brauer group!of a field} is a group with identity $[K]=[M_n(K)]$ and inverse $[A]^{-1}=[A^\circ]$. We call $\br(K)$ the {\bf Brauer group} of $K$.

\begin{example} We provides some brief examples of Brauer groups of fields:
\begin{enumerate} 
 \item $\br(\mb{R})\simeq\mb{Z}/2\mb{Z}$, the only nontrivial class being $[\mb{H}]$.
 \item $\br(K)=0$ if $K$ is any one of the following: (i) algebraically closed, (ii) finite, or (iii) the function field of a curve over an algebraically closed field.
\end{enumerate}
\end{example}

\begin{remark}
 The assignment of a Brauer group to a field is functorial, that is, given a field homomorphism $f\colon K\to F$, there is a corresponding group homomorphism $\br(f)\colon\br(K)\to\br(F)$ defined by $\br(f)(A)\assign A\otimes_KF$. The {\bf relative Brauer group} is defined to be  $\br(F/K)\assign\tmop{ker}(\br(f)\colon\br(K)\to\br(F))$.
\end{remark}

\section{Azumaya algebras and Brauer groups}
%
%
%
%


We now define sheaves of Azumaya algebras for an integral, noetherian, normal scheme $X$. 

\begin{definition}\index{Azumaya algebra!over a scheme}
 Let $A$ be a coherent sheaf of $\os_X$-algebras. We call $A$ a {\bf sheaf of Azumaya algebras} if the canonical morphism $A\otimes_{\os_X}A^\circ\to {\mathcal End}_{\os_X}(A)$ is an isomorphism.
\end{definition}

\begin{remark}
There are equivalent formulations for the definition of a sheaf of Azumaya algebras over a scheme $X$ which may be found in \cite[Section 1]{dixexp}.
\end{remark}

We call Azumaya algebras $A,A'$ Morita equivalent if there exists an equivalence of categories $f\colon\tmop{Mod}A\to \tmop{Mod} A'$. Moreover, the operation $\otimes_{\os_X}$ endows $\br(X)$, the set of Morita equivalence classes $[A]$ of Azumaya algebras over $X$, with a group structure with identity $[\os_X]$ and inverse $[A]^{-1}=[A^\circ]$.\index{Brauer group!of a scheme} We call $\br(X)$\index{$\br(X)$} the {\bf Brauer group} of $X$. As in the case of fields, the assignment of a Brauer group to a scheme is functorial. For further details, see \cite{dixexp}.

%% file: c6.tex
\chapter{MATLAB Code}\label{code}

\section{The double cover of $\p2$}
In Theorem \ref{thawesome}, we define an involution $\phi$ on $\tmop{Pic}Y\oplus T_Y\subset \Ll\simeq H^2(Y,\mb{Z})$ and wish to show it extends to an involution on $\Ll$. As demonstrated in the proof of Theorem \ref{thawesome}, to show $\phi$ extends to $\Ll$ we need only show that $\phi\otimes_{\mb{Z}}\mb{Q}$ preserves the integral lattice. The following MATLAB code verifies this in the rank 18 case. The other cases are analogous. 

\begin{verbatim}
 L = [-2,0,0,1,0,0,0,0,0,0,0,0,0,0,0,0,0,0,0,0,0,0;
       0,-2,1,0,0,0,0,0,0,0,0,0,0,0,0,0,0,0,0,0,0,0;
       0,1,-2,1,0,0,0,0,0,0,0,0,0,0,0,0,0,0,0,0,0,0;
       1,0,1,-2,1,0,0,0,0,0,0,0,0,0,0,0,0,0,0,0,0,0;
       0,0,0,1,-2,1,0,0,0,0,0,0,0,0,0,0,0,0,0,0,0,0;
       0,0,0,0,1,-2,1,0,0,0,0,0,0,0,0,0,0,0,0,0,0,0;
       0,0,0,0,0,1,-2,1,0,0,0,0,0,0,0,0,0,0,0,0,0,0;
       0,0,0,0,0,0,1,-2,0,0,0,0,0,0,0,0,0,0,0,0,0,0;
       0,0,0,0,0,0,0,0,-2,0,0,1,0,0,0,0,0,0,0,0,0,0;
       0,0,0,0,0,0,0,0,0,-2,1,0,0,0,0,0,0,0,0,0,0,0;
       0,0,0,0,0,0,0,0,0,1,-2,1,0,0,0,0,0,0,0,0,0,0;
       0,0,0,0,0,0,0,0,1,0,1,-2,1,0,0,0,0,0,0,0,0,0;
       0,0,0,0,0,0,0,0,0,0,0,1,-2,1,0,0,0,0,0,0,0,0;
       0,0,0,0,0,0,0,0,0,0,0,0,1,-2,1,0,0,0,0,0,0,0;
       0,0,0,0,0,0,0,0,0,0,0,0,0,1,-2,1,0,0,0,0,0,0;
       0,0,0,0,0,0,0,0,0,0,0,0,0,0,1,-2,0,0,0,0,0,0;
       0,0,0,0,0,0,0,0,0,0,0,0,0,0,0,0,0,1,0,0,0,0;
       0,0,0,0,0,0,0,0,0,0,0,0,0,0,0,0,1,0,0,0,0,0;
       0,0,0,0,0,0,0,0,0,0,0,0,0,0,0,0,0,0,0,1,0,0;
       0,0,0,0,0,0,0,0,0,0,0,0,0,0,0,0,0,0,1,0,0,0;
       0,0,0,0,0,0,0,0,0,0,0,0,0,0,0,0,0,0,0,0,0,1;
       0,0,0,0,0,0,0,0,0,0,0,0,0,0,0,0,0,0,0,0,1,0];
%This is the K3 lattice \Lambda

Pic = [1,0,0,0,0,0,0,0,0,0,0,0,0,0,0,0,0,0;
       0,1,0,0,0,0,0,0,0,0,0,0,0,0,0,0,0,0;
       0,0,1,0,0,0,0,0,0,0,0,0,0,0,0,0,0,0;
       0,0,0,1,0,0,0,0,0,0,0,0,0,0,0,0,0,0;
       0,0,0,0,1,0,0,0,0,0,0,0,0,0,0,0,0,0;
       0,0,0,0,0,1,0,0,0,0,0,0,0,0,0,0,0,0;
       0,0,0,0,0,0,1,0,0,0,0,0,0,0,0,0,0,0;
       0,0,0,0,0,0,0,1,0,0,0,0,0,0,0,0,0,0;
       0,0,0,0,0,0,0,0,1,0,0,0,0,0,0,0,0,0;
       0,0,0,0,0,0,0,0,0,1,0,0,0,0,0,0,0,0;
       0,0,0,0,0,0,0,0,0,0,1,0,0,0,0,0,0,0;
       0,0,0,0,0,0,0,0,0,0,0,1,0,0,0,0,0,0;
       0,0,0,0,0,0,0,0,0,0,0,0,1,0,0,0,0,0;
       0,0,0,0,0,0,0,0,0,0,0,0,0,1,0,0,0,0;
       0,0,0,0,0,0,0,0,0,0,0,0,0,0,1,0,0,0;
       0,0,0,0,0,0,0,0,0,0,0,0,0,0,0,1,0,0;
       1,0,0,0,0,0,0,0,0,0,0,0,0,0,0,0,0,0;
       0,3,0,0,1,1,1,1,1,1,1,1,1,1,1,1,1,1;
       0,0,0,0,0,0,0,0,0,0,0,0,0,0,0,0,1,0;
       0,0,0,0,0,0,0,0,0,0,0,0,0,0,0,0,-1,0;
       0,0,0,0,0,0,0,0,0,0,0,0,0,0,0,0,0,1;
       0,0,0,0,0,0,0,0,0,0,0,0,0,0,0,0,0,-1];
%This matrix defines the embedding of Pic Y in the K3 
lattice

T = null(Pic'*L);
% This provides us with the embedding of the transcendental 
lattice of Y in in the lattice

A = horzcat(Pic,T);
% This matrix defines the embedding of Pic Y \oplus T_Y in 
\Lambda

i_Pic = 
[0,1,1,1,1,1,1,1,1,1,1,1,1,1,1,1,1,1;
 1,0,1,1,1,1,1,1,1,1,1,1,1,1,1,1,1,1;
 0,0,-1,0,0,0,0,0,0,0,0,0,0,0,0,0,0,0;
 0,0,0,-1,0,0,0,0,0,0,0,0,0,0,0,0,0;
 0,0,0,0,-1,0,0,0,0,0,0,0,0,0,0,0,0,0;
 0,0,0,0,0,-1,0,0,0,0,0,0,0,0,0,0,0,0;
 0,0,0,0,0,0,-1,0,0,0,0,0,0,0,0,0,0,0;
 0,0,0,0,0,0,0,-1,0,0,0,0,0,0,0,0,0,0;
 0,0,0,0,0,0,0,0,-1,0,0,0,0,0,0,0,0,0;
 0,0,0,0,0,0,0,0,0,-1,0,0,0,0,0,0,0,0;
 0,0,0,0,0,0,0,0,0,0,-1,0,0,0,0,0,0,0;
 0,0,0,0,0,0,0,0,0,0,0,-1,0,0,0,0,0,0;
 0,0,0,0,0,0,0,0,0,0,0,0,-1,0,0,0,0,0;
 0,0,0,0,0,0,0,0,0,0,0,0,0,-1,0,0,0,0;
 0,0,0,0,0,0,0,0,0,0,0,0,0,0,-1,0,0,0;
 0,0,0,0,0,0,0,0,0,0,0,0,0,0,0,-1,0,0;
 0,0,0,0,0,0,0,0,0,0,0,0,0,0,0,0,-1,0;
 0,0,0,0,0,0,0,0,0,0,0,0,0,0,0,0,0,-1];
% the involution phi as defined on Pic Y

i = blkdiag(i_Pic,-eye(4));
% the involution phi as defined on Pic Y \oplus T_Y

phi = A*i*inv(A) % the involution phi on \Lambda
\end{verbatim}
MATLAB then outputs the matrix $\phi$ which preserves the integral lattice as required:
\begin{verbatim}
phi = 
[-3,-2,1,1,0,0,0,0,0,0,0,0,0,0,0,0,3,1,0,0,0,0;
-2,-3,1,1,0,0,0,0,0,0,0,0,0,0,0,0,3,1,0,0,0,0;
0,0,-1,0,0,0,0,0,0,0,0,0,0,0,0,0,0,0,0,0,0,0;
0,0,0,-1,0,0,0,0,0,0,0,0,0,0,0,0,0,0,0,0,0,0;
0,0,0,0,-1,0,0,0,0,0,0,0,0,0,0,0,0,0,0,0,0,0;
0,0,0,0,0,-1,0,0,0,0,0,0,0,0,0,0,0,0,0,0,0,0;
0,0,0,0,0,0,-1,0,0,0,0,0,0,0,0,0,0,0,0,0,0,0;
0,0,0,0,0,0,0,-1,0,0,0,0,0,0,0,0,0,0,0,0,0,0;
0,0,0,0,0,0,0,0,-1,0,0,0,0,0,0,0,0,0,0,0,0,0;
0,0,0,0,0,0,0,0,0,-1,0,0,0,0,0,0,0,0,0,0,0,0;
0,0,0,0,0,0,0,0,0,0,-1,0,0,0,0,0,0,0,0,0,0,0;
0,0,0,0,0,0,0,0,0,0,0,-1,0,0,0,0,0,0,0,0,0,0;
0,0,0,0,0,0,0,0,0,0,0,0,-1,0,0,0,0,0,0,0,0,0;
0,0,0,0,0,0,0,0,0,0,0,0,0,-1,0,0,0,0,0,0,0,0;
0,0,0,0,0,0,0,0,0,0,0,0,0,0,-1,0,0,0,0,0,0,0;
0,0,0,0,0,0,0,0,0,0,0,0,0,0,0,-1,0,0,0,0,0,0;
-2,-2,1,1,0,0,0,0,0,0,0,0,0,0,0,0,2,1,0,0,0,0;
-6,-6,3,3,0,0,0,0,0,0,0,0,0,0,0,0,9,2,0,0,0,0;
0,0,0,0,0,0,0,0,0,0,0,0,0,0,0,0,0,0,-1,0,0,0;
0,0,0,0,0,0,0,0,0,0,0,0,0,0,0,0,0,0,0,-1,0,0;
0,0,0,0,0,0,0,0,0,0,0,0,0,0,0,0,0,0,0,0,-1,0;
0,0,0,0,0,0,0,0,0,0,0,0,0,0,0,0,0,0,0,0,0,-1]
\end{verbatim}

\section{The double cover of $\p1\times\p1$}\label{b2}
The code for this case is analogous to that above.
We include here the matrix defining the involution $\phi$ in Proposition \ref{p1}, which preserves the integral lattice, indicating $\phi$ is an isometry on $\Ll$.

\begin{verbatim}
 phi = 
[-1,0,0,0,0,0,0,0,0,0,0,0,0,0,0,0,0,1,0,1,0,0;
0,-1,0,0,0,0,0,0,0,0,0,0,0,0,0,0,0,0,0,0,0,0;
0,0,-1,0,0,0,0,0,0,0,0,0,0,0,0,0,0,0,0,0,0,0;
0,0,0,-1,0,0,0,0,0,0,0,0,0,0,0,0,0,1,0,1,0,0;
0,0,0,0,-1,0,0,0,0,0,0,0,0,0,0,0,0,0,0,0,0,0;
0,0,0,0,0,-1,0,0,0,0,0,0,0,0,0,0,0,0,0,0,0,0;
0,0,0,0,0,0,-1,0,0,0,0,0,0,0,0,0,0,0,0,0,0,0;
0,0,0,0,0,0,0,-1,0,0,0,0,0,0,0,0,0,0,0,0,0,0;
0,0,0,0,0,0,0,0,-1,0,0,0,0,0,0,0,0,0,0,0,0,0;
0,0,0,0,0,0,0,0,0,-1,0,0,0,0,0,0,0,0,0,0,0,0;
0,0,0,0,0,0,0,0,0,0,-1,0,0,0,0,0,0,0,0,0,0,0;
0,0,0,0,0,0,0,0,0,0,0,-1,0,0,0,0,0,0,0,0,0,0;
0,0,0,0,0,0,0,0,0,0,0,0,-1,0,0,0,0,0,0,0,0,0;
0,0,0,0,0,0,0,0,0,0,0,0,0,-1,0,0,0,0,0,0,0,0;
0,0,0,0,0,0,0,0,0,0,0,0,0,0,-1,0,0,0,0,0,0,0;
0,0,0,0,0,0,0,0,0,0,0,0,0,0,0,-1,0,0,0,0,0,0;
-1,0,1,-1,1,0,0,0,0,0,0,0,0,0,0,0,1,0,0,0,1,1;
0,0,0,0,0,0,0,0,0,0,0,0,0,0,0,0,0,1,0,2,0,0;
-1,0,1,-1,1,0,0,0,0,0,0,0,0,0,0,0,2,0,-1,0,1,1;
0,0,0,0,0,0,0,0,0,0,0,0,0,0,0,0,0,0,0,-1,0,0;
0,0,0,0,0,0,0,0,0,0,0,0,0,0,0,0,0,1,0,1,-1,0;
0,0,0,0,0,0,0,0,0,0,0,0,0,0,0,0,0,1,0,1,0,-1]
\end{verbatim}

\section{The double cover of $\mathbb{F}_2$}\label{b3}
The code for this case is analogous to that above.
We include here the matrix defining the involution $\phi$ in Proposition \ref{f2}, which preserves the integral lattice, indicating $\phi$ is an isometry on $\Ll$.

\begin{verbatim}
phi = [-4,-2,2,0,1,2,-4,2,0,0,0,0,0,0,0,0,2,5,0,0,0,0;
-2,-1,0,1,0,1,-2,1,0,0,0,0,0,0,0,0,1,2,0,0,0,0;
0,0,-1,0,0,0,0,0,0,0,0,0,0,0,0,0,0,0,0,0,0,0;
-2,0,0,0,0,1,-2,1,0,0,0,0,0,0,0,0,1,2,0,0,0,0;
0,0,0,0,-1,0,0,0,0,0,0,0,0,0,0,0,0,0,0,0,0,0;
0,0,0,0,0,-1,0,0,0,0,0,0,0,0,0,0,0,0,0,0,0,0;
-3,-2,2,0,1,2,-5,2,0,0,0,0,0,0,0,0,2,5,0,0,0,0;
0,0,0,0,0,0,0,-1,0,0,0,0,0,0,0,0,0,0,0,0,0,0;
0,0,0,0,0,0,0,0,-1,0,0,0,0,0,0,0,0,0,0,0,0,0;
0,0,0,0,0,0,0,0,0,-1,0,0,0,0,0,0,0,0,0,0,0,0;
0,0,0,0,0,0,0,0,0,0,-1,0,0,0,0,0,0,0,0,0,0,0;
0,0,0,0,0,0,0,0,0,0,0,-1,0,0,0,0,0,0,0,0,0,0;
0,0,0,0,0,0,0,0,0,0,0,0,-1,0,0,0,0,0,0,0,0,0;
0,0,0,0,0,0,0,0,0,0,0,0,0,-1,0,0,0,0,0,0,0,0;
0,0,0,0,0,0,0,0,0,0,0,0,0,0,-1,0,0,0,0,0,0,0;
0,0,0,0,0,0,0,0,0,0,0,0,0,0,0,-1,0,0,0,0,0,0;
-8,-4,4,1,2,5,-10,5,0,0,0,0,0,0,0,0,4,12,0,0,0,0;
-3,-2,2,0,1,2,-4,2,0,0,0,0,0,0,0,0,2,4,0,0,0,0;
0,0,0,0,0,0,0,0,0,0,0,0,0,0,0,0,0,0,-1,0,0,0;
0,0,0,0,0,0,0,0,0,0,0,0,0,0,0,0,0,0,0,-1,0,0;
0,0,0,0,0,0,0,0,0,0,0,0,0,0,0,0,0,0,0,0,-1,0;
0,0,0,0,0,0,0,0,0,0,0,0,0,0,0,0,0,0,0,0,0,-1;]
\end{verbatim}

%% file: thesishugo.bbl
\begin{thebibliography}{BPVdV84}

\bibitem[AG60]{ausgold}
Maurice Auslander and Oscar Goldman.
\newblock Maximal orders.
\newblock {\em Trans. Amer. Math. Soc.}, 97:1--24, 1960.

\bibitem[AM72]{artmum}
M.~Artin and D.~Mumford.
\newblock Some elementary examples of unirational varieties which are not
  rational.
\newblock {\em Proc. London Math. Soc. (3)}, 25:75--95, 1972.

\bibitem[Ati56]{aks}
M.~Atiyah.
\newblock On the {K}rull-{S}chmidt theorem with application to sheaves.
\newblock {\em Bull. Soc. Math. France}, 84:307--317, 1956.

\bibitem[AVdB90]{av_twist}
M.~Artin and M.~Van~den Bergh.
\newblock Twisted homogeneous coordinate rings.
\newblock {\em J. Algebra}, 133(2):249--271, 1990.

\bibitem[Bar85]{k3en}
W.~Barth.
\newblock Lectures on {$K3$}- and {E}nriques surfaces.
\newblock In {\em Algebraic geometry, {S}itges ({B}arcelona), 1983}, volume
  1124 of {\em Lecture Notes in Math.}, pages 21--57. Springer, Berlin, 1985.

\bibitem[Bea96]{bo}
A.~Beauville.
\newblock {\em Complex algebraic surfaces}.
\newblock Cambridge University Press, 1996.

\bibitem[BLR90]{neron}
Siegfried Bosch, Werner L{\"u}tkebohmert, and Michel Raynaud.
\newblock {\em N\'eron models}, volume~21 of {\em Ergebnisse der Mathematik und
  ihrer Grenzgebiete (3) [Results in Mathematics and Related Areas (3)]}.
\newblock Springer-Verlag, Berlin, 1990.

\bibitem[BPVdV84]{bpv}
W.~Barth, C.~Peters, and A.~Van~de Ven.
\newblock {\em Compact complex surfaces}, volume~4 of {\em Ergebnisse der
  Mathematik und ihrer Grenzgebiete (3) [Results in Mathematics and Related
  Areas (3)]}.
\newblock Springer-Verlag, Berlin, 1984.

\bibitem[Cal]{cald}
A.~Caldararu.
\newblock {\em {Derived categories of twisted sheaves on Calabi-Yau manifolds,
  2000}}.
\newblock PhD thesis, Thesis, Cornell University.

\bibitem[Cha05]{cyc}
Daniel Chan.
\newblock Noncommutative cyclic covers and maximal orders on surfaces.
\newblock {\em Adv. Math.}, 198(2):654--683, 2005.

\bibitem[CI05]{mmpord}
Daniel Chan and Colin Ingalls.
\newblock The minimal model program for orders over surfaces.
\newblock {\em Invent. Math.}, 161(2):427--452, 2005.

\bibitem[CK05]{ncy}
Daniel Chan and Rajesh~S. Kulkarni.
\newblock Numerically {C}alabi-{Y}au orders on surfaces.
\newblock {\em J. London Math. Soc. (2)}, 72(3):571--584, 2005.

\bibitem[CK09]{del_mod}
Daniel Chan and Rajesh~S. Kulkarni.
\newblock Moduli of bundles on exotic del pezzo orders.
\newblock {\em preprint}, 2009.

\bibitem[DG94]{dol94}
Igor Dolgachev and Mark Gross.
\newblock Elliptic threefolds. {I}. {O}gg-{S}hafarevich theory.
\newblock {\em J. Algebraic Geom.}, 3(1):39--80, 1994.

\bibitem[Enr14]{enclass}
F.~Enriques.
\newblock Sulla classificazione délie superficie algebriche e particolarmente
  sulle superficie di genere lineare p(1) = 1.
\newblock {\em Rend. Accad. Lincei}, pages 206--214, 1914.

\bibitem[Fra08]{k3sym}
Kristina Frantzen.
\newblock K3 surfaces with special symmetry,.
\newblock {\em arXiv:0902.3761v1}, October 2008.

\bibitem[Gro95]{dixexp}
Alexander Grothendieck.
\newblock Le groupe de {B}rauer. {I}. {A}lg\`ebres d'{A}zumaya et
  interpr\'etations diverses [ {MR}0244269 (39 \#5586a)].
\newblock In {\em S\'eminaire {B}ourbaki, {V}ol.\ 9}, pages Exp.\ No.\ 290,
  199--219. Soc. Math. France, Paris, 1995.

\bibitem[Har97]{ag}
R.~Hartshorne.
\newblock {\em {Algebraic geometry}}.
\newblock Springer, 1997.

\bibitem[Kle05]{picard}
Steven~L. Kleiman.
\newblock The {P}icard scheme.
\newblock In {\em Fundamental algebraic geometry}, volume 123 of {\em Math.
  Surveys Monogr.}, pages 235--321. Amer. Math. Soc., Providence, RI, 2005.

\bibitem[Kod60]{fibres}
K.~Kodaira.
\newblock On compact complex analytic surfaces. {I}.
\newblock {\em Ann. of Math. (2)}, 71:111--152, 1960.

\bibitem[Lam98]{lamlec}
T-Y Lam.
\newblock {\em {Lectures on modules and rings}}.
\newblock Springer, 1998.

\bibitem[Lam01]{lam}
T-Y Lam.
\newblock {\em {A first course in noncommutative rings}}.
\newblock Springer, 2001.

\bibitem[LP81]{torelli}
Eduard Looijenga and Chris Peters.
\newblock Torelli theorems for {K}\"ahler {$K3$} surfaces.
\newblock {\em Compositio Math.}, 42(2):145--186, 1980/81.

\bibitem[Mil]{milne-abelian}
J.S. Milne.
\newblock {Abelian varieties}.
\newblock {\em Arithmetic, Geometry}, pages 103--150.

\bibitem[Mil80]{milnet}
James~S. Milne.
\newblock {\em \'{E}tale cohomology}, volume~33 of {\em Princeton Mathematical
  Series}.
\newblock Princeton University Press, Princeton, N.J., 1980.

\bibitem[Mor84]{mor84}
D.~R. Morrison.
\newblock On {$K3$} surfaces with large {P}icard number.
\newblock {\em Invent. Math.}, 75(1):105--121, 1984.

\bibitem[MP83]{morrison1983group}
I.~Morrison and U.~Persson.
\newblock {The group of sections on a rational elliptic surface}.
\newblock {\em Open Problems in Algebraic Geometry, Springer, LNM}, 997, 1983.

\bibitem[MP91]{k3tors}
Rick Miranda and Ulf Persson.
\newblock Mordell-{W}eil groups of extremal elliptic {$K3$} surfaces.
\newblock In {\em Problems in the theory of surfaces and their classification
  ({C}ortona, 1988)}, Sympos. Math., XXXII, pages 167--192. Academic Press,
  London, 1991.

\bibitem[Per85]{doub}
Ulf Persson.
\newblock Double sextics and singular {$K$}-{$3$} surfaces.
\newblock In {\em Algebraic geometry, {S}itges ({B}arcelona), 1983}, volume
  1124 of {\em Lecture Notes in Math.}, pages 262--328. Springer, Berlin, 1985.

\bibitem[Per90]{rat}
Ulf Persson.
\newblock Configurations of {K}odaira fibers on rational elliptic surfaces.
\newblock {\em Math. Z.}, 205(1):1--47, 1990.

\bibitem[SD74]{projk31}
B.~Saint-Donat.
\newblock Projective models of {$K-3$} surfaces.
\newblock {\em Amer. J. Math.}, 96:602--639, 1974.

\bibitem[Ser90]{serr}
Fernando Serrano.
\newblock Divisors of bielliptic surfaces and embeddings in {${\bf P}\sp 4$}.
\newblock {\em Math. Z.}, 203(3):527--533, 1990.

\bibitem[Sha67]{algsurf}
I.R. Shafarevich.
\newblock Algebraic surfaces.
\newblock 1967.

\bibitem[Sha96]{algeom2}
I.R. Shafarevich.
\newblock Algebraic geometry. {II}.
\newblock 35:vi+262, 1996.

\bibitem[Sil86]{silverman}
Joseph~H. Silverman.
\newblock {\em The arithmetic of elliptic curves}, volume 106 of {\em Graduate
  Texts in Mathematics}.
\newblock Springer-Verlag, New York, 1986.

\bibitem[Tan81]{tannen}
Allen Tannenbaum.
\newblock The {B}rauer group and unirationality: an example of
  {A}rtin-{M}umford.
\newblock In {\em The {B}rauer group ({S}em., {L}es {P}lans-sur-{B}ex, 1980)},
  volume 844 of {\em Lecture Notes in Math.}, pages 103--128. Springer, Berlin,
  1981.

\bibitem[VdB01]{ncp1}
Michel Van~den Bergh.
\newblock Non-commutative $\p1$-bundles over commutative schemes.
\newblock {\em math.RA/0102005}, 2001.

\end{thebibliography}
